\documentclass{svproc}
\usepackage{url}
\usepackage[symbol]{footmisc}

\usepackage{graphicx}
\usepackage{adjustbox}
\usepackage{enumitem}
\usepackage{latexsym}
\usepackage{algorithmic}
\usepackage{algorithm}
\usepackage{comment}
\usepackage{multirow}
\usepackage{tikz}
\usepackage{subcaption}
\usepackage{amssymb}
\usepackage{amsmath} 
\newcommand{\vect}[1]{\boldsymbol{\mathbf{#1}}}
\newcommand{\kernel}{G}
\newcommand{\JKV}{J}
\newcommand{\Uad}{\mathcal{U}}
\newcommand{\Aad}{\mathcal{A}}
\newcommand{\uk}{u_{k}}
\newcommand{\gk}{g_{k}}
\newcommand{\fk}{f_{k}} 
\newcommand{\ui}{u_{i}}  
\newcommand{\un}{u_{N}}
\newcommand{\ud}{u_{D}}  
\newcommand{\wdn}{w}  
\newcommand{\wdne}{w^{\varepsilon}}  
\newcommand{\dwdn}{\widetilde{w}}  
\newcommand{\unk}{u_{N,k}}
\newcommand{\udk}{u_{D,k}} 
\newcommand{\une}{u_{N}^{\varepsilon}}
\newcommand{\ude}{u_{D}^{\varepsilon}}
\newcommand{\uie}{\ui^{\varepsilon}}
\newcommand{\alpe}{\alpha^{\varepsilon}} 
\newcommand{\dun}{\widetilde{u}_{N}}
\newcommand{\dud}{\widetilde{u}_{D}}
\newcommand{\dui}{\widetilde{u}_{i}}
\newcommand{\sfTheta}{\Theta}

\newcommand{\VV}{\theta}
\newcommand{\Vn}{\theta_n}
\newcommand{\nn}{n}     
\newcommand{\dn}[1]{\partial_{\nn}{#1}} 
 
\newcommand{\dalp}[1]{\frac{\partial{#1}}{\partial \alpha}}

\newcommand{\intO}[1]{\int_{\Omega}{#1}{\, {dx}}}

\newcommand{\intS}[1]{\int_{\Sigma}{#1}{\, {ds}}}  
\newcommand{\intG}[1]{\int_{\Gamma}{#1}{\, {ds}}}  
\newcommand{\inS}[1]{\langle{#1}\rangle}  

\newcommand{\norm}[1]{\left\|{#1}\right\|}

\usepackage{parskip}

\begin{document}
\mainmatter              
\title{Simultaneous recovery of a corroded boundary and admittance using the Kohn-Vogelius method${}^{\ast}$}
\titlerunning{Simultaneous recovery of corroded boundaries and admittance using the Kohn-Vogelius method}  
%
\author{Moustapha Essahraoui\inst{1}
\and Elmehdi Cherrat\inst{1} \and Lekbir Afraites \inst{1}\and Julius Fergy Tiongson Rabago\inst{2}}
\authorrunning{Moustapha Essahraoui et al.} 
%
%
\institute{Mathematics and Interactions Teams (EMI)\\%
        Faculty of Sciences and Techniques\\%
        Sultan Moulay Slimane University\\%
        Beni Mellal, Morocco,\\
\email{essahraouimoustapha@gmail.com,\ cherrat.elmehdi@gmail.com,\\ \ l.afraites@usms.ma}\\ 
\and
Faculty of Mathematics and Physics\\%
	 Institute of Science and Engineering\\%
         Kanazawa University, Kanazawa 920-1192, Japan\\
         \email{jftrabago@gmail.com}
}    

\maketitle              
\begin{abstract}
We address the problem of identifying an unknown portion $\Gamma$ of the boundary of a $d$-dimensional ($d \in \{1, 2\}$) domain $\Omega$ and its associated Robin admittance coefficient, using two sets of boundary Cauchy data $(f, g)$--representing boundary temperature and heat flux--measured on the accessible portion $\Sigma$ of the boundary. Identifiability results \cite{Bacchelli2009,PaganiPierotti2009} indicate that a single measurement on $\Sigma$ is insufficient to uniquely determine both $\Gamma$ and $\alpha$, but two independent inputs yielding distinct solutions ensure the uniqueness of the pair $\Gamma$ and $\alpha$. 
In this paper, we propose a cost function based on the energy-gap of two auxiliary problems. 
We derive the variational derivatives of this objective functional with respect to both the Robin boundary $\Gamma$ and the admittance coefficient $\alpha$. 
These derivatives are utilized to develop a nonlinear gradient-based iterative scheme for the simultaneous numerical reconstruction of $\Gamma$ and $\alpha$. 
Numerical experiments are presented to demonstrate the effectiveness and practicality of the proposed method.
\footnotetext[1]{in \textit{Computational Methods for Inverse Problems and Applications}, ICMDS 2024, Khouribga, Morocco, October 21--22, 2024, Springer Proc. Math. Stat., vol. 498, Springer, Cham, 2025.}

\end{abstract}
%
\section{Introduction}
\label{sec:Introduction}
We consider a classical inverse geometry problem arising in non-destructive testing and evaluation, specifically in the detection of an internally corroded boundary.
Our objective is to determine a doubly connected domain $\Omega$ in $\mathbb{R}^{d}$, $d \in \{2,3\}$, externally bounded by the accessible boundary $\Sigma$ and internally by the unknown boundary $\Gamma$, along with the coefficient $\alpha$ that appears in the boundary condition of the PDE model. 
This is done using two pairs of Cauchy data, $\{(\fk, \gk)\}$, $k=1,2$, on $\Sigma$ for a harmonic function $u$ defined in $\Omega$.
On the unknown boundary $\Gamma$, the function $u$ is assumed to satisfy a homogeneous Robin boundary condition. 
Consequently, given $\Sigma$ and $\{(\fk, \gk)\}$, $k=1,2$, we examine the overdetermined boundary value problem:
\begin{equation}\label{eq:main_problem}
    - \Delta{\uk} = 0 \quad \text{in}\ \Omega, \quad {\uk} = {\fk}, \quad \frac{\partial {\uk}}{\partial n} = {\gk} \quad \text{on}\ \Sigma, \quad \frac{\partial {\uk}}{\partial n} + \alpha {\uk} = 0 \quad \text{on}\ \Gamma,
\end{equation}
where $\alpha \in L^{\infty}(\mathbb{R}^{d})$ such that $\alpha(x) \geqslant \alpha_{0} > 0$ for all $x \in \Gamma$, and $\alpha_{0}$ is a known constant. 
The vector $\nn$ denotes the outward unit normal to $\partial\Omega$.

In practical applications, the goal is to identify the inaccessible boundary using electrostatic measurements or thermal imaging techniques on the externally accessible part $\Sigma$. 
For instance, $u$ can be interpreted as the electrostatic potential in a specimen $\Omega$, where only the boundary portion $\Sigma$ is accessible for measurements. 
In other words, \eqref{eq:main_problem} can be interpreted as the determination of the shape of the inaccessible boundary portion $\Gamma$ based on two sets of imposed voltages ${\uk}|_{\Sigma}$ and the corresponding measured resulting currents $\dn{{\uk}}|_{\Sigma}$ on $\Sigma$.
Thus, in the inverse problem setting (see, e.g., \cite{FasinoInglese2007,KaupSantosaVogelius1996}), we are interested in solving the following:
\begin{problem} \label{prob:main_problem}
Determine the unknown boundary $\Gamma$, the Robin coefficient $\alpha$, and the electrostatic potentials $u_{k}$, for $k = 1, 2$, that satisfy the Cauchy data problem:
\[
    - \Delta{\uk} = 0 \quad \text{in}\ \Omega, \quad {\uk} = \fk, \quad \frac{\partial {\uk}}{\partial n} = \gk \quad \text{on}\ \Sigma, \quad \frac{\partial {\uk}}{\partial n} + \alpha {\uk} = 0 \quad \text{on}\ \Gamma.
\] 
\end{problem}
Extensive theoretical and numerical investigations have been conducted on the inverse problem of determining $\alpha$ from a single measurement of $u$ on a sub-boundary of $\Sigma$, assuming that $\Gamma$ is known. 
Notable works in this area include \cite{ChaabaneJaoua1999,ChaabaneElhechmiJaoua2004,FangLu2004,IngleseMariani2004,Jin2007,KaupSantosa1995,LinFang2005}, among others.
In contrast, numerical studies on recovering $\Gamma$ from measurements on part or all of $\Sigma$, assuming $\alpha$ is known, are discussed in \cite{AfraitesRabago2025,AfraitesRabago2024,FangLinMa2019,FangZeng2009,Hauptmannetal2019,RabagoAzegami2018}.
Additionally, the more general problem of simultaneous reconstruction of both $\Gamma$ and $\alpha$ has been addressed in \cite{Bacchelli2009,PaganiPierotti2009,Sincich2010}. 
Analytical and numerical recovery approaches, based on the equivalent boundary integral equation formulation of the problem, are presented in \cite{CakoniKress2007,CakoniKressSchuft2010a,CakoniKressSchuft2010b}.
Recently, a numerical shape optimization approach for the simultaneous recovery of $\Gamma$ and $\alpha$, using two Cauchy pairs and leveraging shape derivatives, was introduced in \cite{Fang2022}. 
Motivated by this study, we propose an energy-gap cost functional for the simultaneous recovery of the unknown inaccessible interior boundary $\Gamma$ and the Robin coefficient $\alpha$, as an alternative to the least-squares boundary tracking functional used in \cite{Fang2022}.

Let us briefly explain why two sets of Cauchy data are required in Problem~\ref{prob:main_problem}, which relates to the identifiability issue. 
Cakoni and Kress \cite{CakoniKress2007} showed that, for constant $\alpha$, a single Cauchy pair $(f, g)$ on $\Sigma$ may correspond to infinitely many domains $\Omega$. 
They provided counterexamples demonstrating that one pair is insufficient to uniquely determine both $\Gamma$ and $\alpha$. 
On the other hand, Bacchelli \cite{Bacchelli2009} proved that two linearly independent Cauchy pairs ensure the unique identification of $\Gamma$ and $\alpha$, provided one input is positive. 
Similarly, Pagani and Pierotti \cite{PaganiPierotti2009} established uniqueness using two measurements. 
On a different note, for stability results with two independent inputs, we refer readers to to Sincich \cite{Sincich2010}.
%
%

\textit{Problem setting.} Let us now present our problem in a mathematically precise manner.
We let $D \subset \mathbb{R}^{d}$, $d \in \{2, 3\}$, be a $C^{2,1}$ bounded open set and for a fixed real number $\delta > 0$, define the admissible set of unknown inclusions $\Uad$ as follows:
\[
\Uad := \left\{ \omega \Subset D \mid \omega \in C^{2,1}, \, d(x, \partial D) > \delta \ \forall{x} \in \omega, \text{ and } D \setminus \overline{\omega} \text{ is connected} \right\}.
\]
We denote $\Omega := D \setminus \overline{\omega}$, $\Sigma := \partial D$, and $\Gamma := \partial \omega$, and assume (only for technical purposes; see \cite{AfraitesRabago2025}) that $f \in H^{5/2}(\Sigma)$, $f \not\equiv 0$.
Moreover, we let $g \in H^{3/2}(\Sigma)$ be an admissible boundary measurement corresponding to $f$.
Furthermore, we restrict our investigation to the following admissible set of Robin parameters:
\[
\Aad := \left\{ \alpha \in L^{\infty}(\Gamma) \mid \text{$\alpha_0 \leqslant \alpha \leqslant \alpha_1$ for a.e. $x \in \Gamma$} \right\}.
\]

In this study, we tacitly assumed that we can find $(\omega^{\ast}, \alpha^{\ast}) \in \Uad \times \Aad$ such that Problem~\ref{prob:main_problem} has a solution. 
That is, we assume that the surface measurement $g$ (or $f$, if $g$ is given instead) is obtained without error.
Therefore, we precisely consider the following inverse geometry problem here:
\begin{problem}\label{prob:exact_problem}
	\text{Find $(\omega, \alpha) \in \Uad \times \Aad$ and $u$ such that \eqref{prob:main_problem} is satisfied.}
\end{problem}
To address Problem~\ref{prob:exact_problem} we define the functional $\JKV$ as
\begin{equation}\label{eq:KV_method}
	\JKV(\omega,\alpha):= \sum_{k=1}^{2}\frac12 \intO{|\nabla (\udk-\unk)|^{2}} + \frac{1}{2} \intG{ \alpha|\udk-\unk|^{2}},
\end{equation} 
and reformulate the overdetermined problem \eqref{eq:main_problem} as the minimization problem:
\begin{problem}\label{prob:min_problem}
	\text{Minimize $\JKV(\omega,\alpha)$ subject to \eqref{eq:state_un} and \eqref{eq:state_ud},}
where $\unk:=\unk(\omega,\alpha)$ and $\udk:=\udk(\omega,\alpha)$ respectively satisfy
\begin{equation}
\label{eq:state_un}
	- \Delta \unk  		= 0		\ \text{in $\Omega$},\qquad
	\dn{\unk}	 		= \gk		\ \text{on $\Sigma$},\qquad
	\dn{\unk} + \alpha \unk = 0	\ \text{on $\Gamma$},
\end{equation}
\begin{equation}
\label{eq:state_ud}
	- \Delta \udk  	= 0		\ \text{in $\Omega$},\qquad
	\udk	 		= \fk		\ \text{on $\Sigma$},\qquad
	\dn{\udk} + \alpha \udk	= 0	\ \text{on $\Gamma$}.
\end{equation}
\end{problem}
Problem~\ref{prob:min_problem} is equivalent to Problem~\ref{prob:main_problem} when, for all $k=1,2$, $\uk = \fk$ and $\dn{\uk} = \gk$ on $\Sigma$ (cf. \cite[Remark 2.6]{AfraitesRabago2025}).

The functional $J$ is known as the Kohn-Vogelius functional \cite{KohnVogelius1987}.
It has been applied in various contexts, including shape detection in electrical impedance tomography \cite{EpplerHarbrecht2005}, transmission problems \cite{AfraitesDambrineKateb2007,AfraitesDambrineKateb2008}, inverse Robin boundary problem \cite{AfraitesRabago2025}, inverse geometry problems in Stokes flow \cite{CaubetDambrineKatebTimimoun2013}, and geometric inverse source problems \cite{AfraitesMasnaouiNachaoui2022}. 
See also \cite{EpplerHarbrecht2012a,BacaniPeichl2013,BenAbdaetal2013} for the application of the method to solving the exterior Bernoulli free boundary problem. 
Notably, this volume-based cost functional often yields more accurate solutions than boundary-integral-based alternatives.

\textit{Paper organization.} The remainder of the paper is organized as follows.
Section~\ref{sec:sensitivity_analysis} discusses the sensitivity of $J$ with respect to the unknown inclusion $\omega$ and the unknown Robin coefficient $\alpha$.
Section~\ref{sec:numerical_experiments} presents a numerical algorithm based on the shape derivative of $J$ with respect to $\omega$ and the Fr\'{e}chet derivative of $J$ with respect to the parameter $\alpha$.
The proposed numerical method is then tested through various numerical examples.

\section{Sensitivity of $J$ respect to $\omega$ and $\alpha$}
\label{sec:sensitivity_analysis} 
In this section, we present the derivative of $\JKV(\omega, \alpha)$ with respect to $\omega$, followed by its derivative with respect to $\alpha$. 
For notational convenience, when one argument is fixed, we omit the functional's dependence on that argument.
\subsection{Shape derivative of $J$ with respect to $\omega$}
\label{sec:shape_derivative}
Let us consider Problem~\ref{prob:min_problem} and suppose that $\alpha \in \Aad$ is fixed.
For ${\varphi},{\psi}\in H^{1}(\Omega)$, we define 
\[
	a(\alpha;{\varphi},{\psi}) = \intO{\nabla \varphi \cdot \nabla \psi} + \intG{ \alpha \varphi \psi}.
\]
The variational formulations of \eqref{eq:state_un} are \eqref{eq:state_ud} are as follows:
\begin{align}
	&\text{Find ${\ud} \in H^{1}(\Omega)$, $\ud|_{\Sigma} = f$, such that $a(\alpha;\ud,\psi) = 0$, $\forall\psi \in {V}(\Omega)$};\label{eq:weak_ud}\\
	&\text{Find ${\un} \in H^{1}(\Omega)$ such that  $a(\alpha;\un,\psi) = \inS{g, \psi}_{\Sigma}$, $\forall\psi \in H^{1}(\Omega)$\label{eq:weak_un}}.
\end{align}
The weak forms \eqref{eq:weak_ud} and \eqref{eq:weak_un} are known to be well-posed if $(f,g) \in H^{1/2}(\Sigma) \times H^{-1/2}(\Sigma)$ and $\Omega$ is Lipschitz regular, as established by the Lax-Milgram lemma (see, e.g., \cite[p. 46]{AfraitesRabago2025}).
Meanwhile, the existence of a minimizer for $\JKV$ can be shown under an appropriate topology, provided that the admissible domain satisfies a \textit{uniform cone property} \cite[Thm. 2.4.7, p. 56]{HenrotPierre2018}. 
While we do not prove this claim here, the arguments used in \cite{RabagoAzegami2019b,AfraitesRabago2024} can be followed to verify it.

To numerically solve Problem~\ref{prob:min_problem}, we will use a shape-gradient-based technique combined with the finite element method (FEM). 
To implement this, we require the expression for the shape derivative of $\JKV$, which has already been derived in \cite[Prop. 3.4]{AfraitesRabago2025}. 
Therefore, we refer the reader to \cite{AfraitesRabago2025} for a detailed characterization of the shape gradient and the additional assumption on the Robin function $\alpha$; see Assumption (A) in \cite[p.~270]{AfraitesRabago2025}.
\begin{proposition}[Shape gradient of $J$]
	\label{prop:shape_gradient}
	Let $D_{\delta}$ be an open set with a ${C}^{\infty}$ boundary, such that $\{ x \in D \mid \text{$d(x,\partial D) > \delta/2$}\} \subset D_{{\delta}} \subset \{ x\in D \mid \text{$d(x,\partial D) > \delta/3$}\}$. 
	Define $\sfTheta$ as the collection of ${C}^{2,1}(\mathbb{R}^{d})^{d}$ vector fields with compact support in $\overline{D}_{\delta}$. 
	Assume that $\Omega \in {C}^{2,1}$ is an admissible domain, i.e., $\Omega = D\setminus\overline{\omega}$ where $\omega \in \Uad$ and ${\VV} \in \sfTheta$.
	Then, the map $t \mapsto \JKV(\omega_{t})$, is ${C}^1$ in a neighborhood of $0$, and its shape derivative at $0$ is given by
	\[
	d\JKV(\omega)[\VV] = \intG{\left[- F(\un,w) + |\nabla w|^{2} + \frac{\alpha}{2} \left( 2 w \dn{w} + \kappa w^{2} \right) \right] \Vn}
		=: \inS{\kernel,\Vn}_{\Gamma},
	\]
	where $w = \ud - \un$, and $F(v,q) = -\nabla_{\Gamma} v \cdot \nabla_{\Gamma} q - \alpha \left( \dn{v} + \kappa v \right)q -\dn \alpha v q$, for $v, q \in H^{2}(\Omega)$.
	Here, $\nabla_{\Gamma}$ denotes the tangential gradient operator while $\kappa$ represents the mean curvature of the free boundary $\Gamma = \partial{\omega}$.
\end{proposition}
An important feature of the Kohn-Vogelius functional $\JKV$ is that its gradient does not depend on the derivatives of the states. 
This is a characteristic commonly found in energy-type cost functionals, allowing the gradient to be computed numerically without the need for adjoint variables. 
We highlight that this approach differs from the one presented in \cite{Fang2022}, which requires the use of adjoint variables.
From a numerical standpoint, our method only needs for a regularization functional for $\alpha$, in contrast to \cite{Fang2022}, as the proposed objective functional inherently includes a smoothing effect. 
This enhances stability and provides a more accurate approximation of the minimizer, offering a solution to the inverse problem of simultaneously recovering $\alpha$ and $\Gamma$.
\subsection{Derivative of $J$ with respect to $\alpha$}
\label{functional_KV_respect_alpha}
Let $\Omega$ be a given admissible domain.
For simplicity, we denote the derivative of $\ui$ with respect to $\alpha$ by $\dui = \dalp{\ui}$.
Hereinafter, without further notice, we assume $i \in \{D, N\}$, indicating the solutions to equations \eqref{eq:state_ud} and \eqref{eq:state_un}.
Given $\alpha, \rho \in \Aad$, there exists a real number $\varepsilon=\varepsilon(\alpha,\rho) > 0$ such that $\alpe=\alpha+\varepsilon \rho \in \Aad$.
Consequently, $\alpe$ is a perturbation of $\alpha$ in the direction $\rho$.
\begin{proposition}[{\cite{ChaabaneJaoua1999}}] \label{eq:Frechet_derivative}
	Let $\Omega$ be a given admissible domain and $\alpha, \rho \in \Aad$.
	There exist $\dui, {O}(\varepsilon^{2}) \in H^{1}(\Omega)$  such that
\begin{equation}\label{eq:u_perturbed}
    \uie=\ui+\varepsilon {\dui}+{O}(\varepsilon^{2})
    \quad\text{and}\quad 
    \lim_{\varepsilon \searrow 0}{O}(\varepsilon^{2})=0,
\end{equation} 
where $\dui$ solves the variational equation
\begin{equation}\label{eq:perturbed_variational}
	a(\dui, \psi) = -\intG{\rho {u_i} \psi},
\end{equation}
for all $\psi \in H^{1}(\Omega)$ when $i=N$ and for all $\psi \in V(\Omega)$ when $i=D$.
Additionally, $J$ is Gateaux differentiable at $\alpha$ in the $\rho$ direction, and its derivative is given by
\[
	d\JKV(\alpha)\rho =\frac{1}{2}\intG{ \rho (\ud^{2}-\un^{2})}.
\]
\end{proposition}
\begin{proof}
Let $\Omega$ be a given admissible domain and $\alpha, \rho \in \Aad$.
The differentiability of the states and the cost functional with respect to $\alpha$ can be shown by standard arguments, so we omit it.
Let us first establish the structure of the derivatives of the states with respect to $\alpha$.
To do this, we let $\varepsilon=\varepsilon(\alpha,\rho) > 0$ be such that $\alpe=\alpha+\varepsilon \rho \in \Aad$ and consider the following perturbed state problems:
\begin{align}
    a(\alpe;\une, \psi) &= \intS{g\psi}, \quad \forall \psi \in H^{1}(\Omega),\label{eq:perturb_state_un}\\
    a(\alpe;\ude, \psi) &= 0, \quad \forall \psi \in {V}(\Omega).\label{eq:perturb_state_ud}
\end{align} 
Note that $g$ represents an observation on the boundary $\Sigma$. 
Since the perturbation of $\alpha$ does not affect $g$, we have $g^{\varepsilon} = g$ on $\Sigma$.

Subtracting \eqref{eq:state_un} from \eqref{eq:perturb_state_un}, we obtain 
\begin{equation}\label{eq:difference_un}
	a(\alpha;\une-\un, \psi) = -\intG{\varepsilon \rho \un \psi}, \qquad \forall \psi \in H^{1}(\Omega).
\end{equation}
For any $\varepsilon > 0$, equation \eqref{eq:difference_un} has a unique weak solution $\une - \un \in H^{1}(\Omega)$.  
By substituting $\psi = \une - \un$ into the above equation, we deduce that $\norm{\une - \un}_{H^1(\Omega)} = {O}(\varepsilon)$, which suffices for computing $d\JKV(\alpha)\rho$.  
Additionally, the strong convergence of $\left\{\frac{1}{\varepsilon}(\une - \un)\right\}$ in $H^1(\Omega)$ can be established, though we omit the proof here.  

To proceed, we divide equation \eqref{eq:difference_un} by $\varepsilon$ and take the limit as $\varepsilon \to 0$. Utilizing \eqref{eq:u_perturbed}, we deduce that $\dun \in H^{1}(\Omega)$ satisfies the weak formulation
\[
	a(\alpha; \dun, \psi) = -\intG{\rho \un \psi}, \quad \forall \psi \in H^1(\Omega).
\]  
Similarly, by applying the same approach, we find that $\dud \in V(\Omega)$ satisfies 
\[
	a(\alpha; \dud, \psi) = -\intG{\rho \ud \psi}, \quad \forall \psi \in V(\Omega).
\]

Let us denote $\wdn = \ud-\un$ and $\wdne = \ude-\une$.
Now, the derivative of $\JKV(\alpha)$ with respect to $\alpha$ in the direction $\rho$ is givem by
\begin{align*}
	d\JKV(\alpha)\rho
	=\lim_{\varepsilon \searrow 0} \frac{\JKV(\alpha + \varepsilon \rho)-\JKV(\alpha)}{\varepsilon}
	=\lim_{\varepsilon \searrow 0}\frac{1}{2\varepsilon} \left[ a(\alpe; \wdne, \wdne) - a(\alpha;\wdn,\wdn) \right].
\end{align*}
Observe that the expression on the right contains the difference ${(\uie-\ui)}/{\varepsilon} = {\dui}+{O}(\varepsilon)$.
Hence, evaluating the limit leads to the expression
\begin{align*}
	d\JKV(\alpha)\rho 
	&=  a(\alpha; \dwdn, \wdn) + \frac12\intG{ \rho |\wdn|^{2} }\\
	&=  a(\alpha; \dud, \ud) - a(\alpha; \dud, \un)  - a(\alpha; \dun, \wdn) + \frac12\intG{ \rho |\wdn|^{2} }.
\end{align*}
We eliminate $\dwdn \in H^{1}(\Omega)$ appearing on the above expression.
To do this, first, we consider \eqref{eq:perturbed_variational} for $i=N$ and choose $\psi = \wdn = \ud - \un \in H^{1}(\Omega)$ to obtain
\[
	-a(\alpha; \dun, \wdn) = \intG{ \rho \un (\ud - \un) }.
\]
Next, we consider \eqref{eq:weak_ud} and take $\psi = \dud \in V(\Omega) \subset H^{1}(\Omega)$ to get
\[
	a(\alpha; \ud, \dud) = 0.
\]
Finally, we consider \eqref{eq:weak_un} and choose $\psi = \dud \in V(\Omega) \subset H^{1}(\Omega)$ (i.e., $\dud|_{\Sigma} = 0$) to obtain
\[
	a(\alpha; \dud, \un) = a(\alpha; \un, \dud) =  \inS{g, \dud}_{\Sigma} = 0.
\]
Thus, we arrive at the following simplication (i.e., eliminating $\dwdn$)
\begin{align*}
	d\JKV(\alpha)\rho  
	&= \intG{ ( \rho \un \ud - \rho \un^{2} ) } + \frac12\intG{ \rho (\ud^{2} - 2\ud\un + \un^{2})}\\ 
	&= \frac12\intG{ \rho (\ud^{2} - \un^{2})}.
\end{align*}
That is, $d\JKV(\alpha)=\frac12{(\ud^{2} - \un^{2})}$. This proved the proposition.
\end{proof}
\section{Numerical algorithm and examples} 
\label{sec:numerical_experiments}
\subsection{Numerical algorithm}
\label{subsec:Numerical_Algorithm}
Our numerical approach is built upon the gradient informations derived in the previous section, incorporating techniques from earlier works (see, e.g., \cite{AfraitesRabago2024,AfraitesRabago2025}) that implement a Sobolev-gradient-based algorithm in a finite element setting. 
Below, we provide the key details of our numerical method.

\textit{Choice of domain deformation vector.} 
The choice $\VV = V_{n}\nn = -\kernel\nn$, where $\kernel \in L^{2}(\Gamma)$ and $\kernel \not\equiv 0$, provides a descent direction for the cost function $\JKV$.
However, if the data $(f, g)$ lack sufficient smoothness, the shape gradient's surface expression given in Proposition~\ref{prop:shape_gradient} may not exist, and the descent vector $\VV$ could have poor regularity.
Specifically, the $L^{2}(\Gamma)$ regularity of the shape gradient $\kernel$ is insufficient to achieve a stable approximation of the unknown boundary. 
To address this issue, we employ the Riesz representation of the shape gradient, a well-established technique in the literature (see, e.g., \cite{Doganetal2007,Azegami1994}).
To ensure a smooth descent direction for $\JKV$, we determine a vector $\VV \in H_{\Sigma,0}^1(\Omega)^{d}$ that solves the variational equation:
\[
	\intO{( \nabla \VV : \nabla {\varphi} + \VV \cdot {\varphi} ) } 
		= - \intG{\kernel\nn \cdot {\varphi}}, \quad \forall{\varphi} \in H_{\Sigma,0}^1(\Omega)^{d}.
\]
This yields a \textit{Sobolev gradient} \cite{Neuberger1997,Azegami2020} representation $\VV$ of $-\kernel\nn$ over $\Omega$.

%
\textit{Step-size computation and stopping condition.}
The $k$th step size $t^{k}$ is computed at each iteration using a backtracking line search with the formula  
\[
t^{k} = \mu \frac{\JKV(\omega^{k},\alpha^{k})}{\norm{\VV^{k}}^{2}_{{H}^{1}(D\setminus\overline{\omega}^{k})^{d}}},
\]  
where $\mu > 0$ is a scaling factor; see \cite[p. 281]{RabagoAzegami2020}. 
The value of $\mu$ is adjusted to prevent the formation of inverted triangles in the mesh after the update.   
Meanwhile, $\varepsilon^{k}$ is determined using the standard Armijo line search procedure.
The algorithm terminates after a finite number of iterations.  
For precise measurements, we set the maximum number of iterations to $500$, and for noisy data, we reduced it to $200$.

\textit{Choice of regularization functional.}
A common regularization approach involves adding specific terms to the minimization process, particularly for addressing ill-posed systems \cite{Rundell2008,Fang2022}.
In our experiments, we employ a single regularization term based on the PDE formulation, alongside the previously mentioned extension-regularization technique. 
This approach is independent of the parametrization of $(\Gamma, \alpha)$ or the choice of numerical solver, simplifying the process by reducing the need for tuning multiple parameters.
One possible regularization functional is given by:
\begin{equation}\label{eq:regularization_functional}
	\frac{\tau_{1}}{2} P(\Gamma) + \frac{\tau_{2}}{2} R(\alpha) 
	= \frac{\tau_{1}}{2} \int_{\Gamma} 1 \, ds + \frac{\tau_{2}}{2} \int_{\Gamma} {\alpha}^{2} \, ds,
\end{equation}
where $\tau_{1}, \tau_{2} > 0$ are constants. 
However, we omit the perimeter term and rely solely on Tikhonov regularization $R$ for $\alpha$. 
As demonstrated in our examples, this choice is sufficient for achieving reasonable reconstruction, even in the presence of noisy data.

\textit{Choice of regularization parameter.}
When reconstructing noisy data, choosing the regularization parameter $\tau = \tau_{2}$ in \eqref{eq:regularization_functional} is crucial. 
This parameter is typically selected using the discrepancy principle, which requires accurate knowledge of the noise level. However, in many cases, the noise level is either unknown or unreliable, and an incorrect estimate can significantly degrade reconstruction accuracy.
To overcome this, heuristic rules for parameter selection are needed, especially when noise level information is unavailable or unreliable. 
We propose a heuristic rule based on the balancing principle \cite{ClasonJinKunisch2010b}, which selects $\tau > 0$ such that
\begin{equation}\label{eq:balancing_principle}
    (\beta - 1) \JKV(\omega, \alpha) - \frac{\tau}{2} {R}(\mu) 
    := (\beta - 1) \frac{1}{2} a(\alpha;w,w) - \frac{\tau}{2} \norm{\alpha}^{2}_{L^{2}(\Gamma)^{2}} = 0,
\end{equation}
for a constant $\beta > 1$. 
This balances the data-fitting term $\JKV(\omega, \alpha)$ with the penalty term ${R}(\alpha)$, where $\beta$ controls the trade-off. 
This rule does not rely on noise level information and has been successfully applied to both linear and nonlinear inverse problems \cite{ClasonJinKunisch2010,ClasonJinKunisch2010b,Clason2012,ClasonJin2012,ItoJinTakeuchi2011}, and recently in \cite{Meftahi2021,MachidaNotsuRabago2024}.

\textit{Algorithm.} 
The $k$th pair approximation $(\omega^{k}, \alpha^{k})$ of the exact solution pair $(\omega^{\ast},\alpha^{\ast})$ is computed as follows:
\begin{algorithm}
\caption{Simultaneous recovery of $\Gamma$ and $\alpha$}
\begin{algorithmic}[1]
\STATE \textbf{Initialization:} Fix $D \subset \mathbb{R}^{d}$, $\tau, \beta > 0$, and choose an initial shape $\omega^{0}$ and Robin coefficient $\alpha^{0}$.
\FOR{$k = 0, 1, 2, \ldots$}
    \STATE Solve the two-state problems on the $D\setminus\overline{\omega}^{k}$ with Robin coefficient $\alpha^{k}$.
    \STATE Choose $t^{k} > 0$ and compute the vector $\VV^{k}$ in $D\setminus\overline{\omega}^{k}$.
    \STATE Update the inclusion: $\omega^{k+1} = (\operatorname{id} + t^{k} \VV^{k}) \omega^{k}$.
    \STATE Compute $\tau$ via the balancing principle \eqref{eq:balancing_principle}.
    \STATE Choose $\varepsilon^{k} > 0$ and update the Robin coefficient: $ \alpha^{k+1} = \alpha^{k} - \varepsilon^{k} d \JKV(\alpha) + \varepsilon^{k} \tau \alpha^{k}$.
\ENDFOR 
\end{algorithmic}
\end{algorithm}
\subsection{Numerical examples}
\label{subsec:Numerical_Examples_{2}D} 
For the numerical examples, we make the following broad assumptions:  
The specimen under examination has a circular shape with a unit radius, centered at the origin; i.e., $D = B(0,1)$. 
For $(x_{1}, x_{2}) \in \Gamma^{\ast}$, we consider three cases for the exact Robin coefficient:
\begin{description}
	\item[Case A1:] $\alpha_{1}^{\ast} = \operatorname{exp}{(x_{1} x_{2})}$;
	\item[Case A2:] $\alpha_{2}^{\ast} = 1 + 0.5 x_{1} x_{2}$;
	\item[Case A3:] $\alpha_{3}^{\ast} = 1 + 0.5 \sin(\pi{x_{1}}) \sin(\pi{x_{2}})$.
\end{description}
The data is generated synthetically. 
Specifically, we impose the Dirichlet boundary conditions $f_{1}(t) = \cos(t)$ and $f_{2}(t) = \sin(t)$, $t \in [0, 2\pi]$, and set $g_{1} = \dn{u_{1}^{\ast}} |_{\Sigma}$ and $g_{2} = \dn{u_{1}^{\ast}} |_{\Sigma}$, respectively.

To prevent ``inverse crimes" (see \cite[p. 154]{ColtonKress2013}), we use different numerical methods for generating synthetic data and performing the inversion. 
For the forward problem, the domain is discretized using 300 nodal points on both the exterior and interior boundaries, with ${P}_{2}$ finite element basis functions implemented in {\sc FreeFem++} \cite{Hecht2012}. 
In contrast, the inversion process uses a uniform mesh with $150$ nodes on each boundary, a mesh size of $h = 0.05$, and ${P}_{1}$ finite elements to solve all variational problems.

We will evaluate the proposed identification procedure by considering the following geometries for the unknown boundary $\Gamma$:
\begin{description}
	\item[Case B1:] $\Gamma^{\ast}_{1} = \left\{\begin{pmatrix} 0.1 + 0.7\cos{t}\\ 0.2 + 0.5\sin{2t}\end{pmatrix}, \forall t \in [0, 2\pi) \right\}$;
	\item[Case B2:] $\Gamma^{\ast}_{2} = \left\{\begin{pmatrix} 0.6 \cos{t} \\ 0.5 \sin{t} (1.8 + \cos{2t}) \end{pmatrix}, \forall t \in [0, 2\pi) \right\}$;
	\item[Case B3:] $\Gamma^{\ast}_{3} 
		= \left\{\begin{pmatrix} -0.25 + \displaystyle\frac{0.6+0.54\cos{t}+0.06\sin{2t}}{1+0.75\cos{t}}\cos{t} \\[0.75em] 0.05 + \displaystyle\frac{0.6+0.54\cos{t}+0.06\sin{2t}}{1+0.75\cos{t}}\sin{t}\end{pmatrix}, \forall t \in [0, 2\pi) \right\}$.
\end{description}
In all cases, the algorithm starts with $\alpha^{0} = 1$ and an initial guess of $\Gamma^{0} = C(\vect{0}, r_{0})$, where $r_{0} \in \{0.5, 0.55, 0.6\}$, depending on the case. 
It is terminated after $500$ iterations in all cases. 
Although the algorithm could run longer and the stopping criterion could be improved, this simple approach already provides satisfactory results based on our experience.

We also test the algorithm on noisy data by perturbing the exact measurements with Gaussian noise of mean zero and standard deviation $0.5$, where the noise level is controlled by the parameter $\delta$ (noise weight).
Due to the severely ill-posed nature of the problem, we limit our analysis to a noise level of $\delta = 0.005$.
\subsection{Numerical results and discussion}
In this subsection, we discuss the results of our numerical experiments. 
Table~\ref{tab:shapes} shows the reconstructed shapes for Cases A1--A3 combined with Cases B1--B3. 
Each row corresponds to a specific case, comparing the shapes obtained from exact and noisy data against the true inclusion. 
The outermost thick black line indicates the object's surface, while the inner black line shows the exact geometry of the unknown boundary. 
The black dashed lines represent the initial guesses. 
The blue line with ``$\times$'' markers and the red dashed line with ``$\bigcirc$'' markers depict the reconstructions under exact and noisy ($\delta = 0.005$) measurements, respectively.
We observe that the method achieves reasonable reconstructions for test problems, especially without noise; see Tables~\ref{tab:robin_30} and \ref{tab:robin_135}. 
Even for complex shapes with concavities, the unknown boundary is reconstructed reliably despite $0.5\%$ noise. 
However, the Robin coefficient reconstruction is poor due to the ill-posedness of the problem; see Tables~\ref{tab:noisy_robin_30} and \ref{tab:noisy_robin_135}. 
To address this issue, multiple measurements (more than two) are needed to obtain reasonable reconstruction of both the shape and the Robin coefficient; see, e.g., \cite{Harris2023}.
Figure~\ref{fig:summary_A3B3} shows the histories of cost values, gradient norms, and balancing parameter $\tau$. 
We see that the cost values are lower (as expected) under exact measurements and stabilize after several iterations for both exact and noisy data. 
Gradient norms also decrease significantly despite some oscillations.

Overall, the method is effective under exact measurements, even for inclusions with concavities. 
Additional tests with varied initial geometries and inclusions (not shown) confirm consistent performance with exact data. 
For noisy data, using multiple measurements is essential for reasonable reconstructions. 
Improving noisy reconstructions will be the focus of future work.
\begin{table}[htp!]
\centering
\begin{adjustbox}{max width=\textwidth}
\begin{tabular}{|c|c|c|c|}
\hline
\textbf{Case} & \textbf{A1} & \textbf{A2} & \textbf{A3} \\ \hline
\textbf{B1} & \resizebox{0.325\linewidth}{!}{\includegraphics{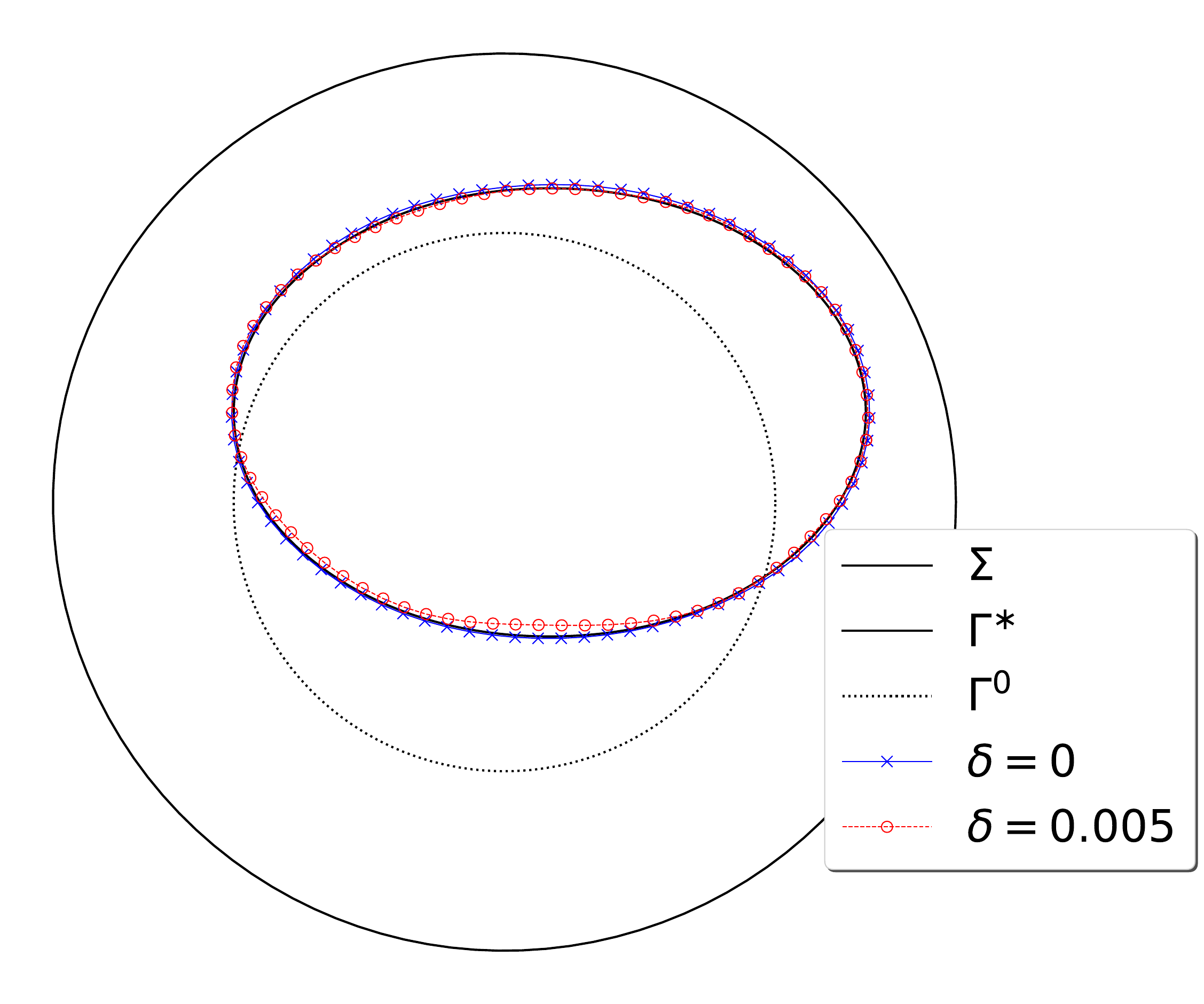}} &
	\resizebox{0.325\linewidth}{!}{\includegraphics{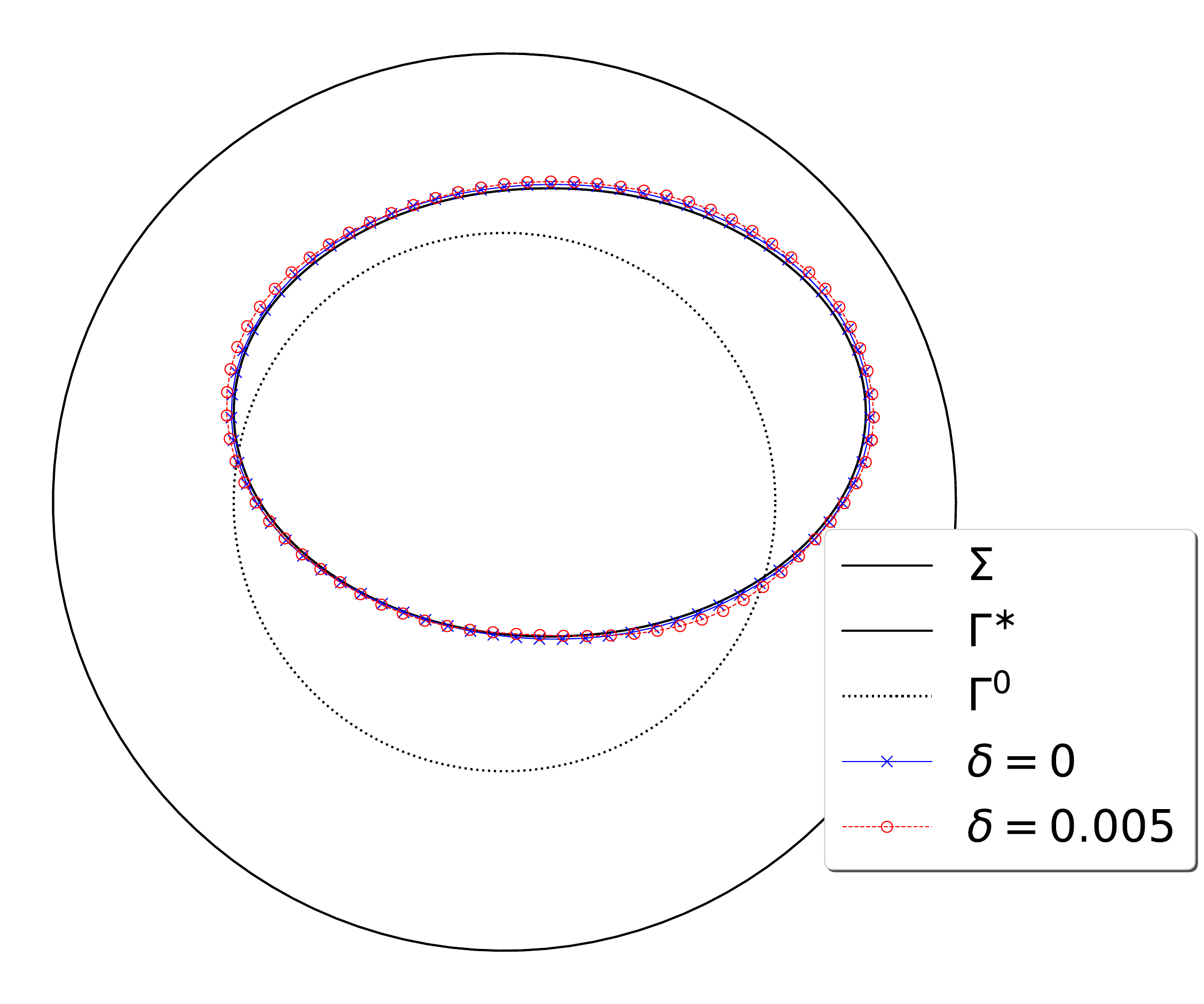}} &
	\resizebox{0.325\linewidth}{!}{\includegraphics{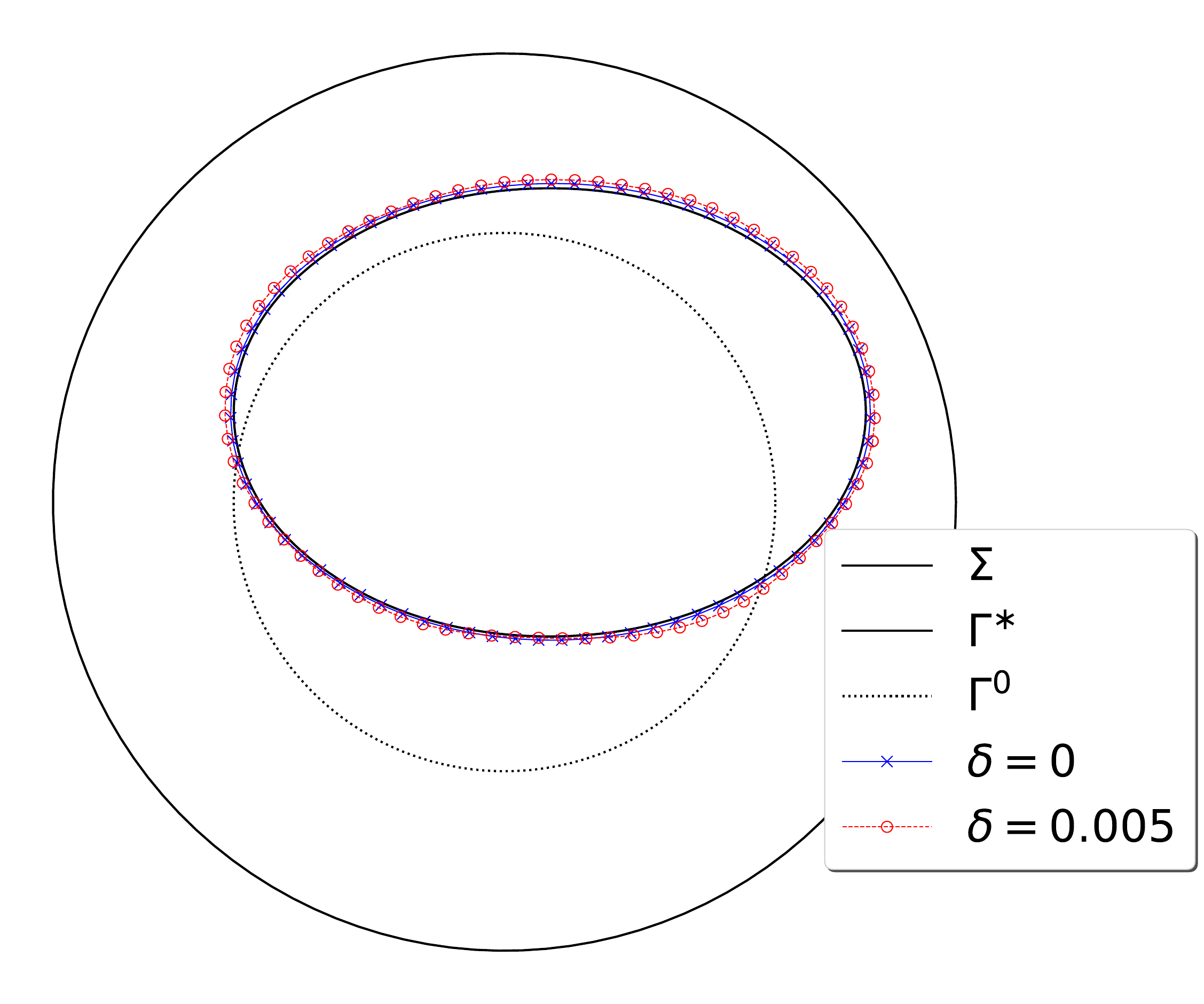}} \\ \hline
\textbf{B2} & \resizebox{0.325\linewidth}{!}{\includegraphics{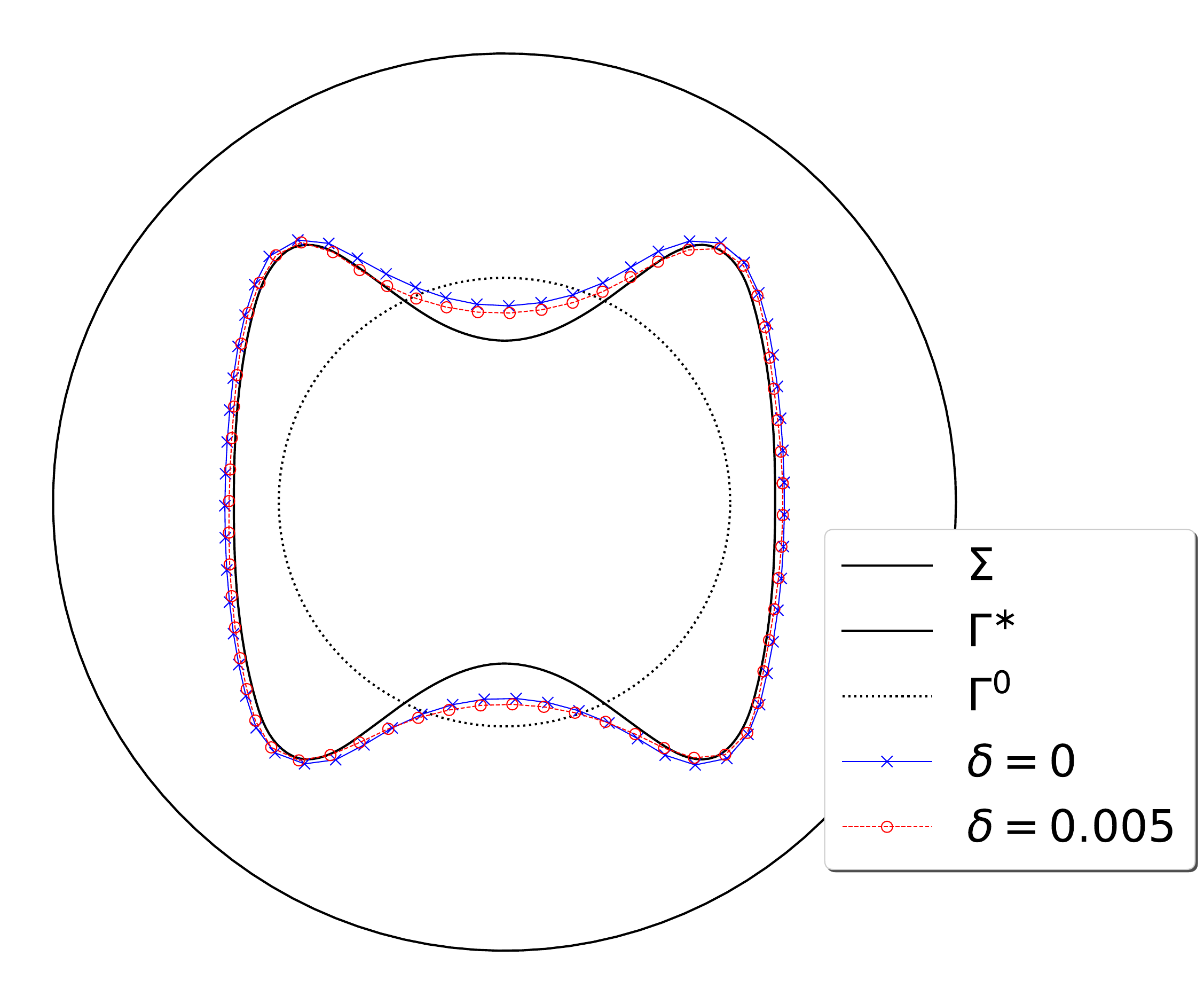}} &
	\resizebox{0.325\linewidth}{!}{\includegraphics{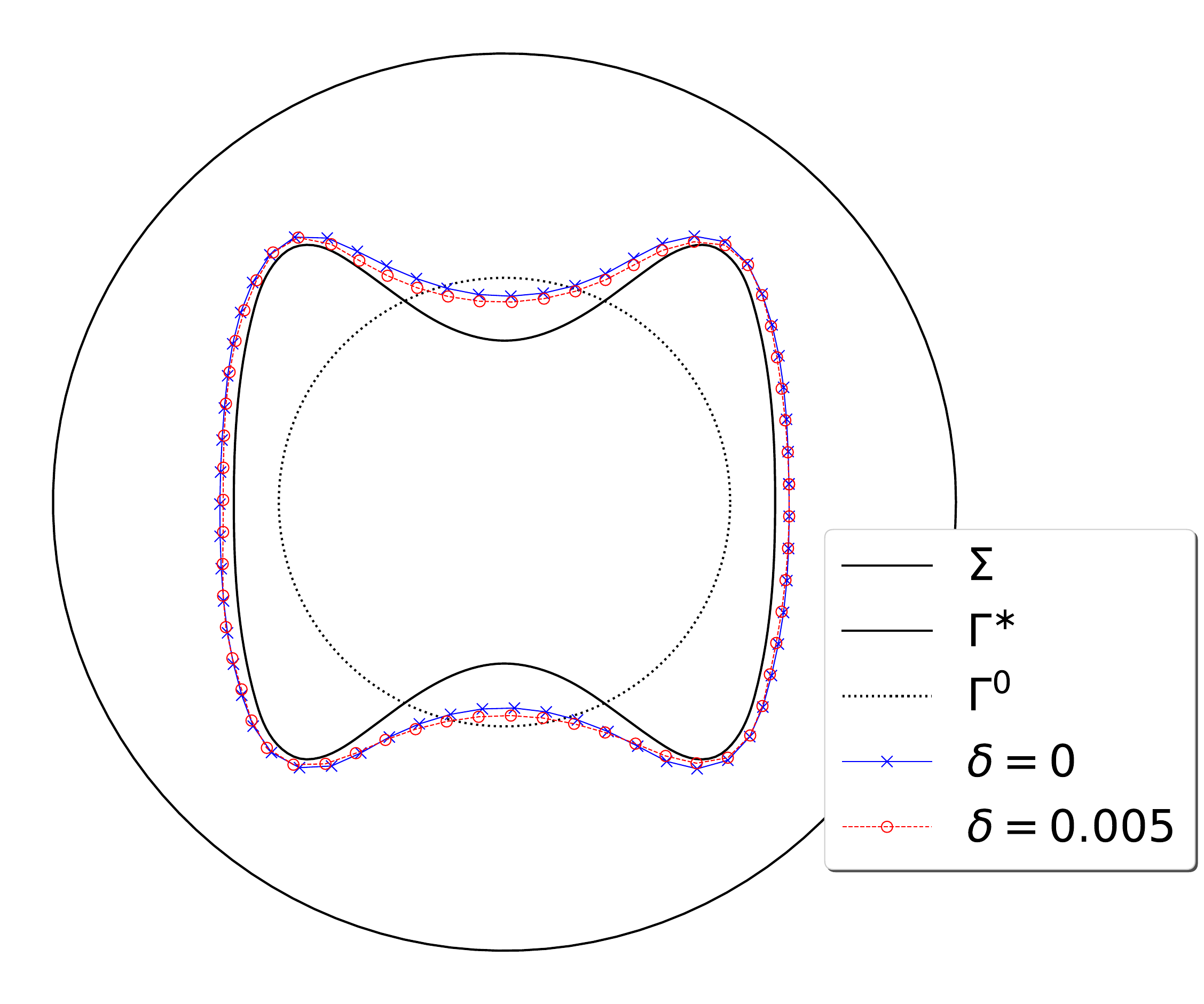}} &
	\resizebox{0.325\linewidth}{!}{\includegraphics{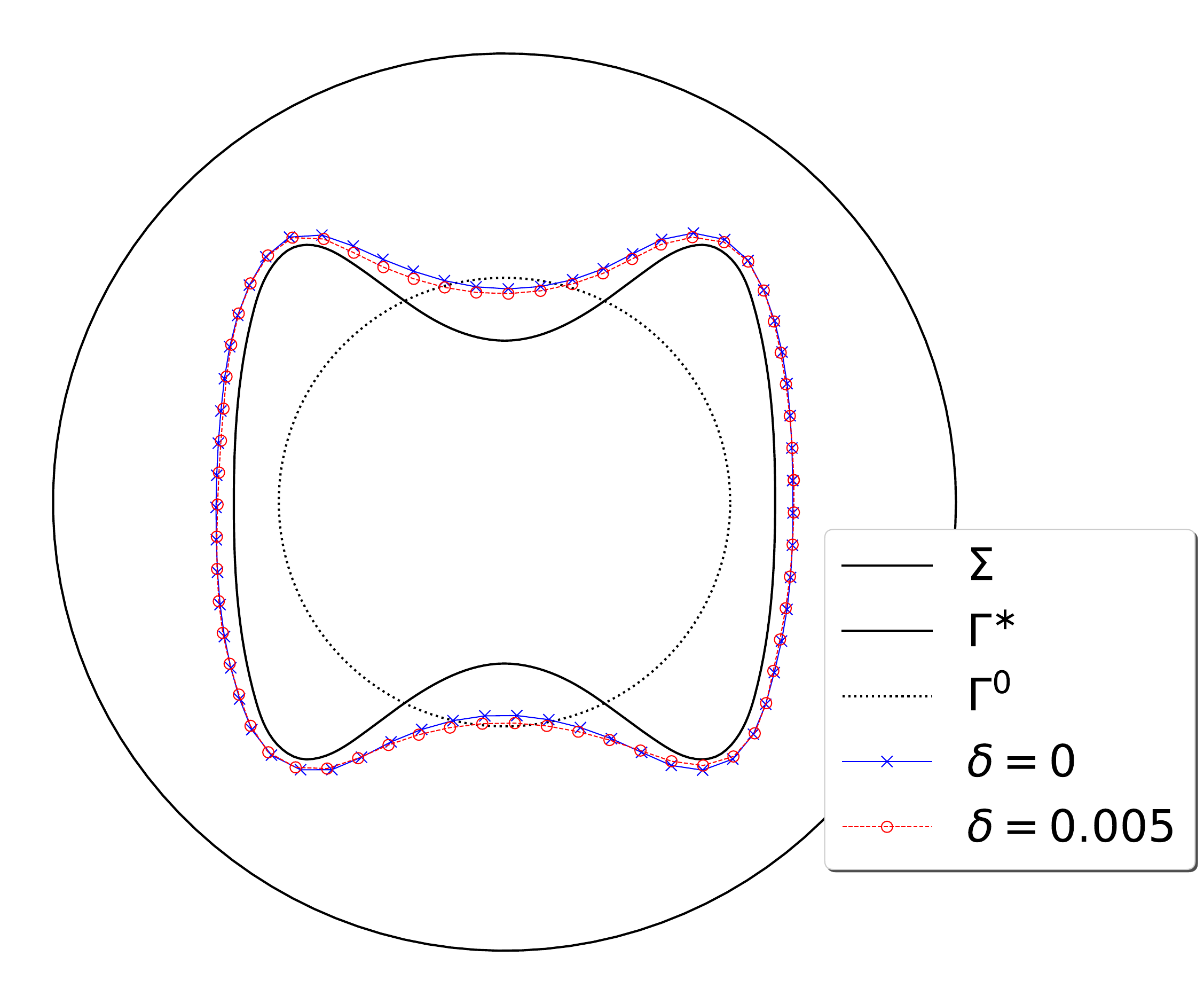}} \\ \hline
\textbf{B3} & \resizebox{0.325\linewidth}{!}{\includegraphics{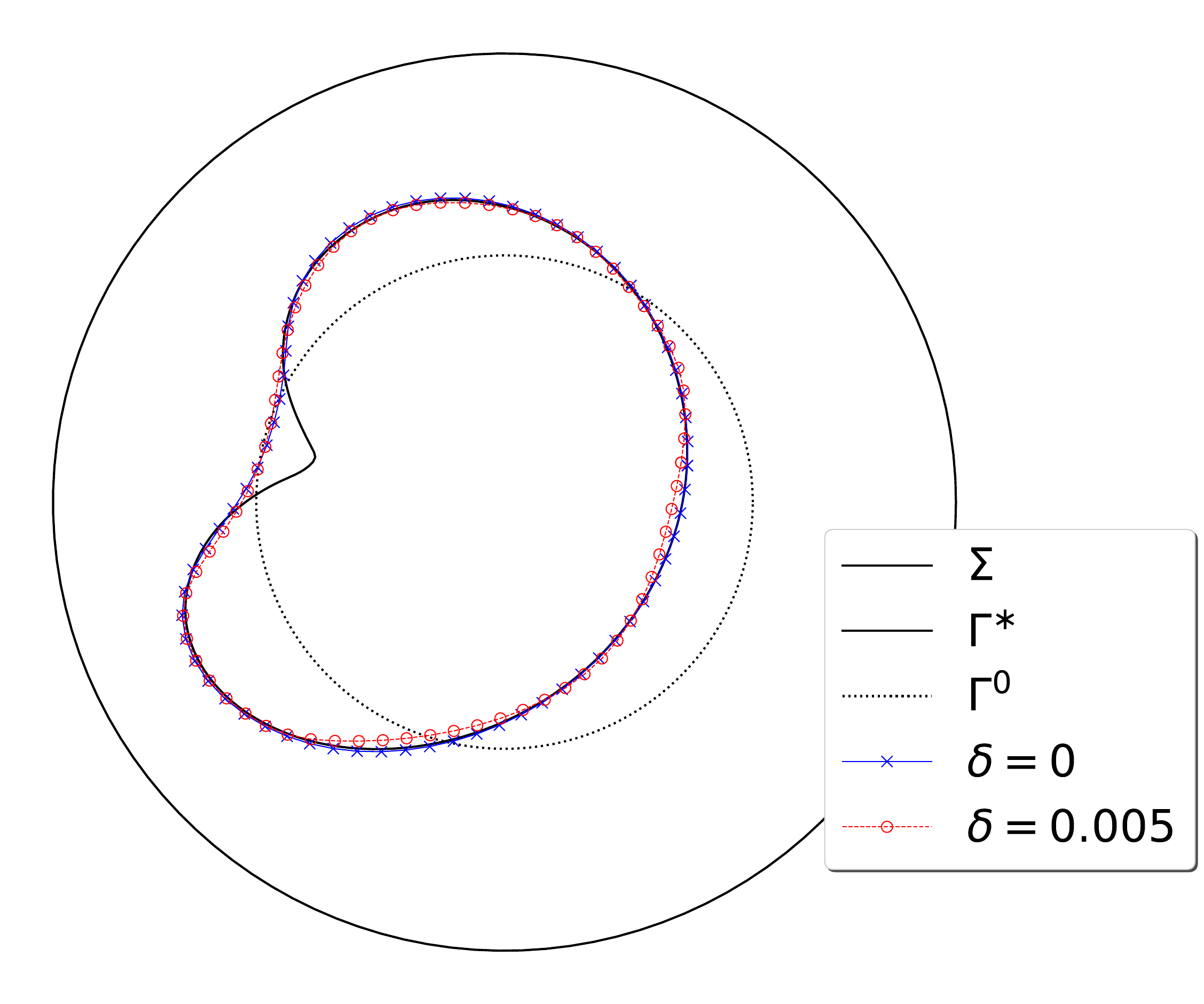}} &
	\resizebox{0.325\linewidth}{!}{\includegraphics{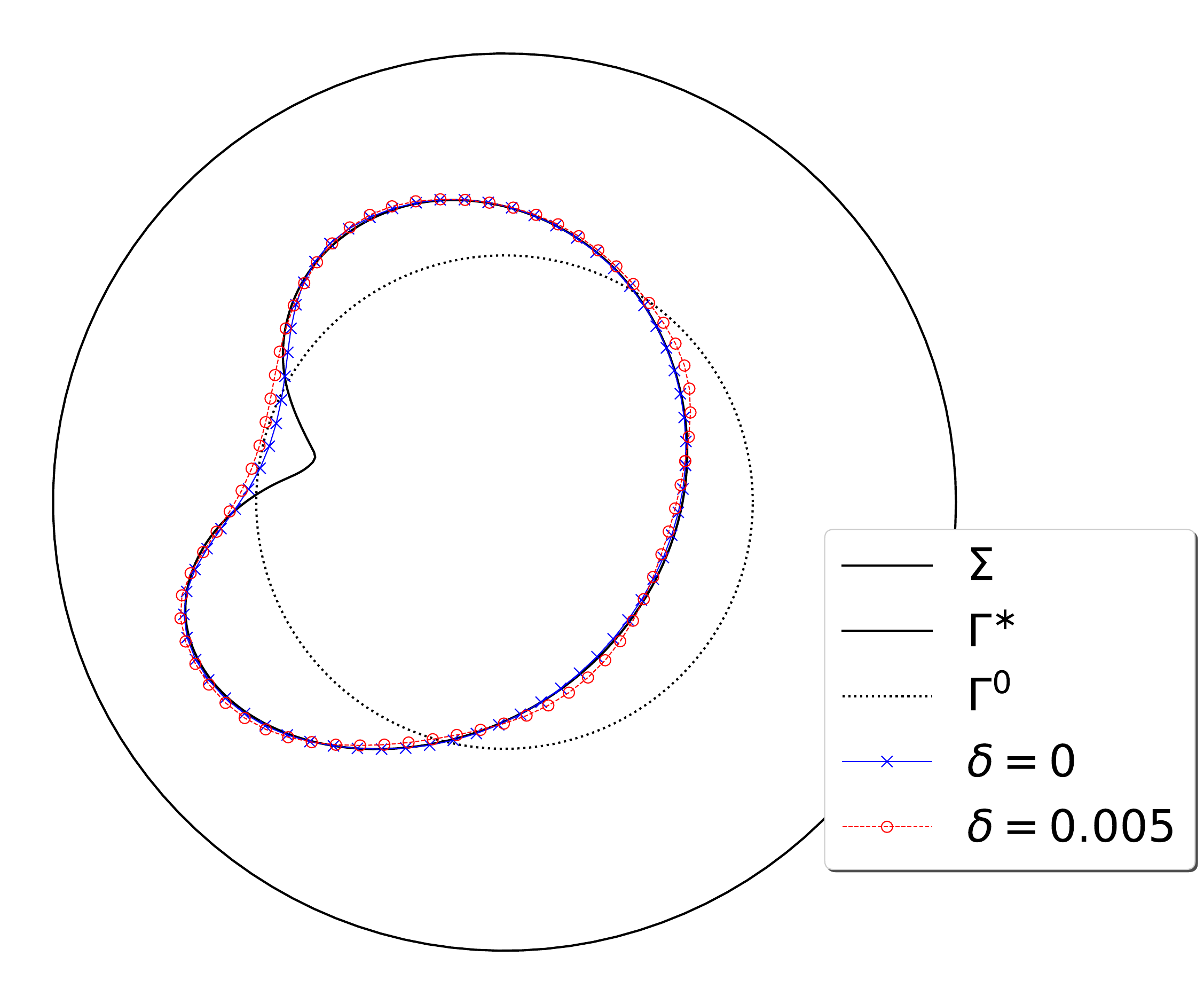}} &
	\resizebox{0.325\linewidth}{!}{\includegraphics{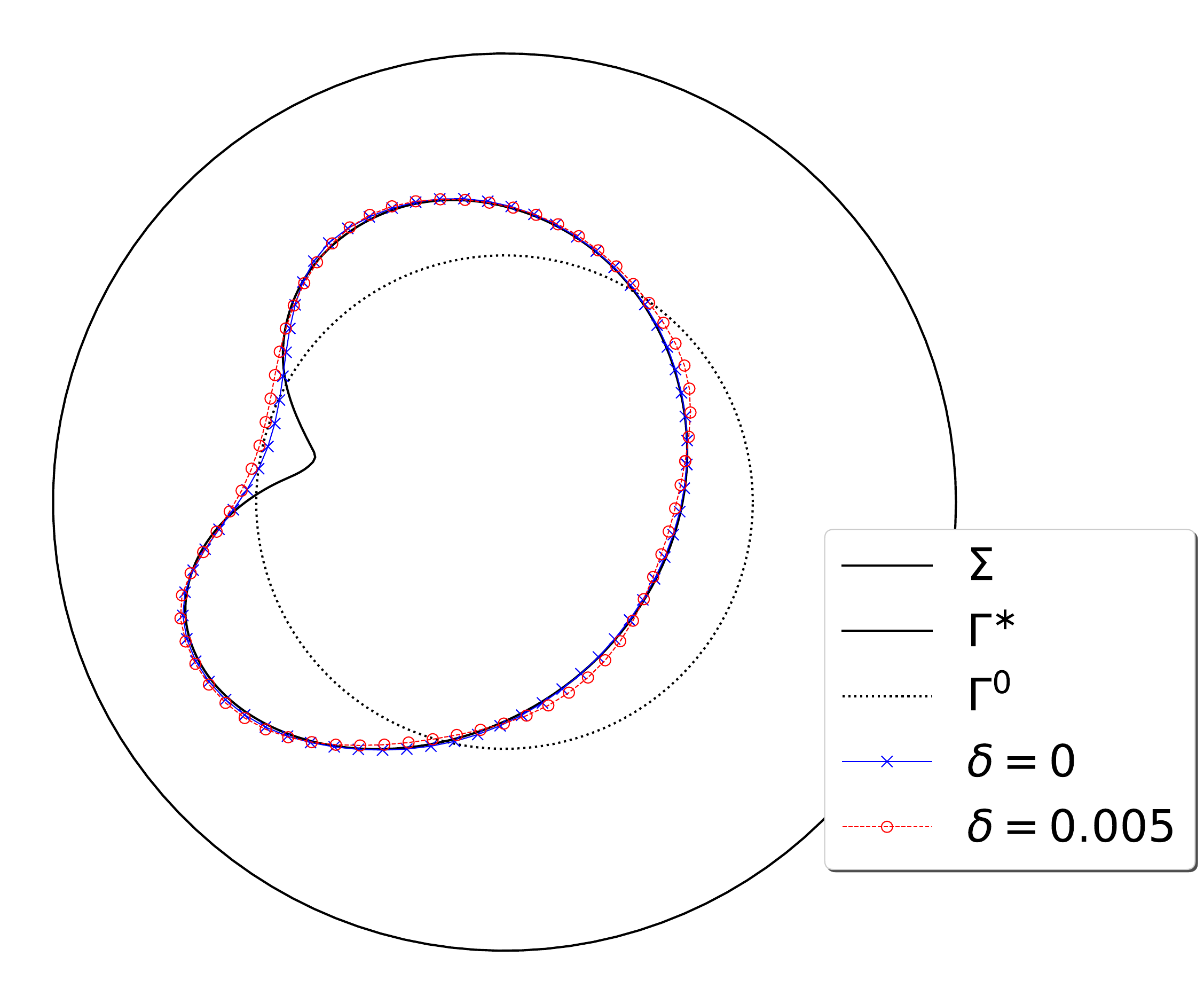}} \\ \hline
\end{tabular}
\end{adjustbox}
\caption{Reconstructed of shapes}
\label{tab:shapes}
\end{table}
\begin{table}[htp!]
\centering
\begin{adjustbox}{max width=\textwidth}
\begin{tabular}{|c|c|c|c|}
\hline
\textbf{Case} & \textbf{A1} & \textbf{A2} & \textbf{A3} \\ \hline
\textbf{B1} & \resizebox{0.325\linewidth}{!}{\includegraphics{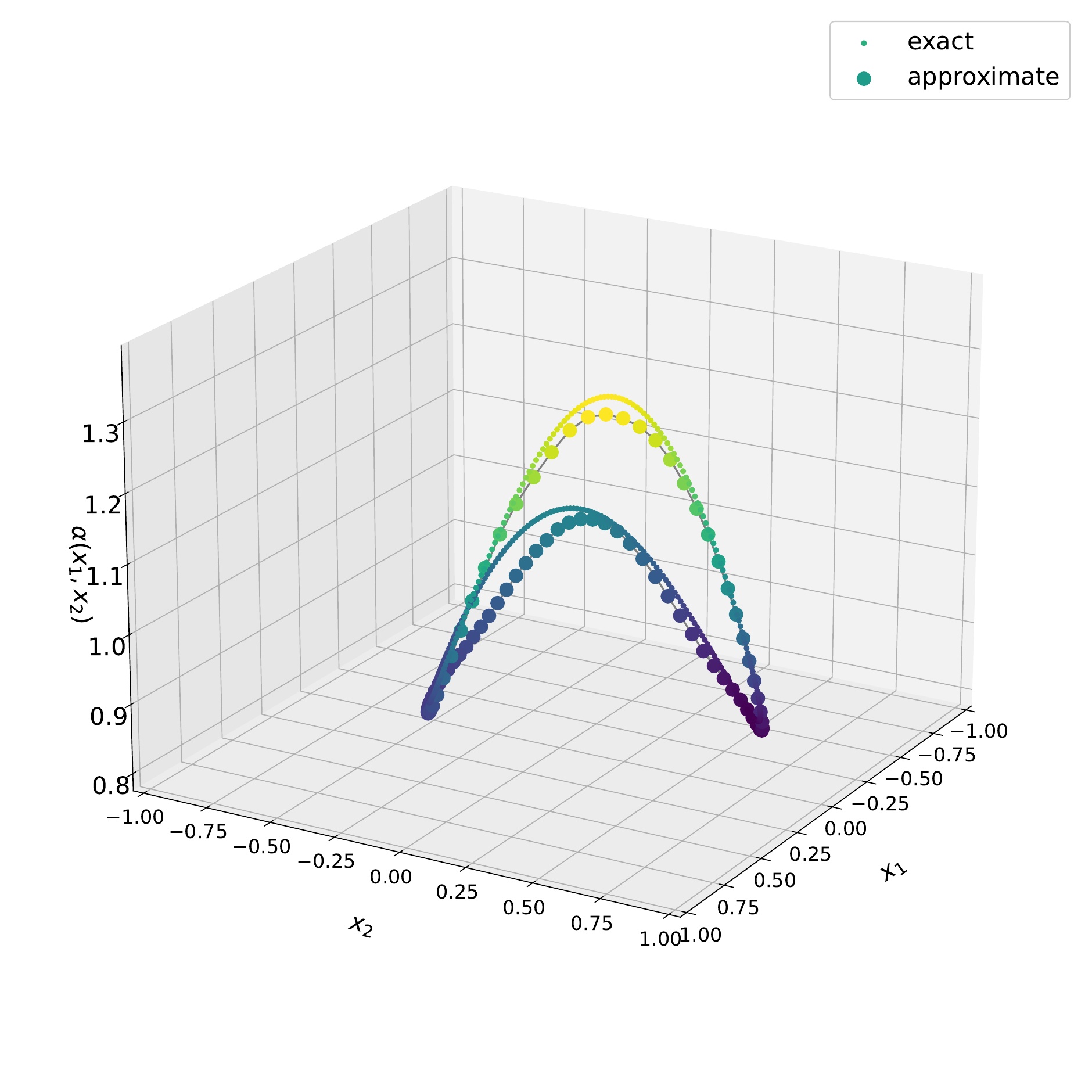}} &
	\resizebox{0.325\linewidth}{!}{\includegraphics{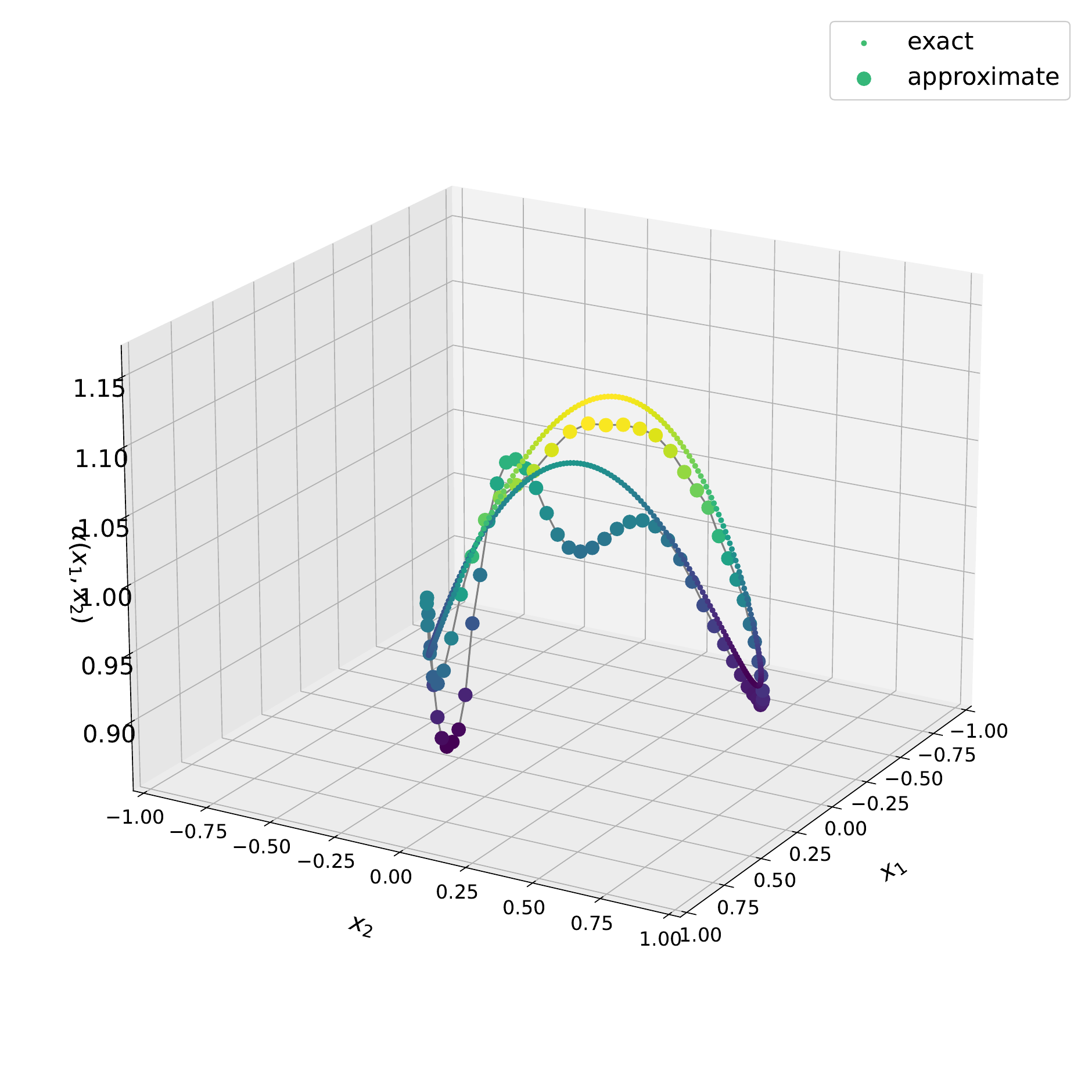}} &
	\resizebox{0.325\linewidth}{!}{\includegraphics{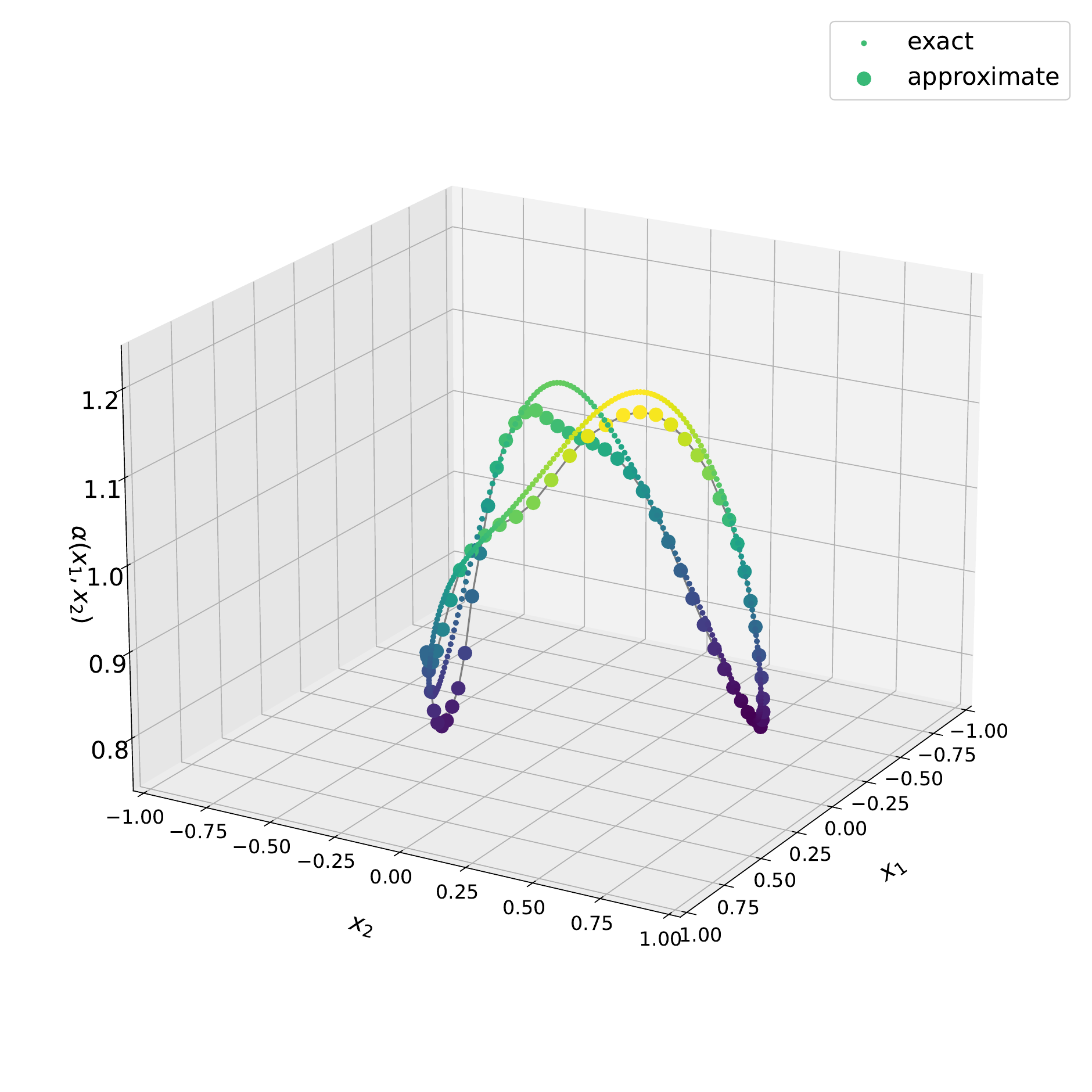}} \\ \hline
\textbf{B2} & \resizebox{0.325\linewidth}{!}{\includegraphics{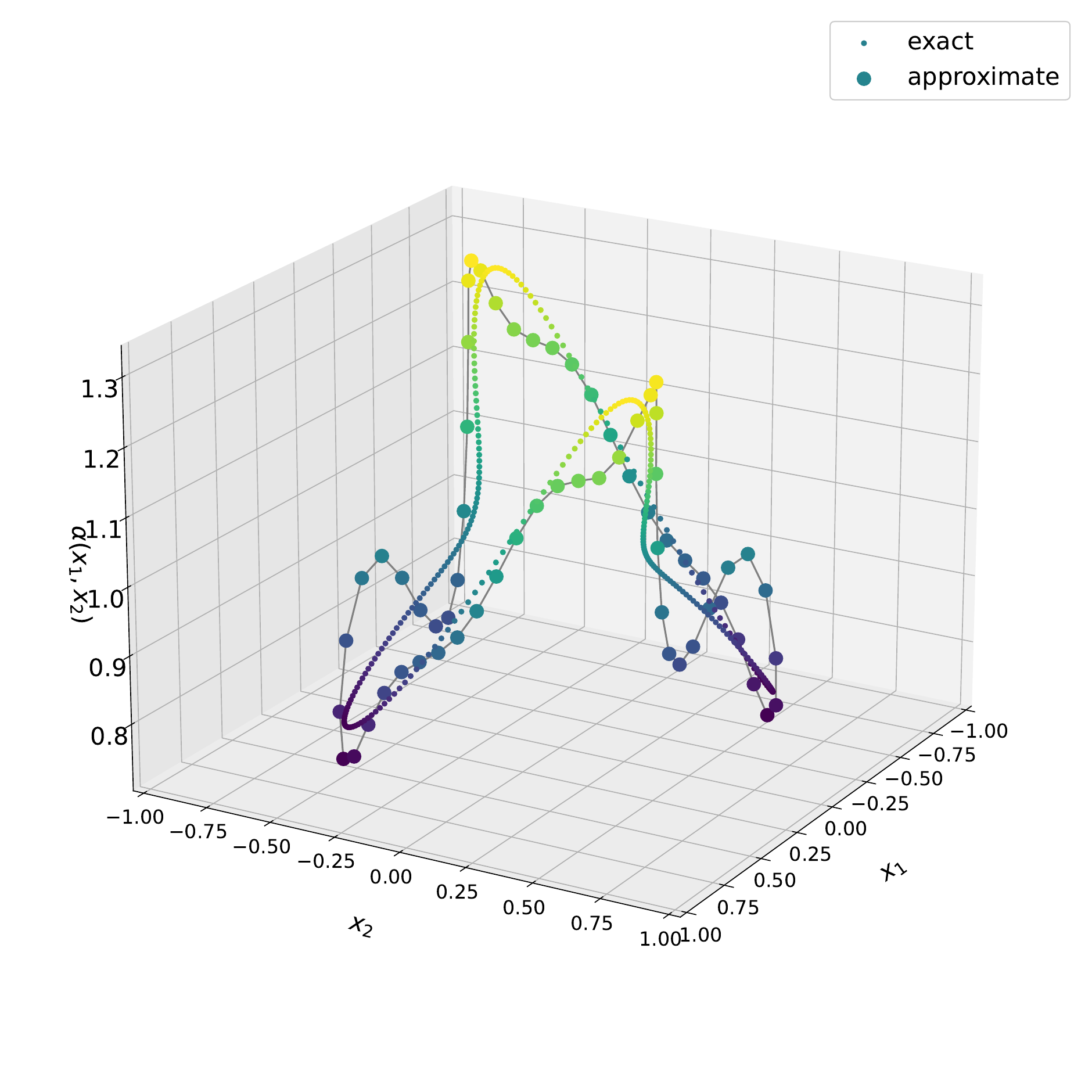}} &
	\resizebox{0.325\linewidth}{!}{\includegraphics{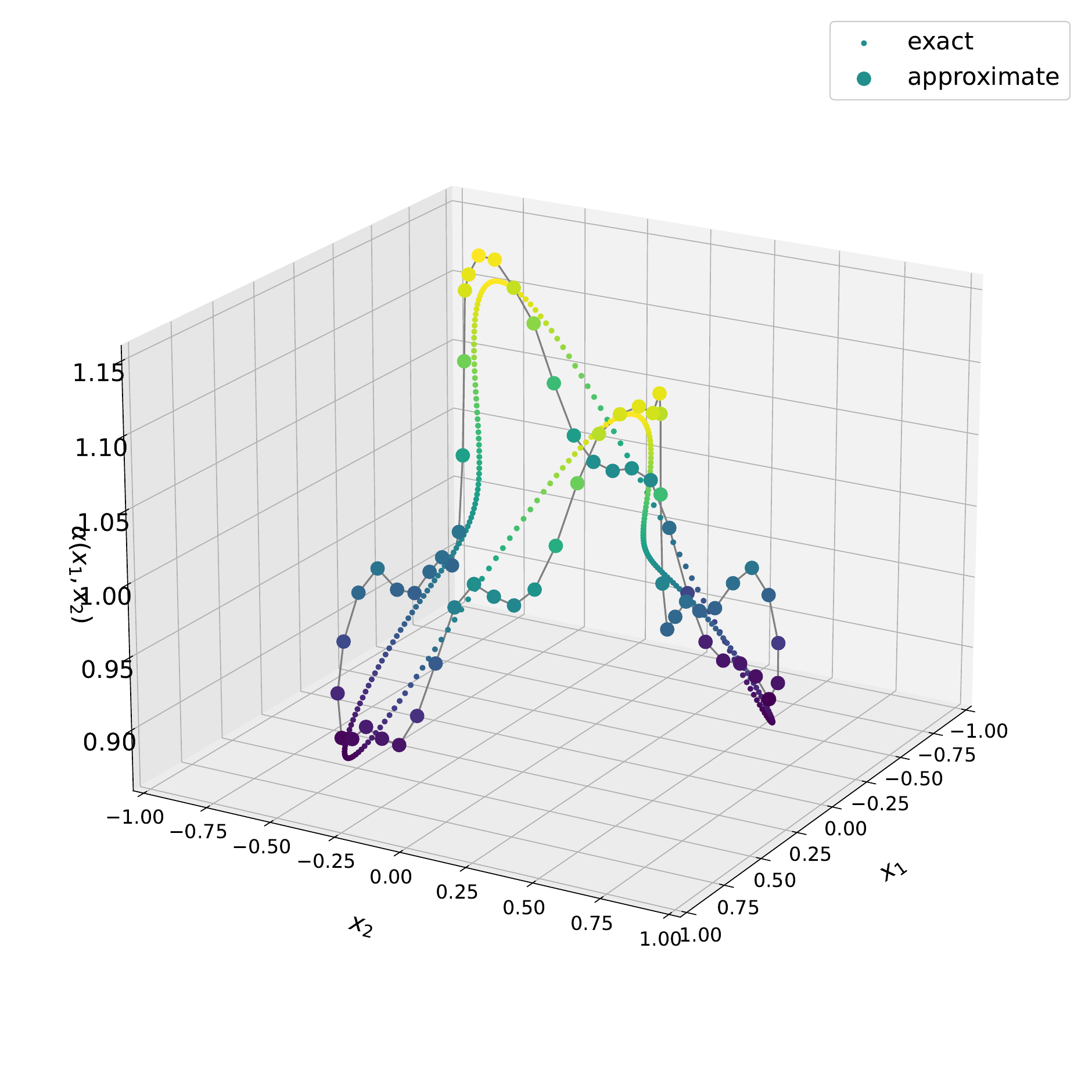}} &
	\resizebox{0.325\linewidth}{!}{\includegraphics{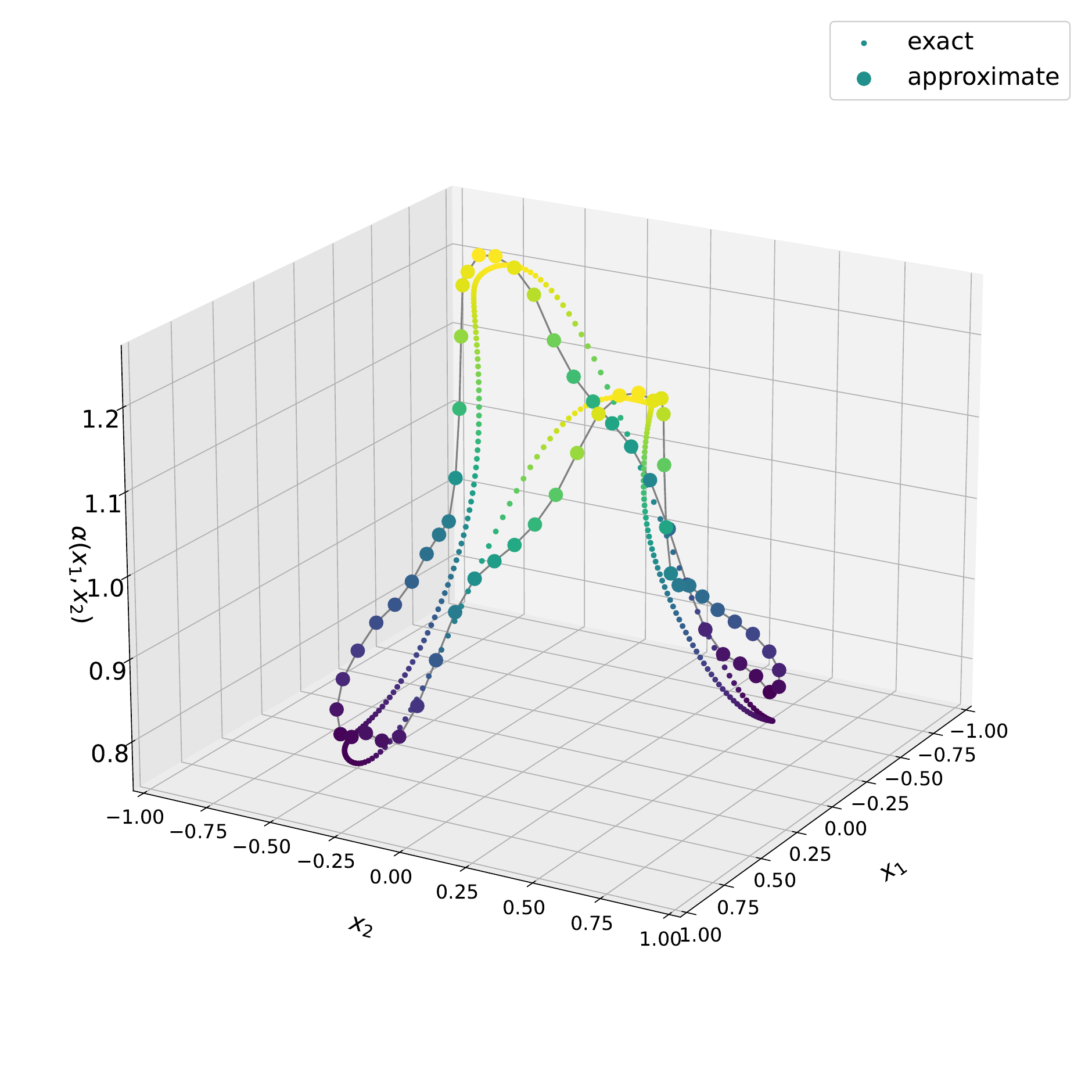}} \\ \hline
\textbf{B3} & \resizebox{0.325\linewidth}{!}{\includegraphics{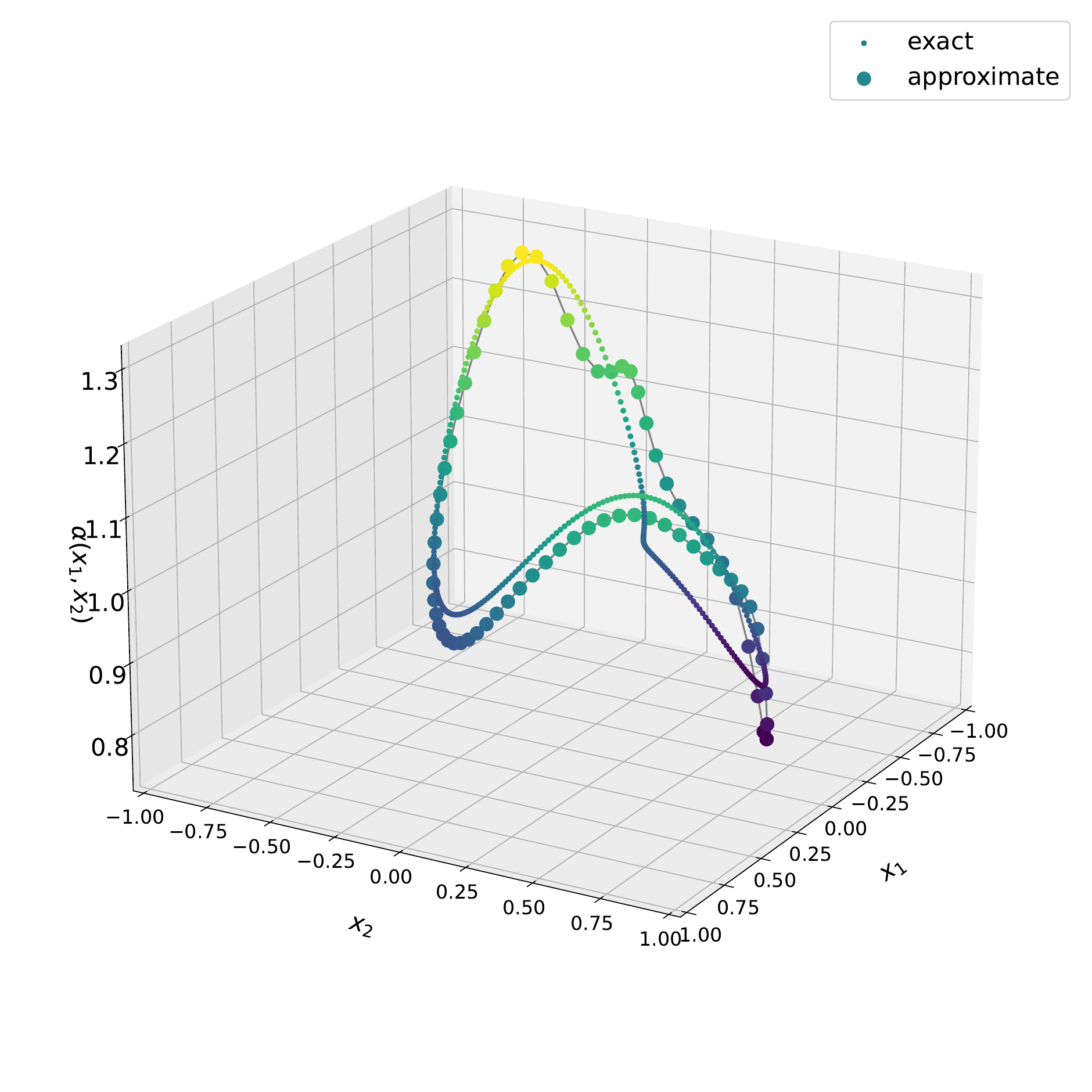}} &
	\resizebox{0.325\linewidth}{!}{\includegraphics{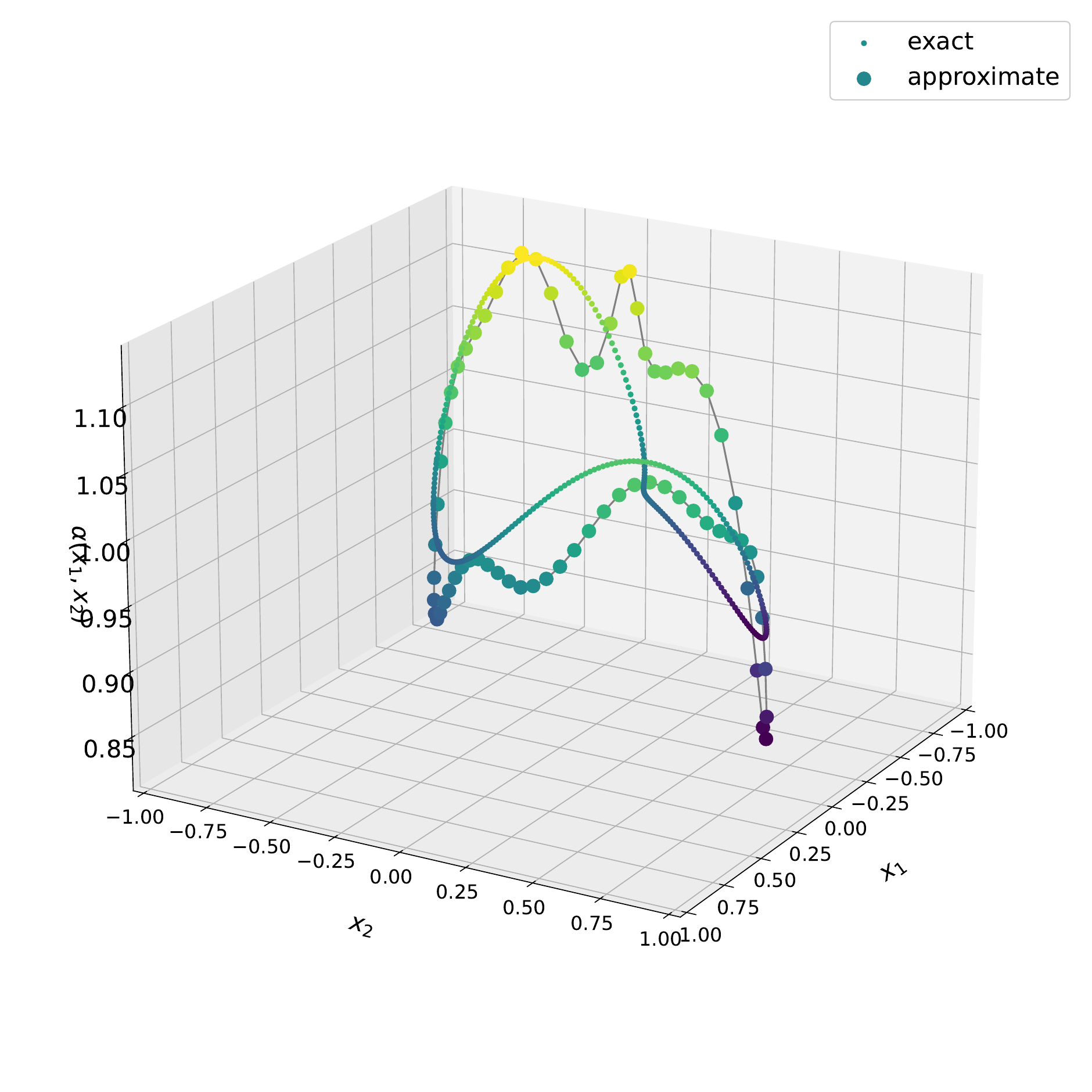}} &
	\resizebox{0.325\linewidth}{!}{\includegraphics{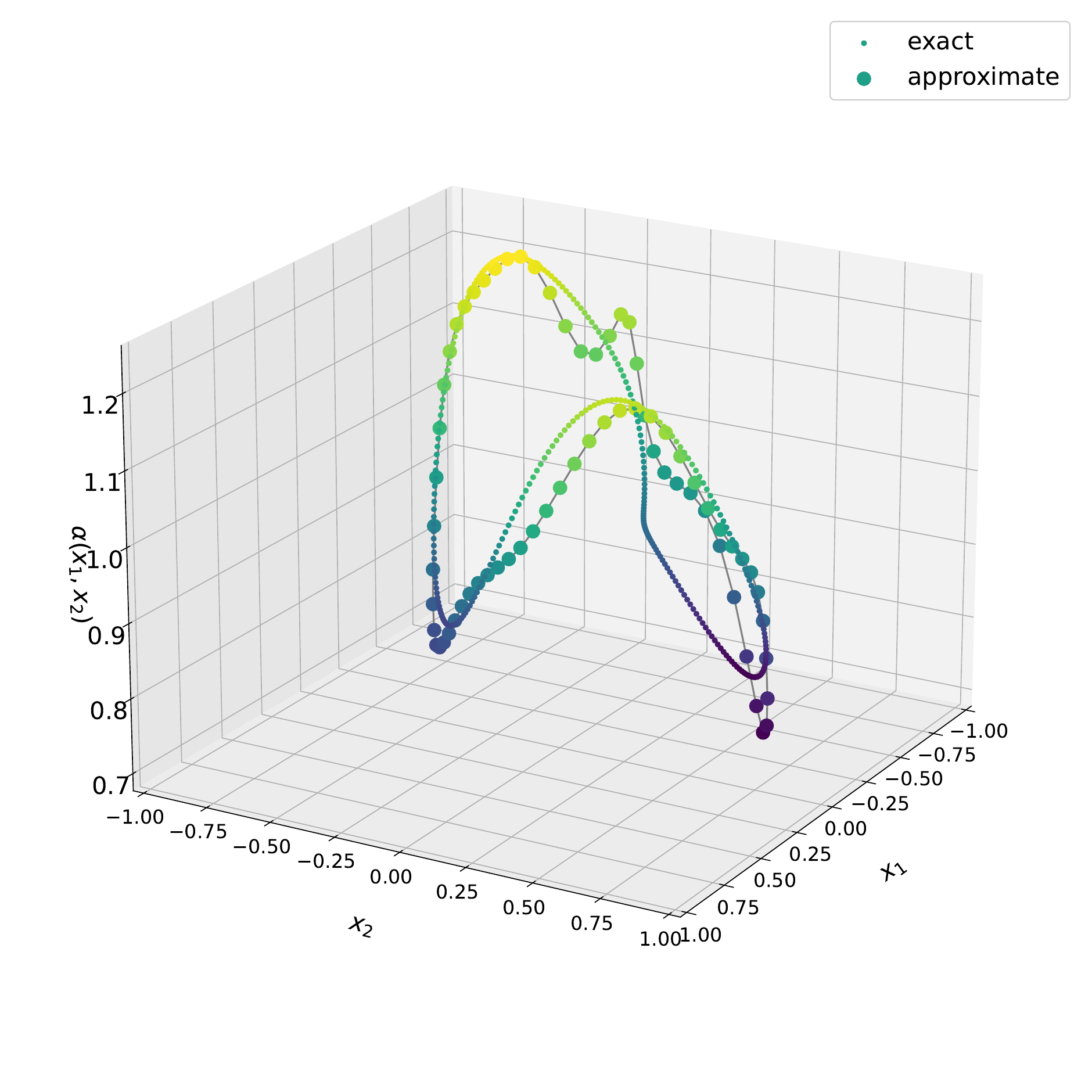}} \\ \hline
\end{tabular}
\end{adjustbox}
\caption{3D view of the reconstructed Robin coefficient, rotated by $30^\circ$}
\label{tab:robin_30}
\end{table}
\begin{table}[htp!]
\centering
\begin{adjustbox}{max width=\textwidth}
\begin{tabular}{|c|c|c|c|}
\hline
\textbf{Case} & \textbf{A1} & \textbf{A2} & \textbf{A3} \\ \hline
\textbf{B1} & \resizebox{0.325\linewidth}{!}{\includegraphics{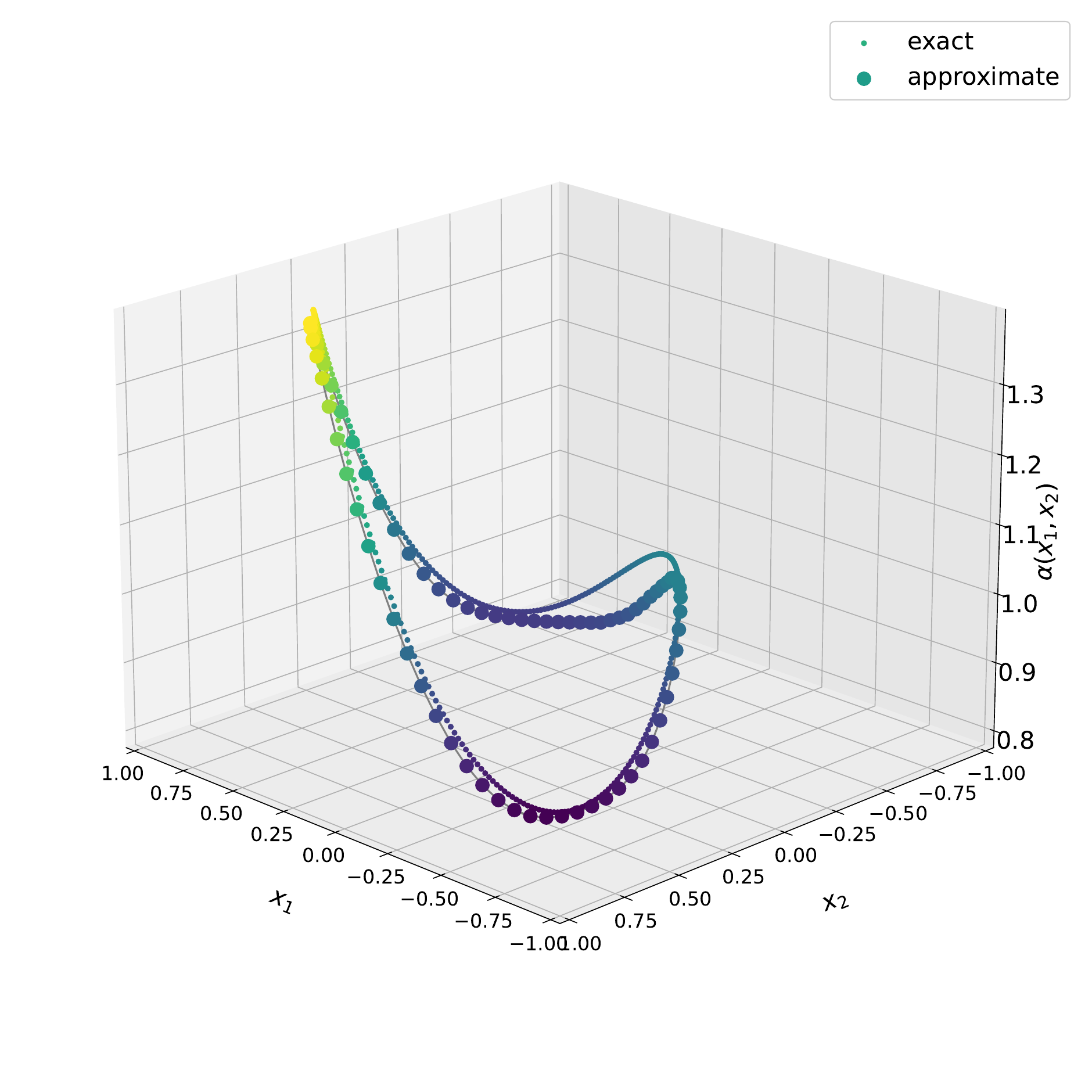}} &
	\resizebox{0.325\linewidth}{!}{\includegraphics{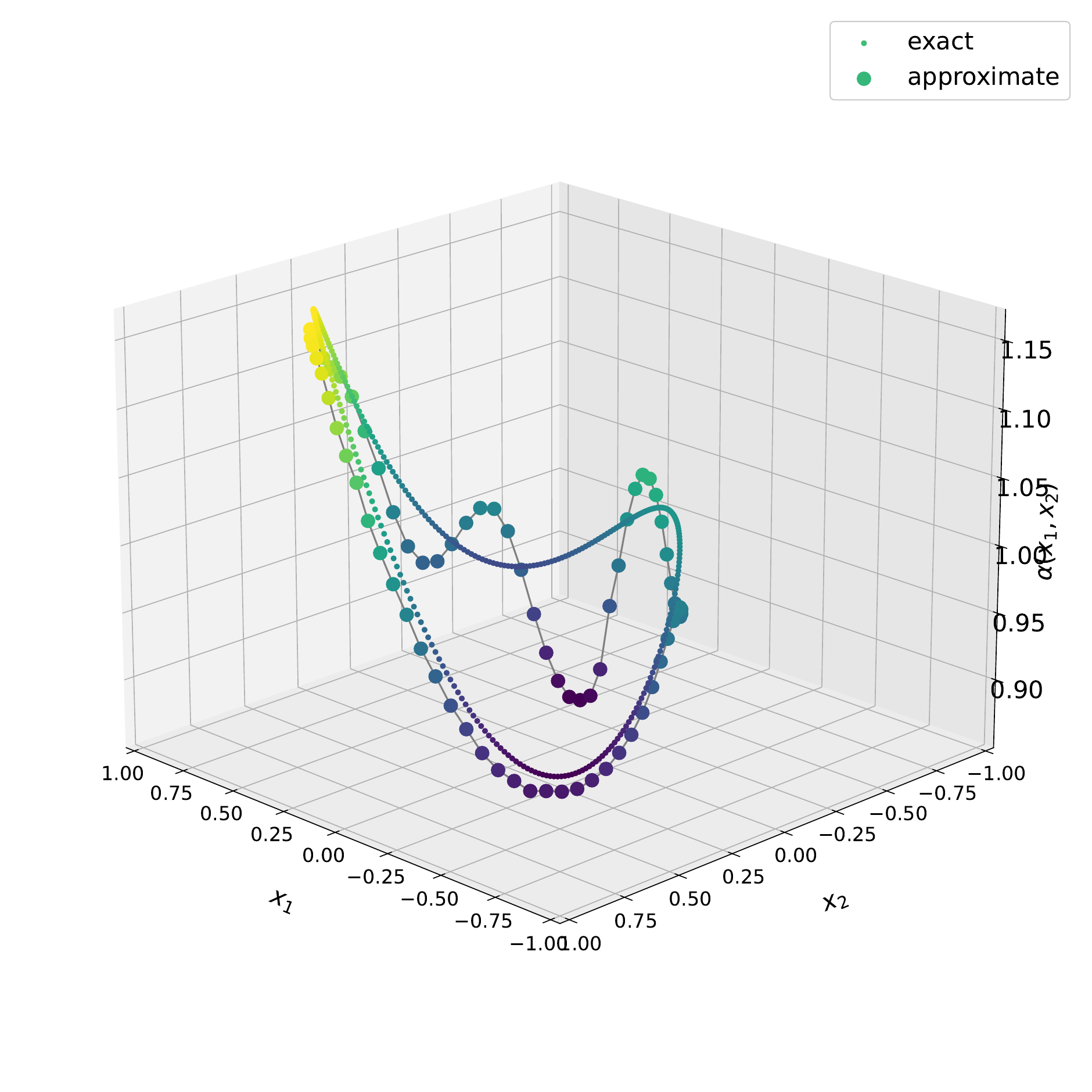}} &
	\resizebox{0.325\linewidth}{!}{\includegraphics{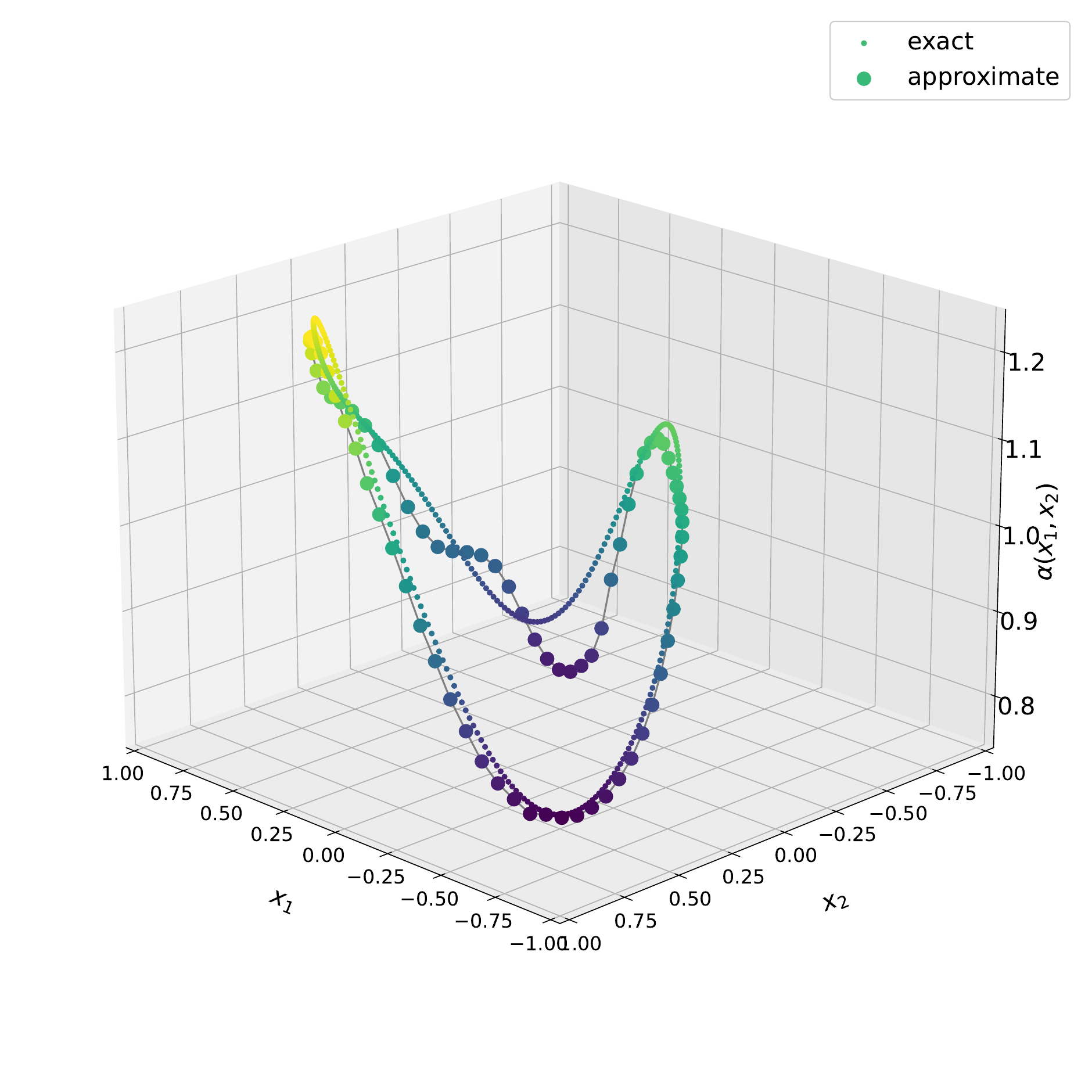}} \\ \hline
\textbf{B2} & \resizebox{0.325\linewidth}{!}{\includegraphics{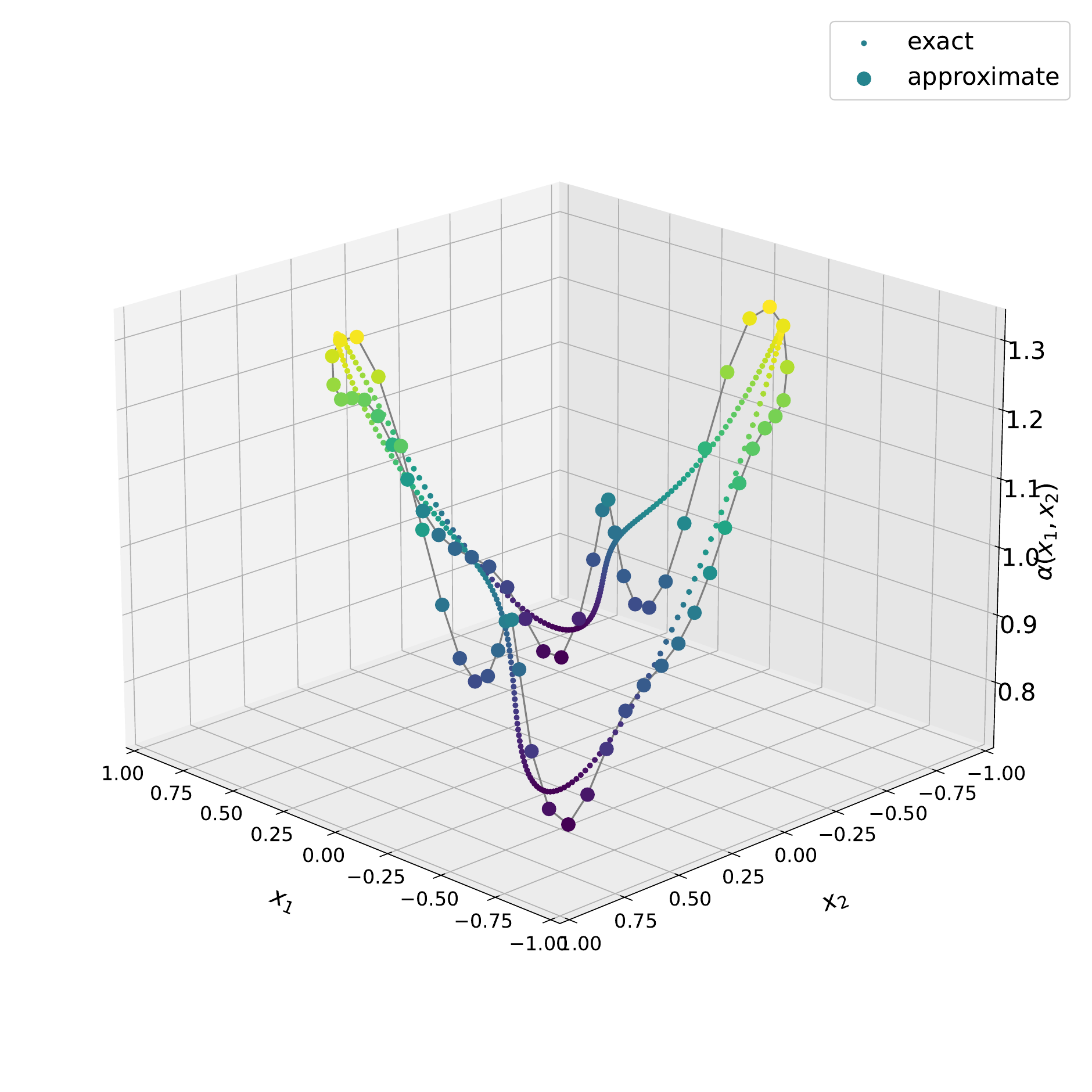}} &
	\resizebox{0.325\linewidth}{!}{\includegraphics{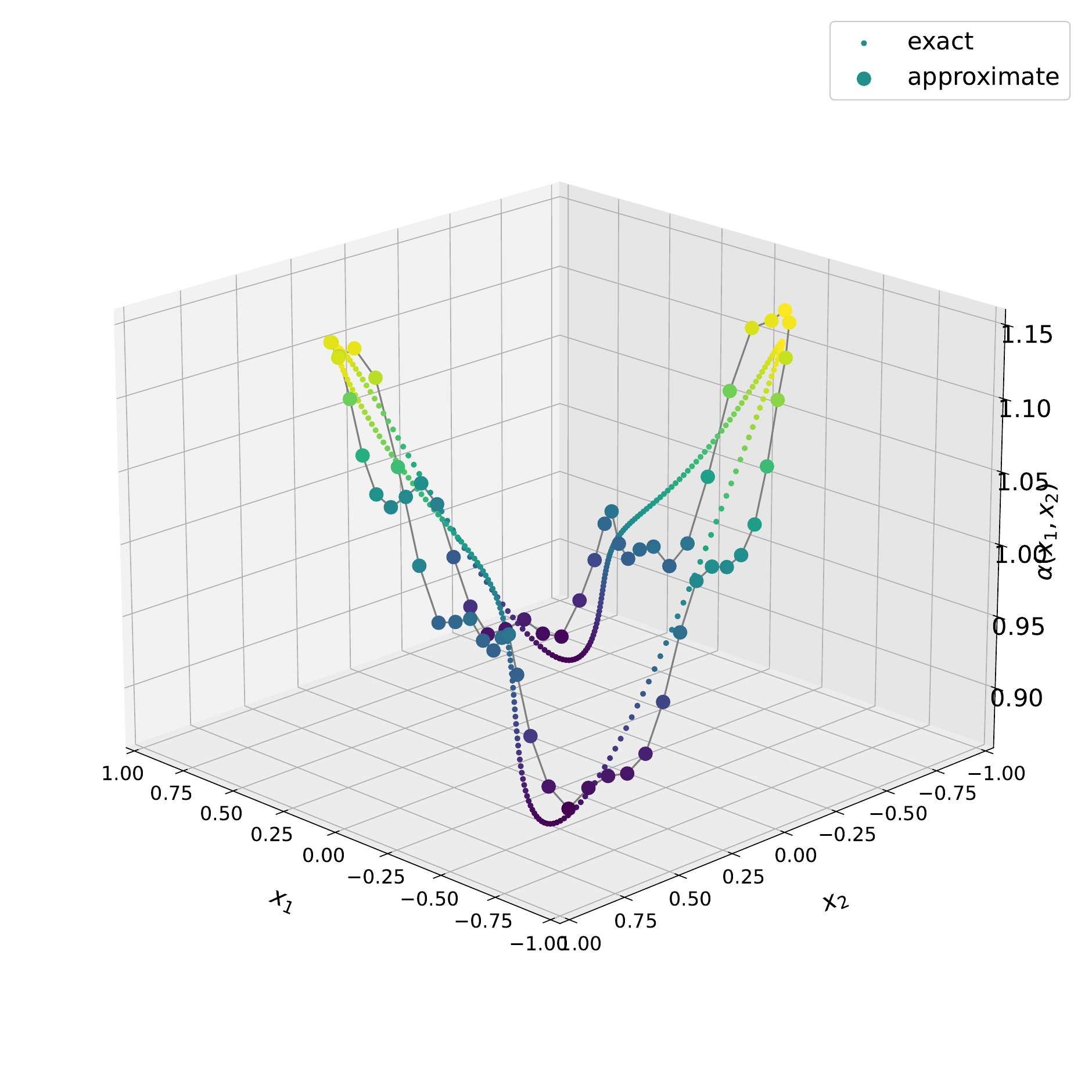}} &
	\resizebox{0.325\linewidth}{!}{\includegraphics{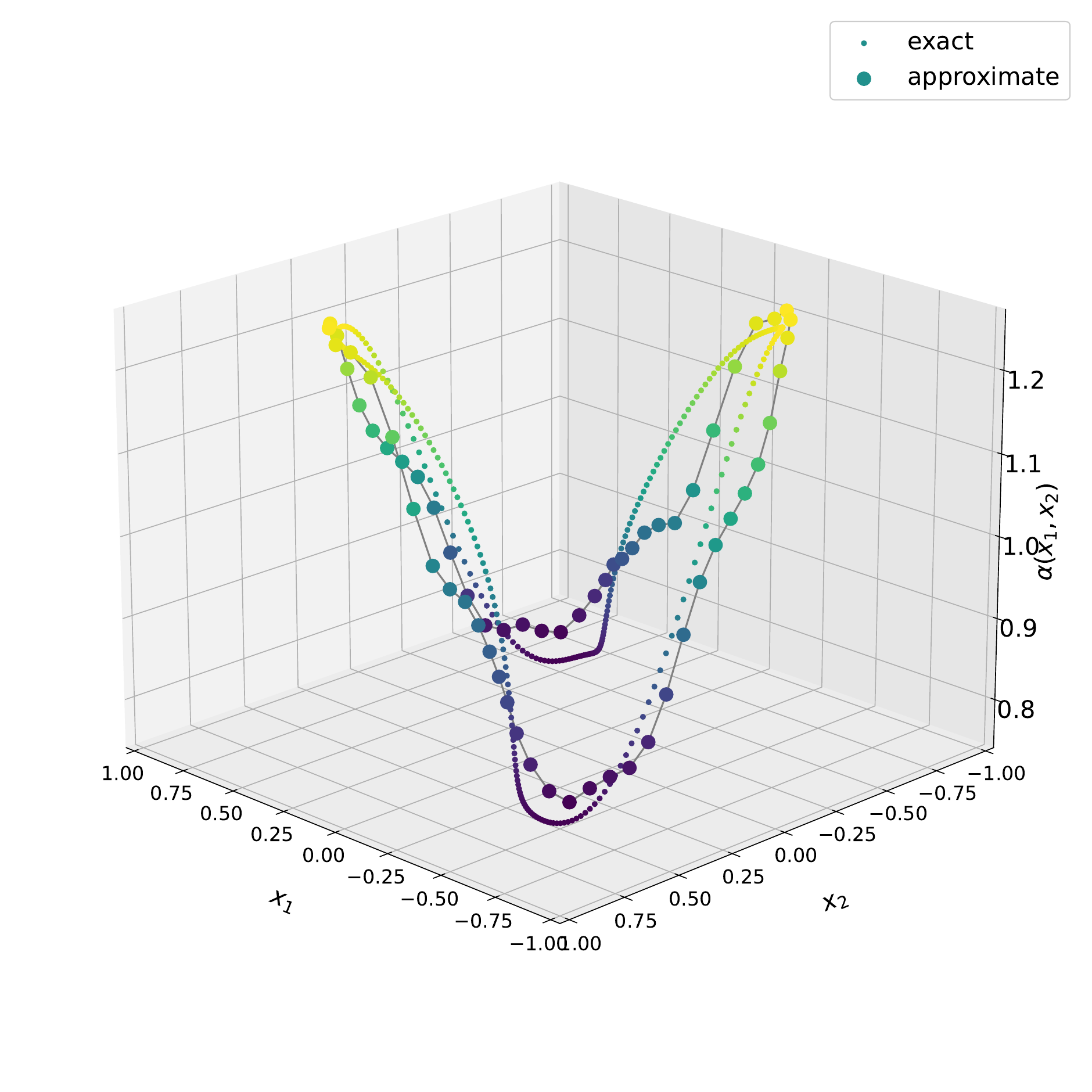}} \\ \hline
\textbf{B3} & \resizebox{0.325\linewidth}{!}{\includegraphics{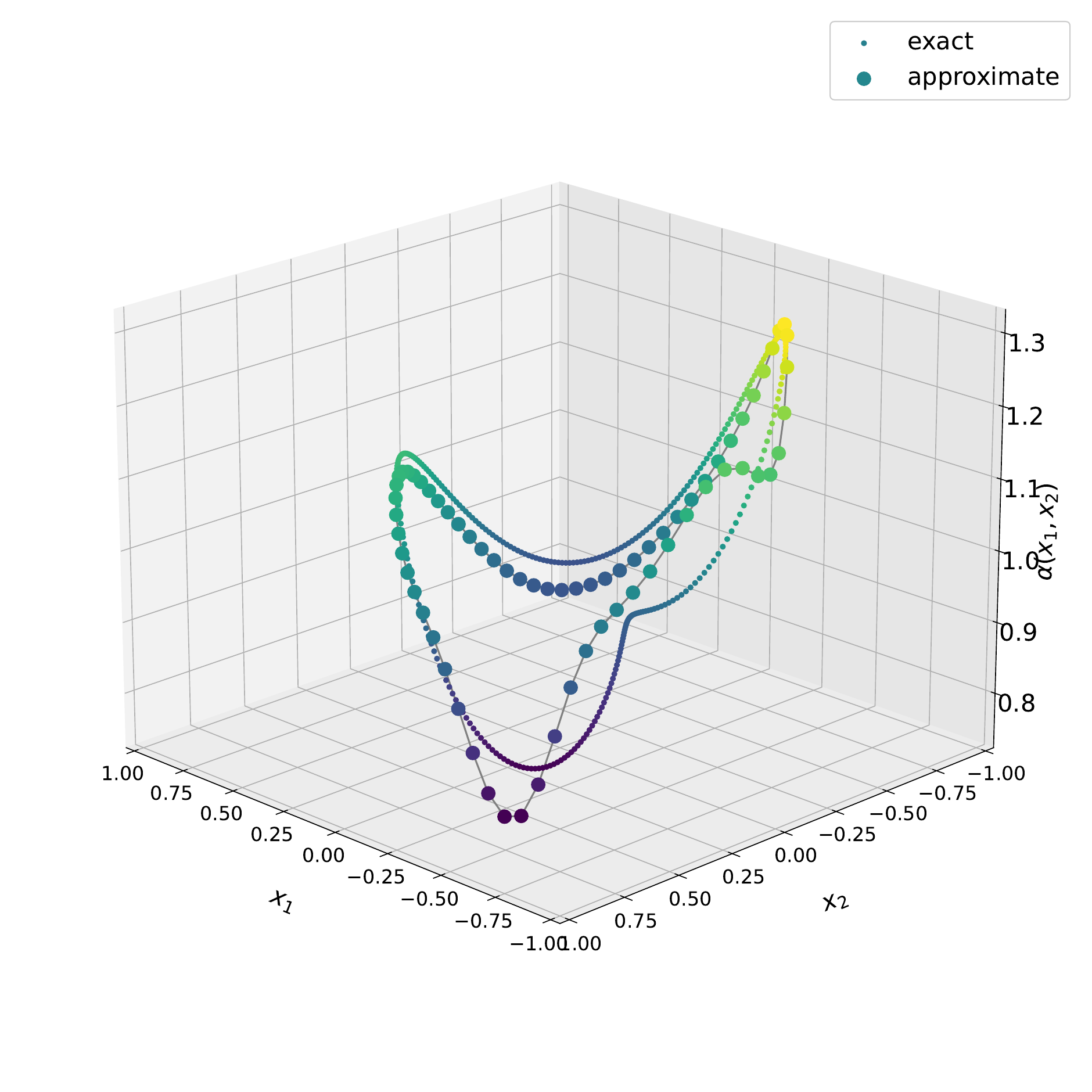}} &
	\resizebox{0.325\linewidth}{!}{\includegraphics{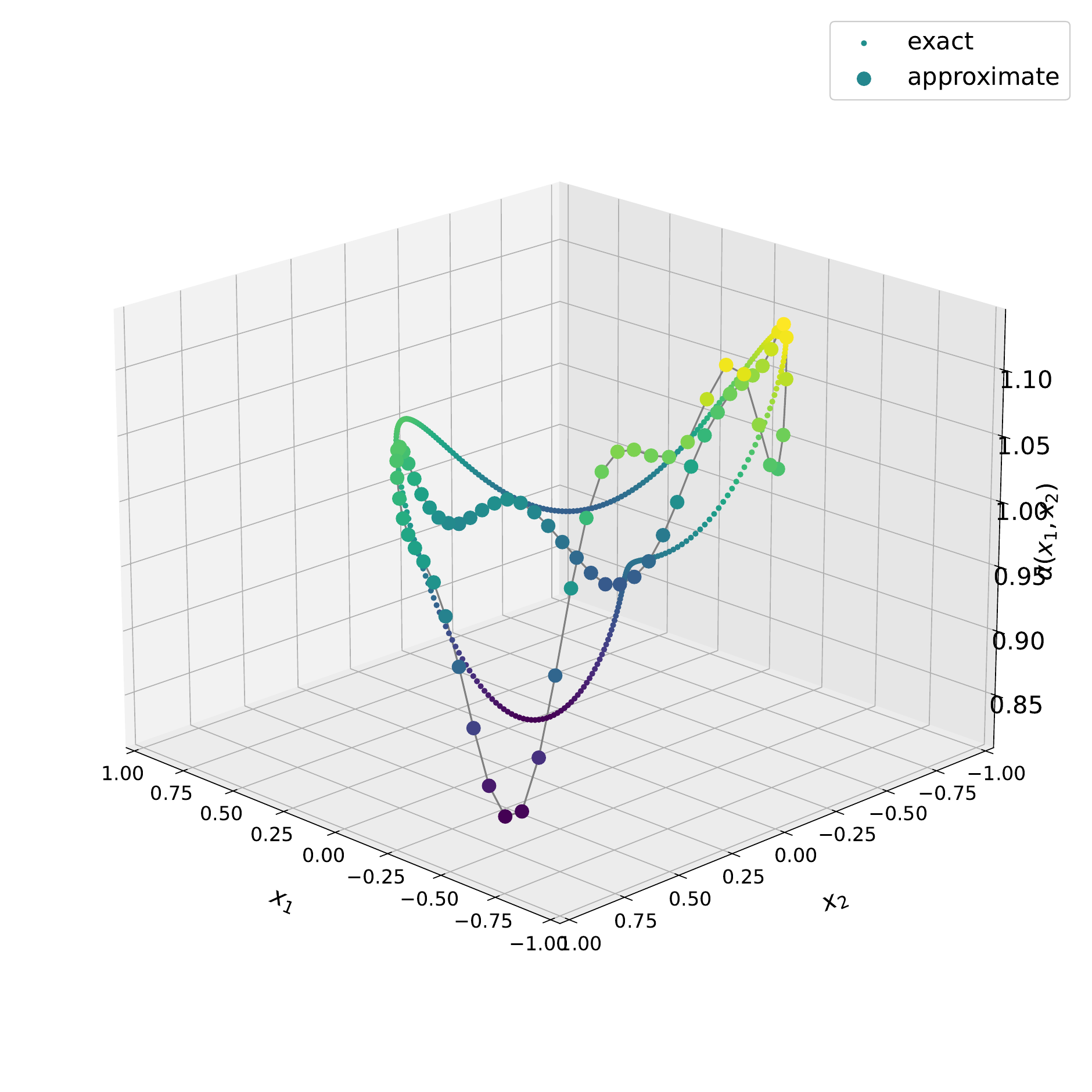}} &
	\resizebox{0.325\linewidth}{!}{\includegraphics{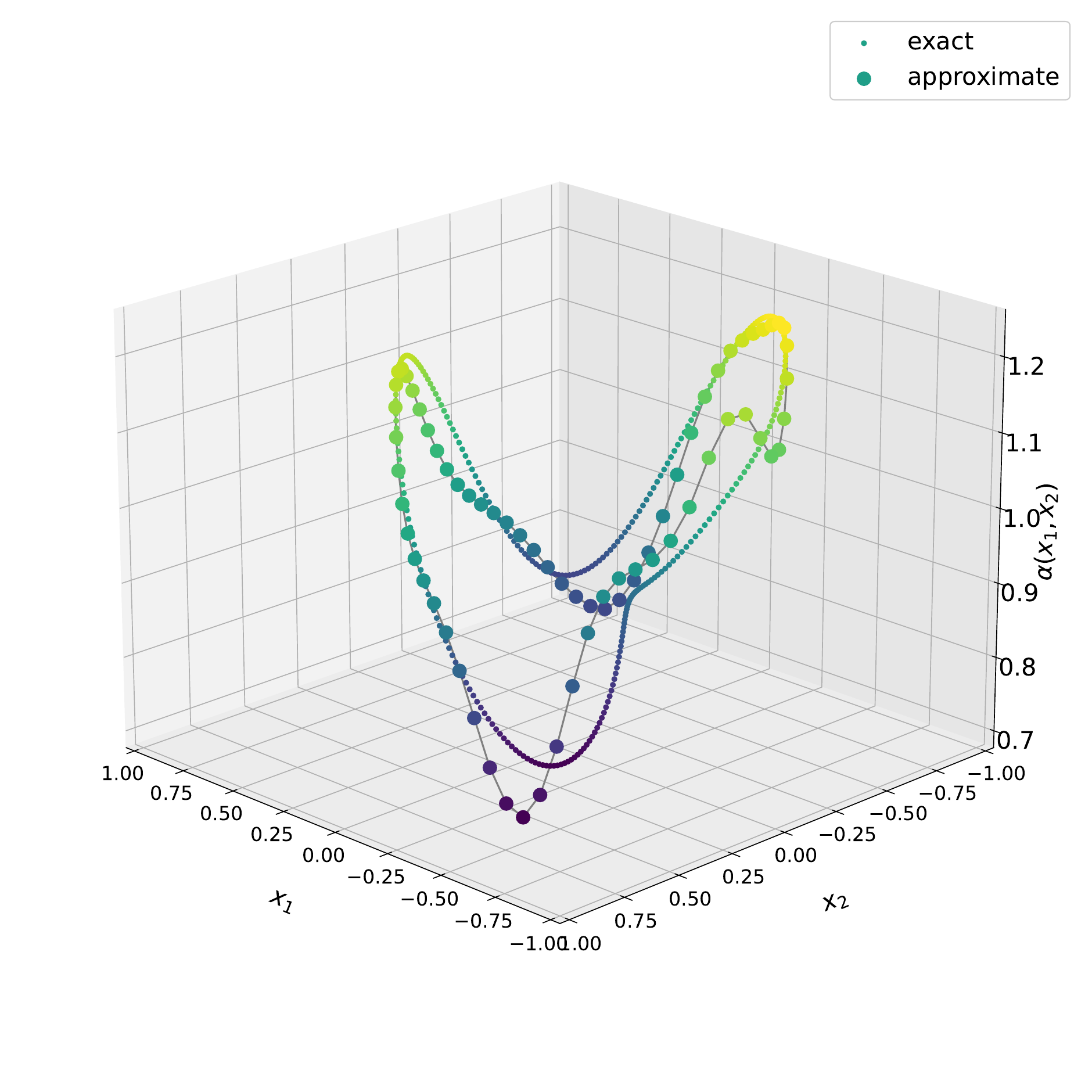}} \\ \hline
\end{tabular}
\end{adjustbox}
\caption{3D view of the reconstructed Robin coefficient, rotated by $135^\circ$}
\label{tab:robin_135}
\end{table}
\begin{table}[htp!]
\centering
\begin{adjustbox}{max width=\textwidth}
\begin{tabular}{|c|c|c|c|}
\hline
\textbf{Case} & \textbf{A1} & \textbf{A2} & \textbf{A3} \\ \hline
\textbf{B1} & \resizebox{0.325\linewidth}{!}{\includegraphics{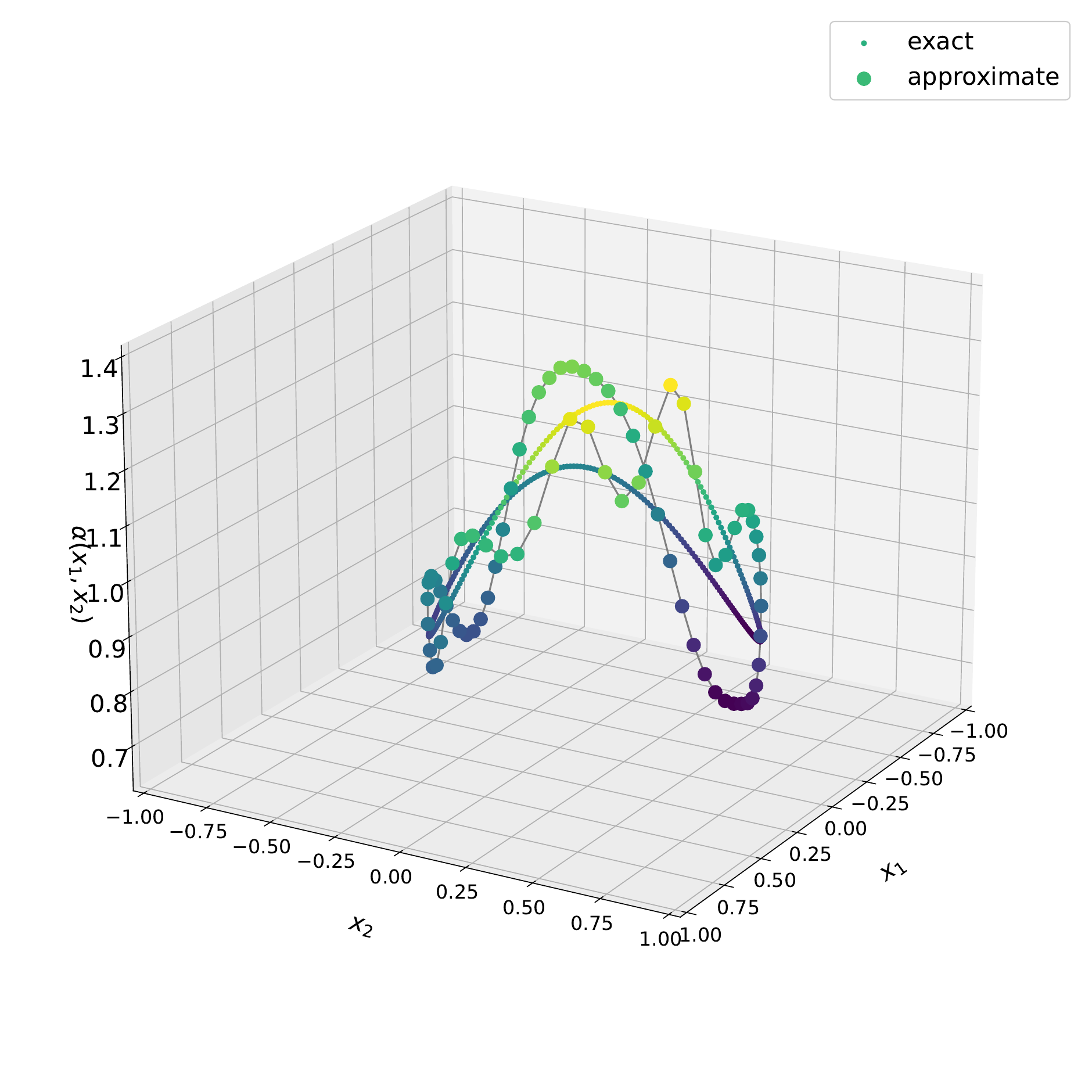}} &
	\resizebox{0.325\linewidth}{!}{\includegraphics{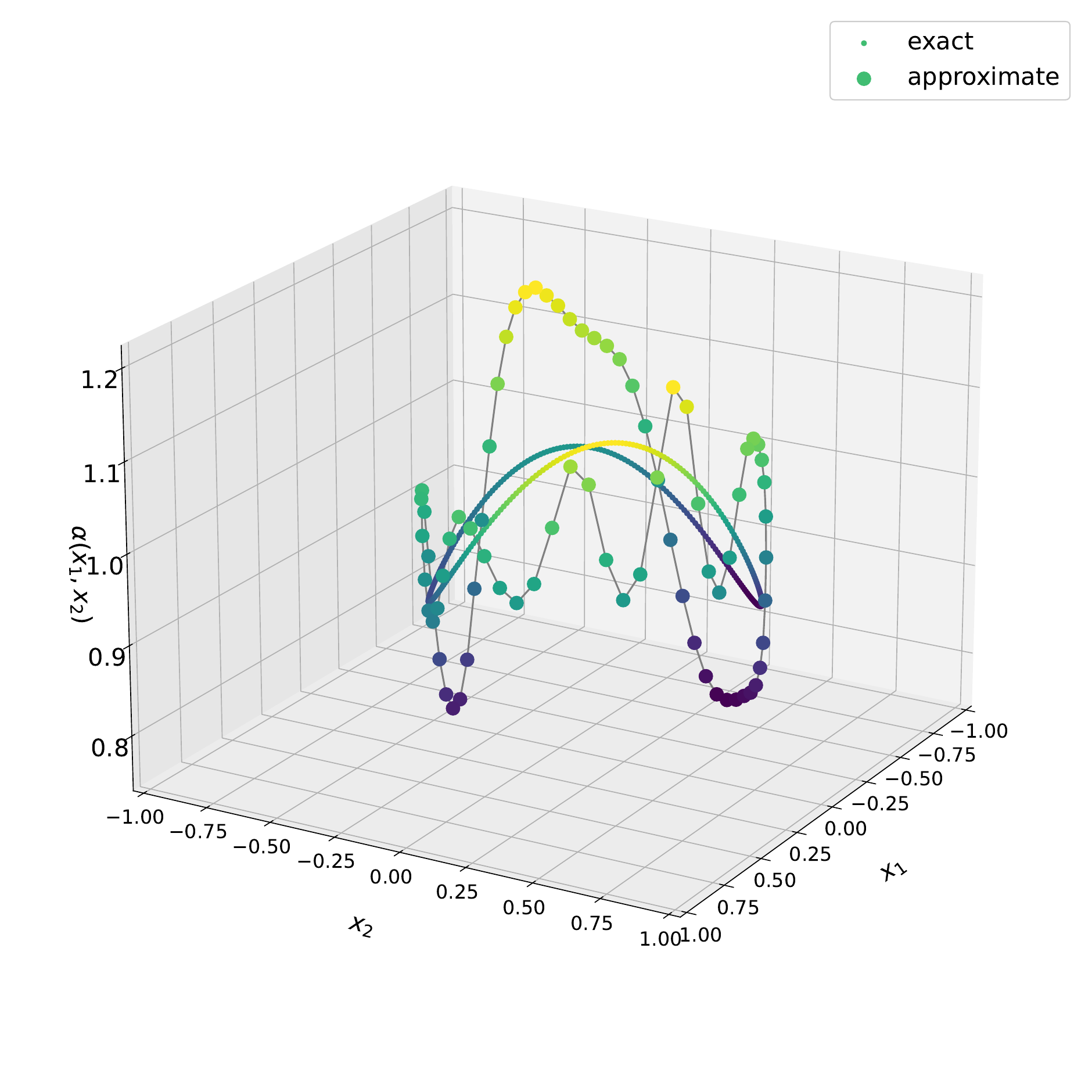}} &
	\resizebox{0.325\linewidth}{!}{\includegraphics{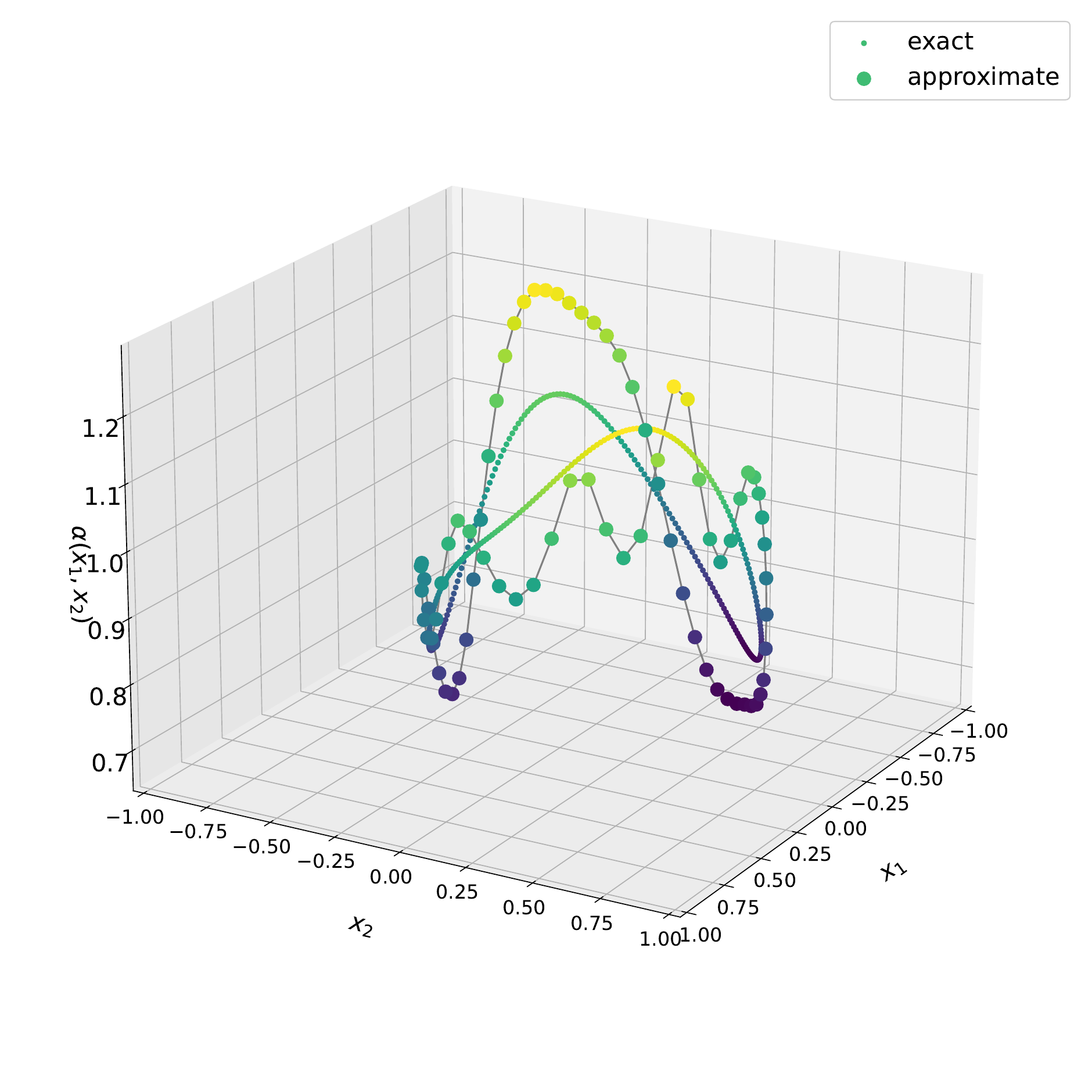}} \\ \hline
\textbf{B2} & \resizebox{0.325\linewidth}{!}{\includegraphics{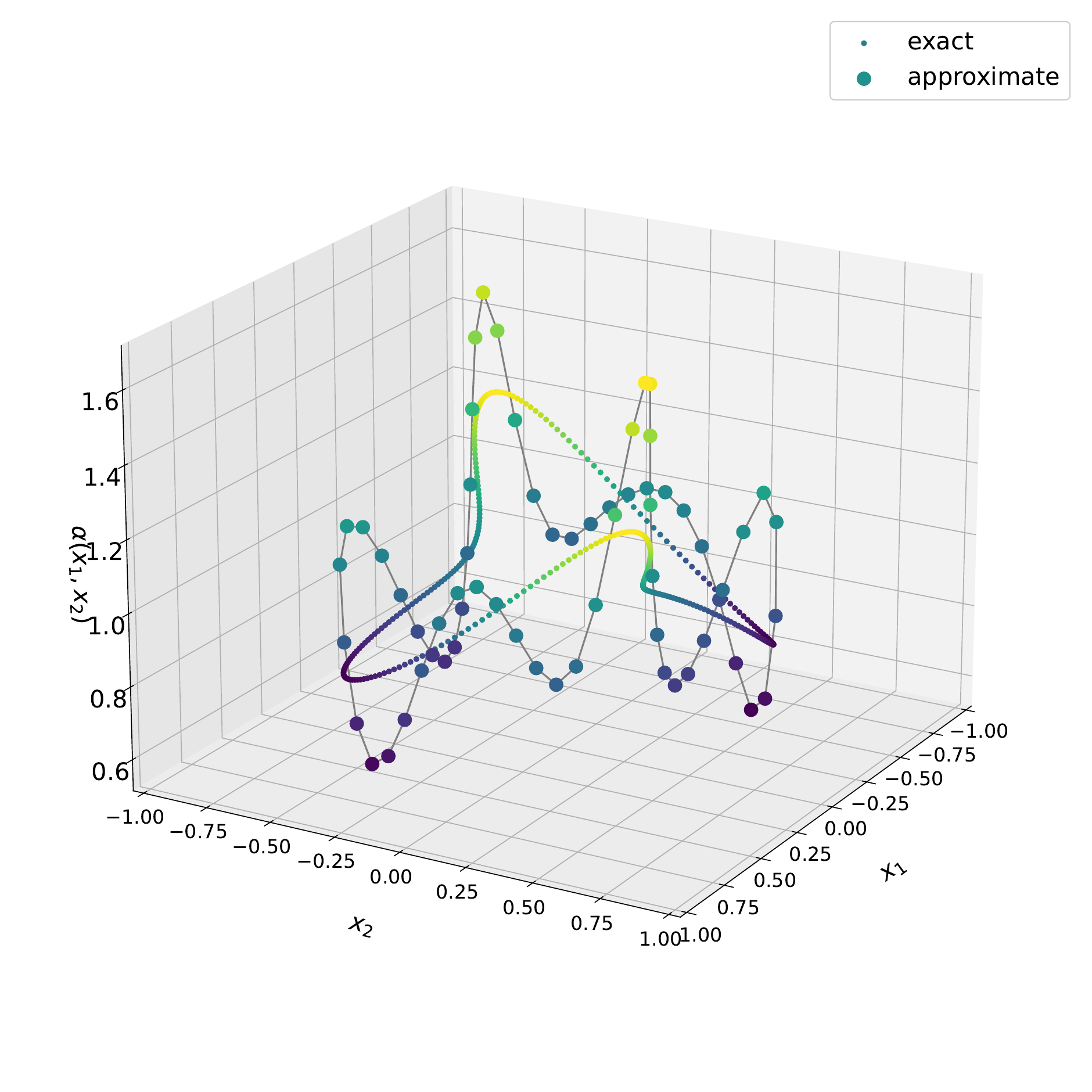}} &
	\resizebox{0.325\linewidth}{!}{\includegraphics{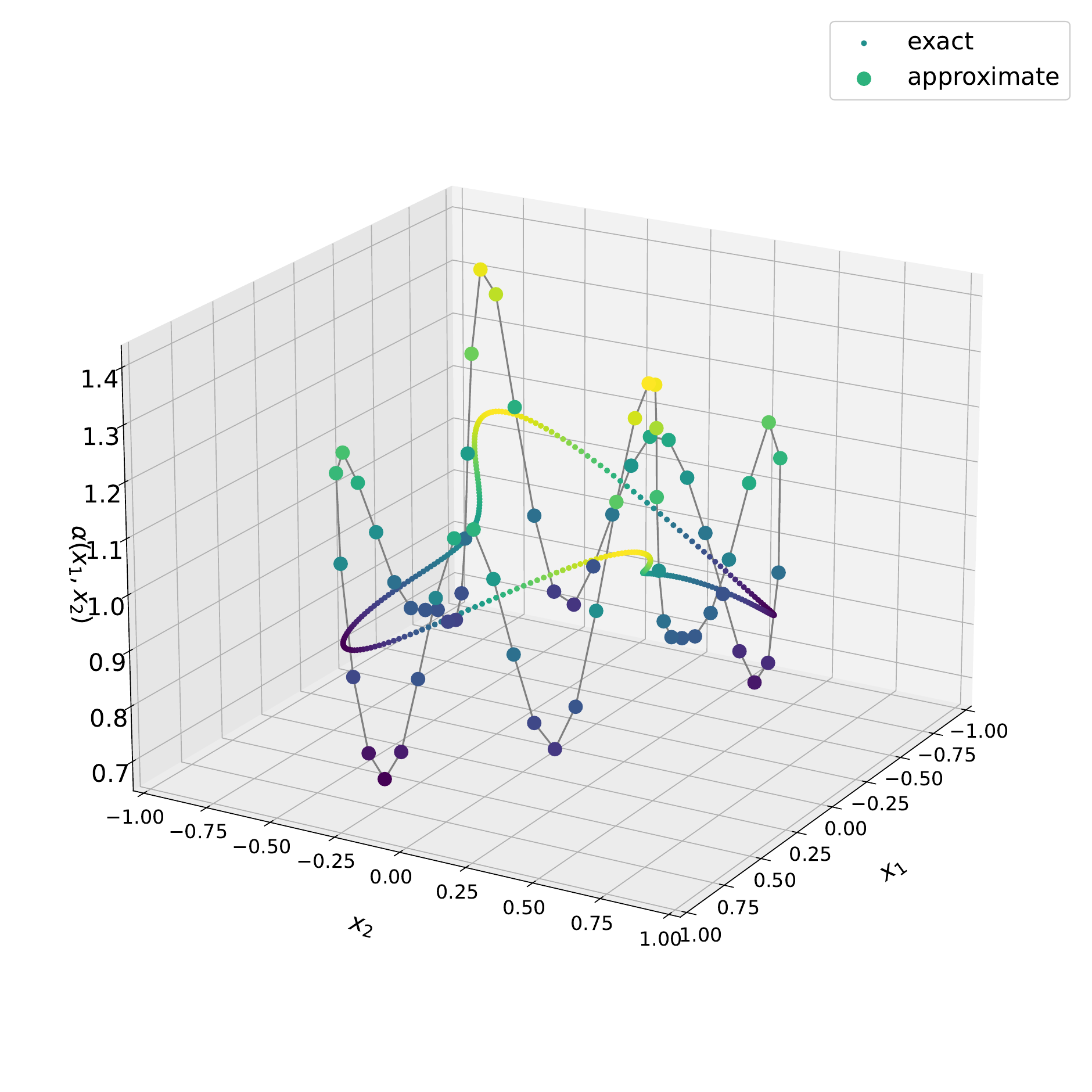}} &
	\resizebox{0.325\linewidth}{!}{\includegraphics{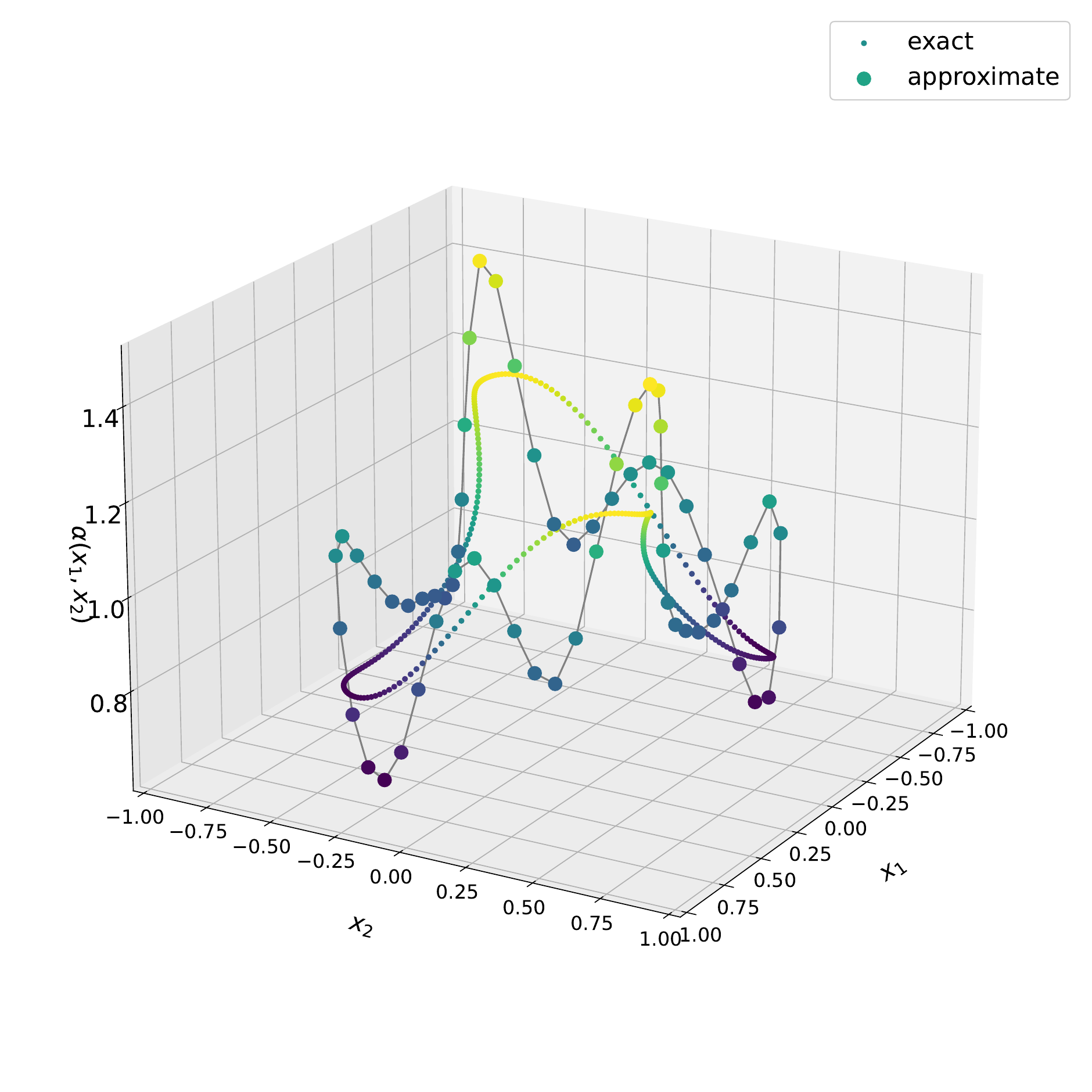}} \\ \hline
\textbf{B3} & \resizebox{0.325\linewidth}{!}{\includegraphics{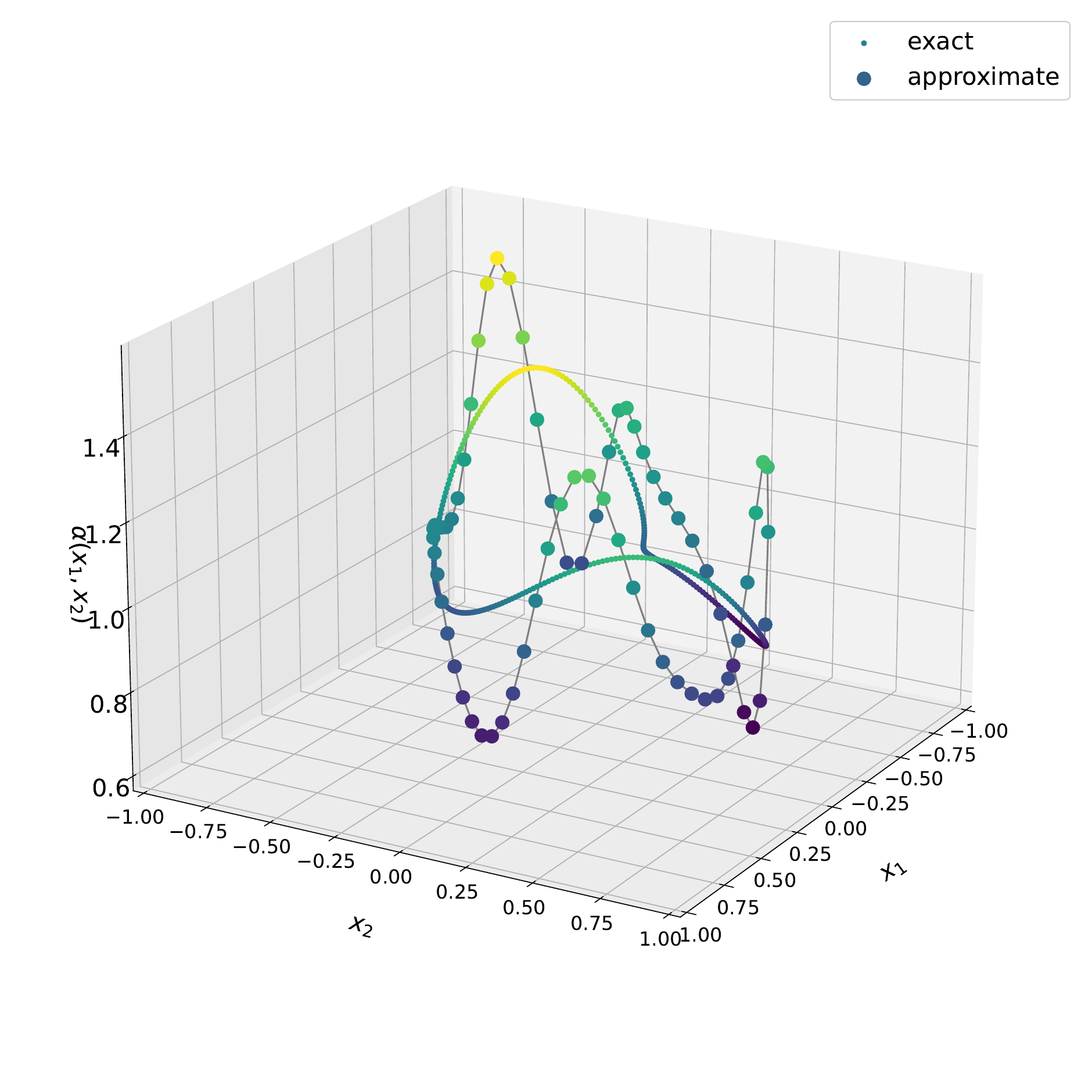}} &
	\resizebox{0.325\linewidth}{!}{\includegraphics{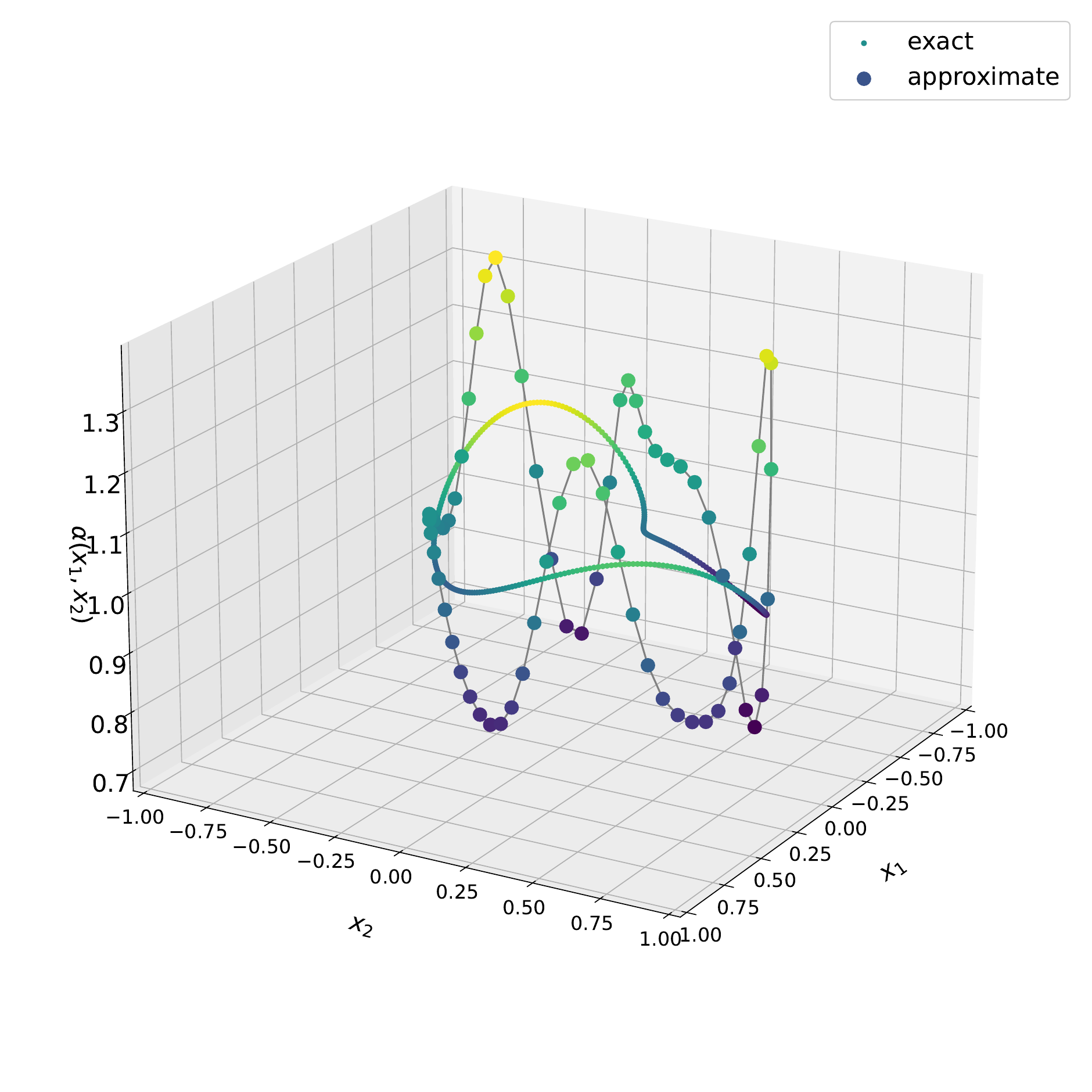}} &
	\resizebox{0.325\linewidth}{!}{\includegraphics{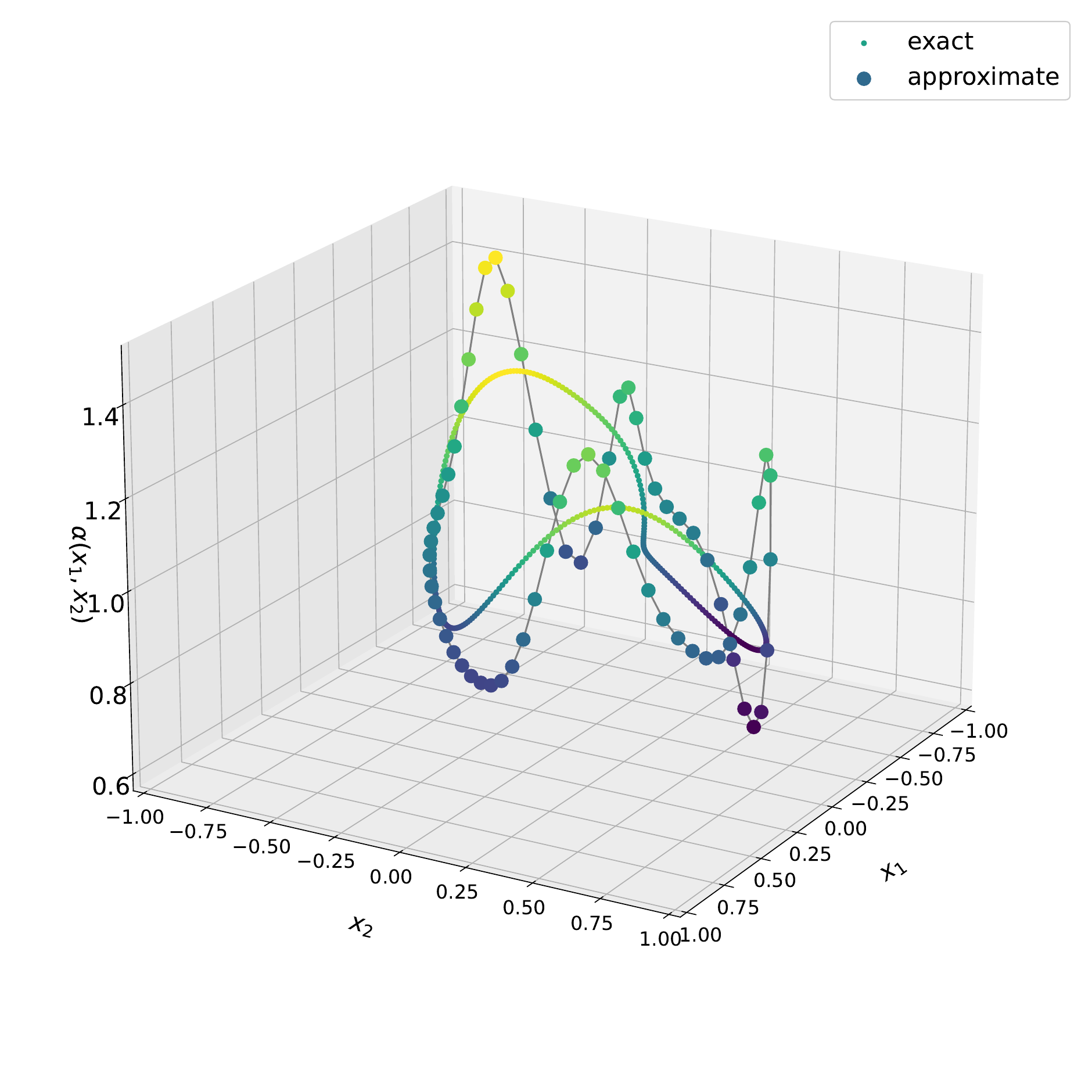}} \\ \hline
\end{tabular}
\end{adjustbox}
\caption{3D view of the reconstructed Robin coefficient, rotated by $30^\circ$, under $0.5\%$ noise}
\label{tab:noisy_robin_30}
\end{table}
\begin{table}[htp!]
\centering
\begin{adjustbox}{max width=\textwidth}
\begin{tabular}{|c|c|c|c|}
\hline
\textbf{Case} & \textbf{A1} & \textbf{A2} & \textbf{A3} \\ \hline
\textbf{B1} & \resizebox{0.325\linewidth}{!}{\includegraphics{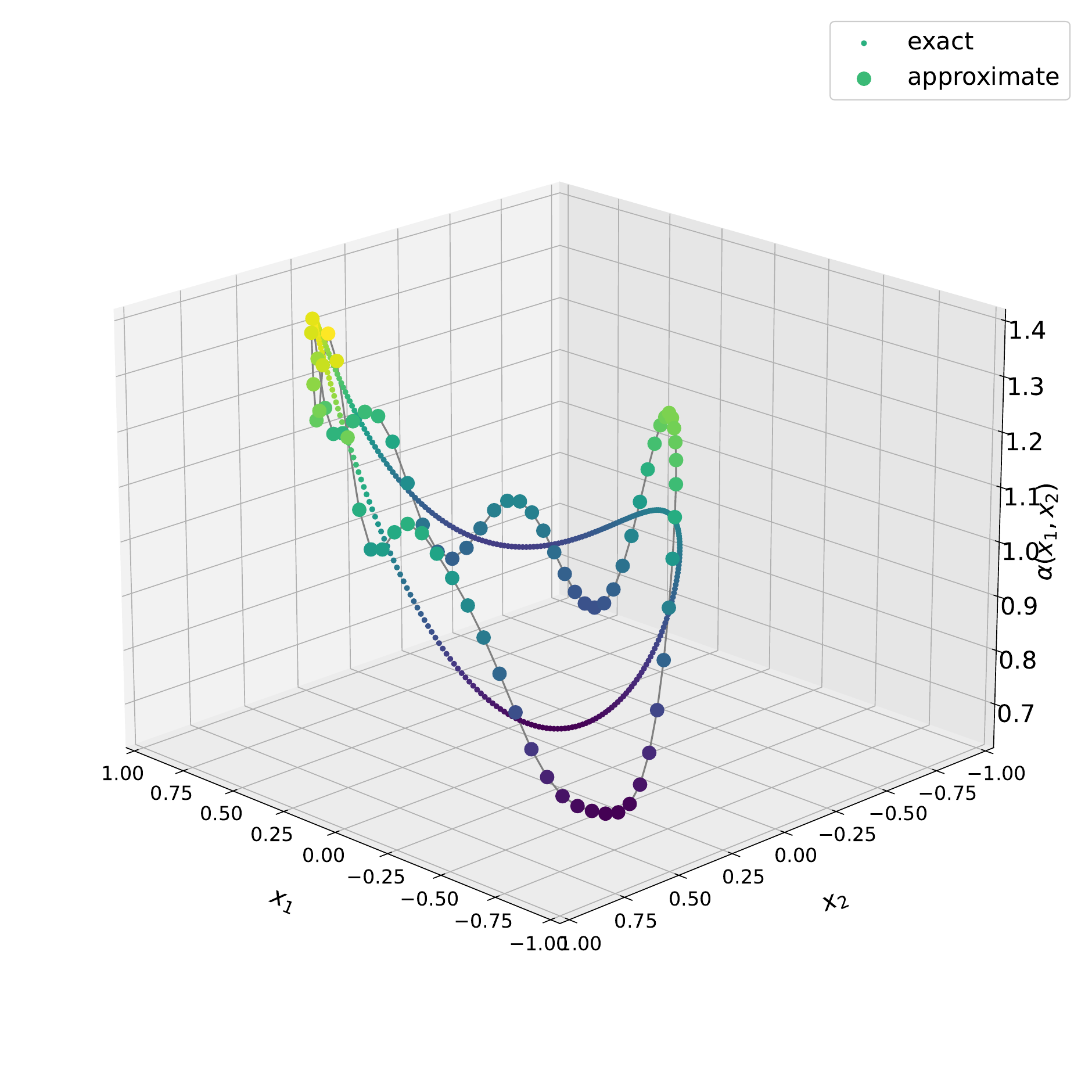}} &
	\resizebox{0.325\linewidth}{!}{\includegraphics{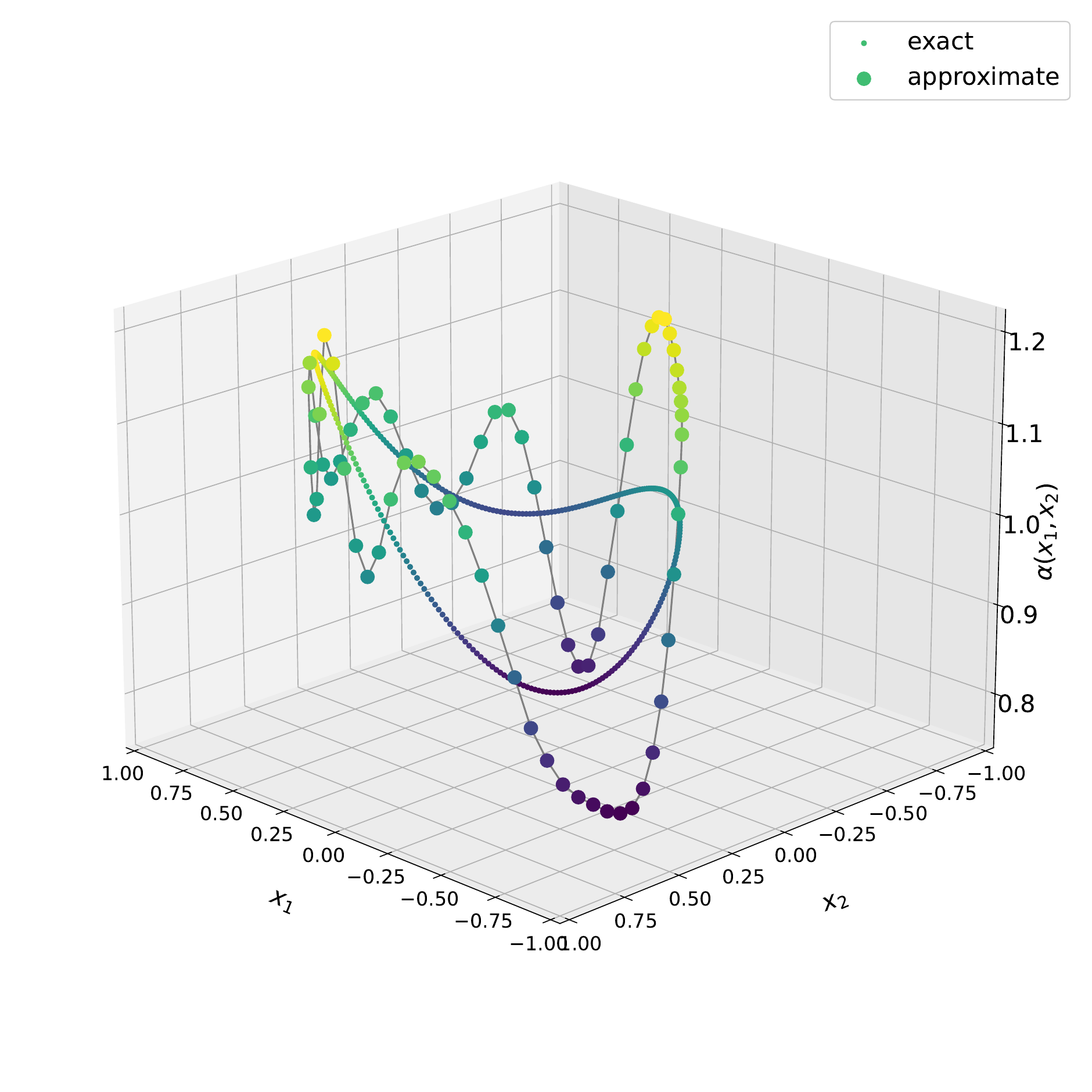}} &
	\resizebox{0.325\linewidth}{!}{\includegraphics{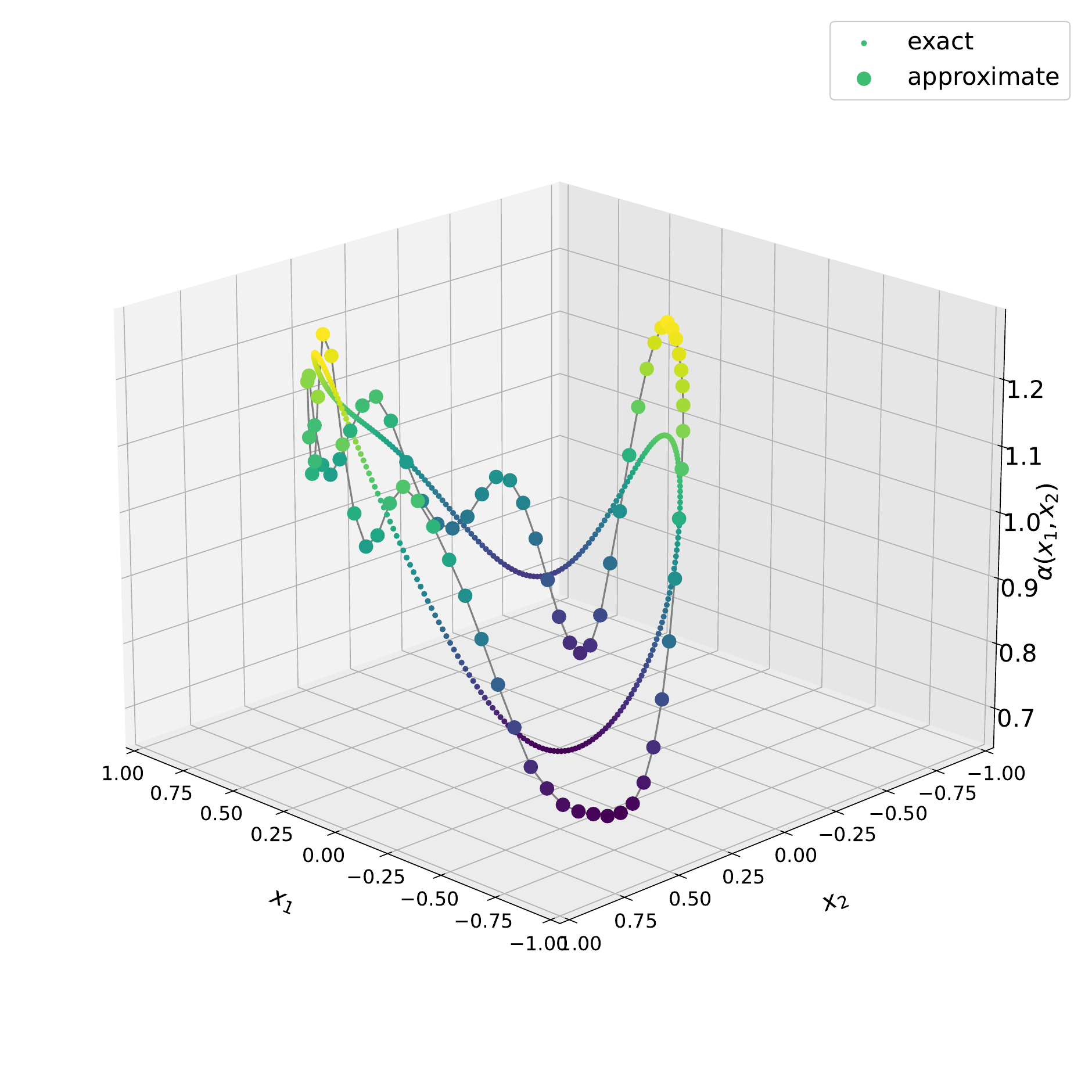}} \\ \hline
\textbf{B2} & \resizebox{0.325\linewidth}{!}{\includegraphics{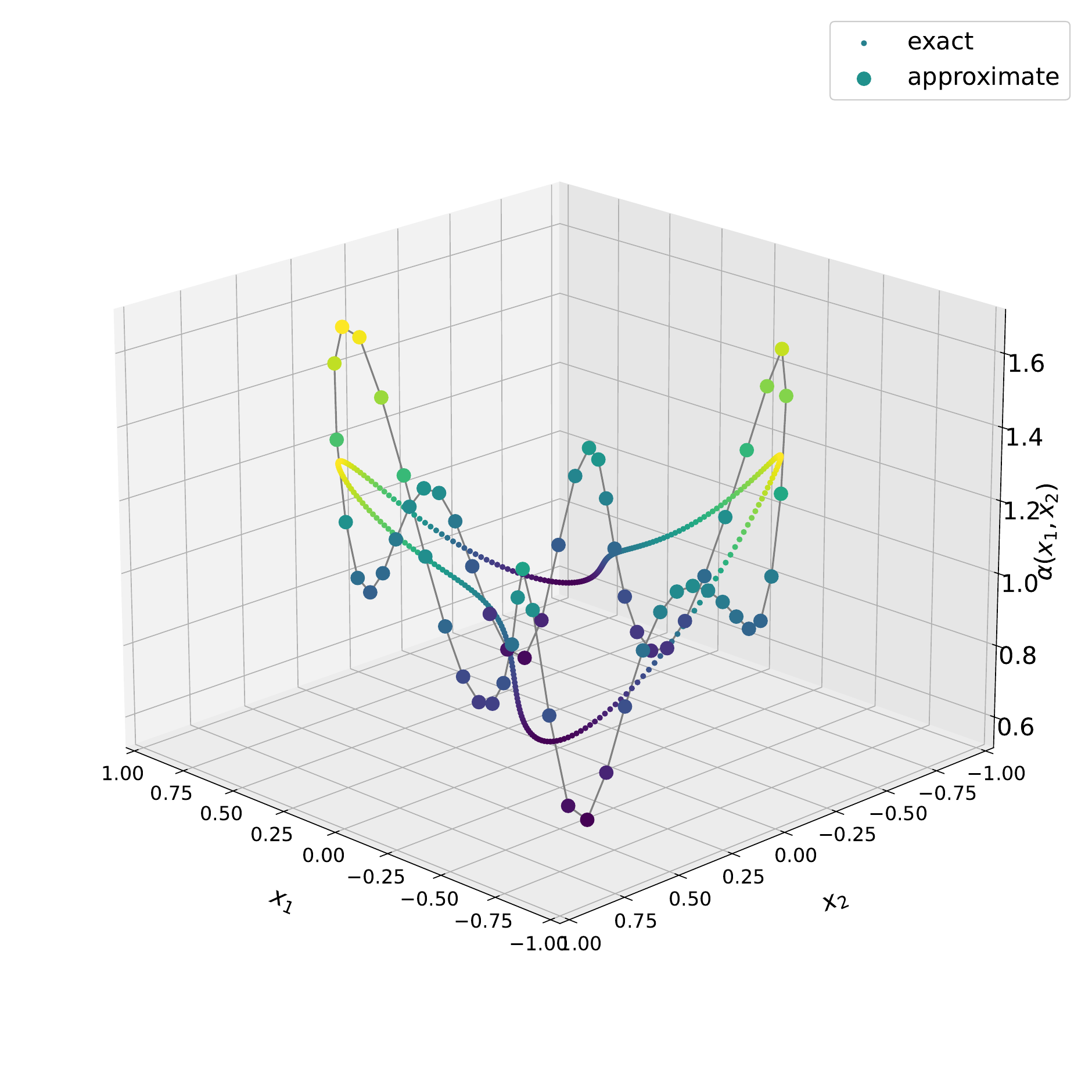}} &
	\resizebox{0.325\linewidth}{!}{\includegraphics{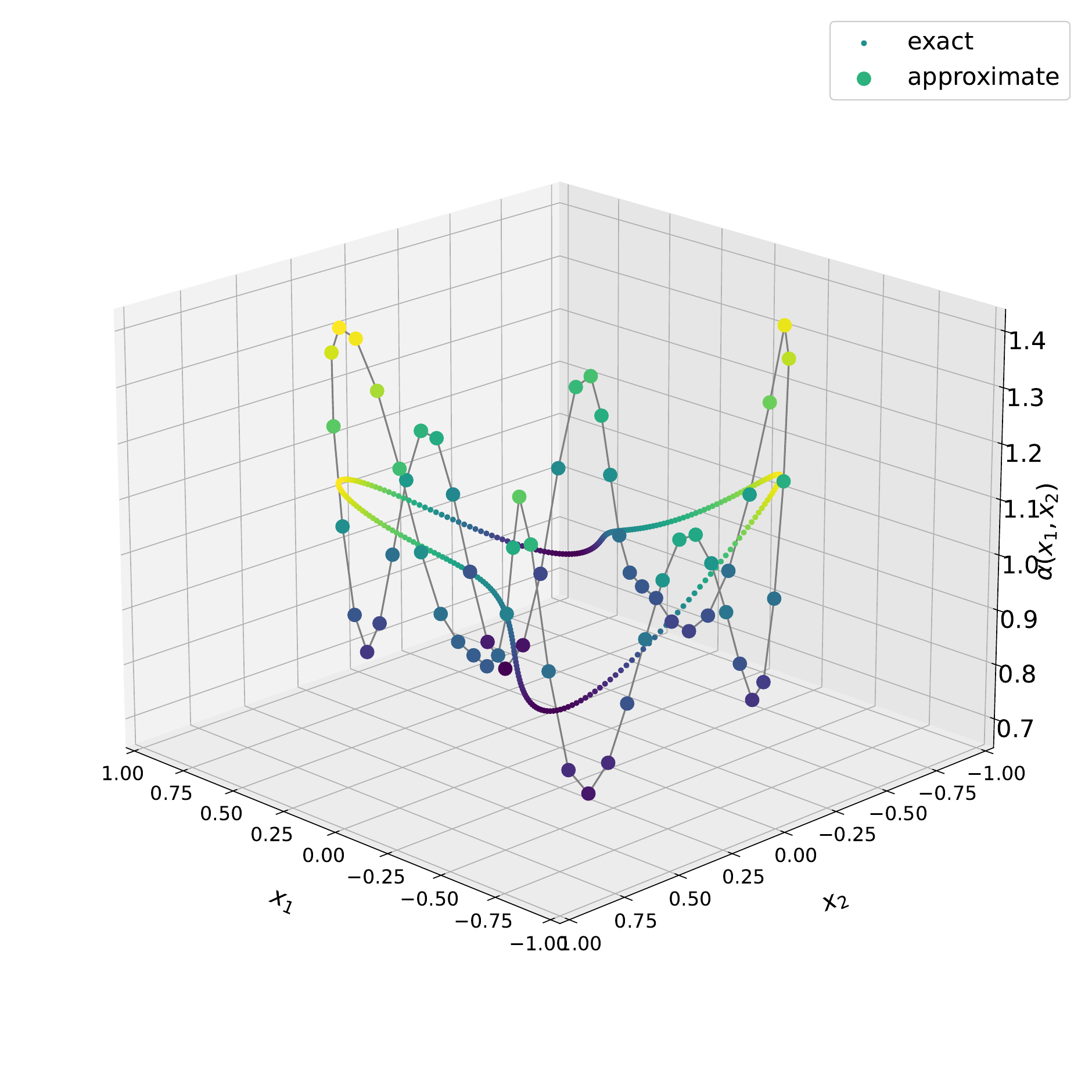}} &
	\resizebox{0.325\linewidth}{!}{\includegraphics{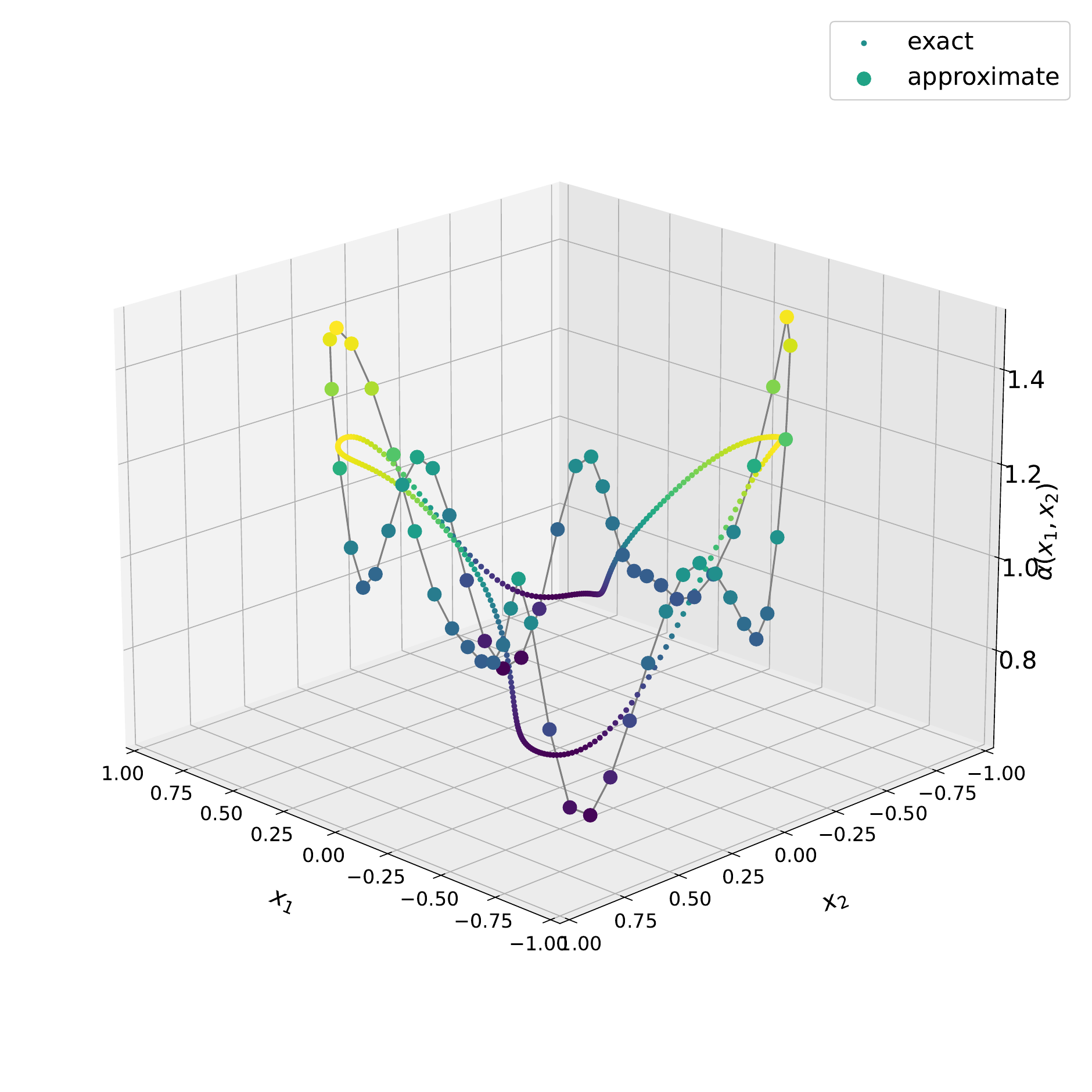}} \\ \hline
\textbf{B3} & \resizebox{0.325\linewidth}{!}{\includegraphics{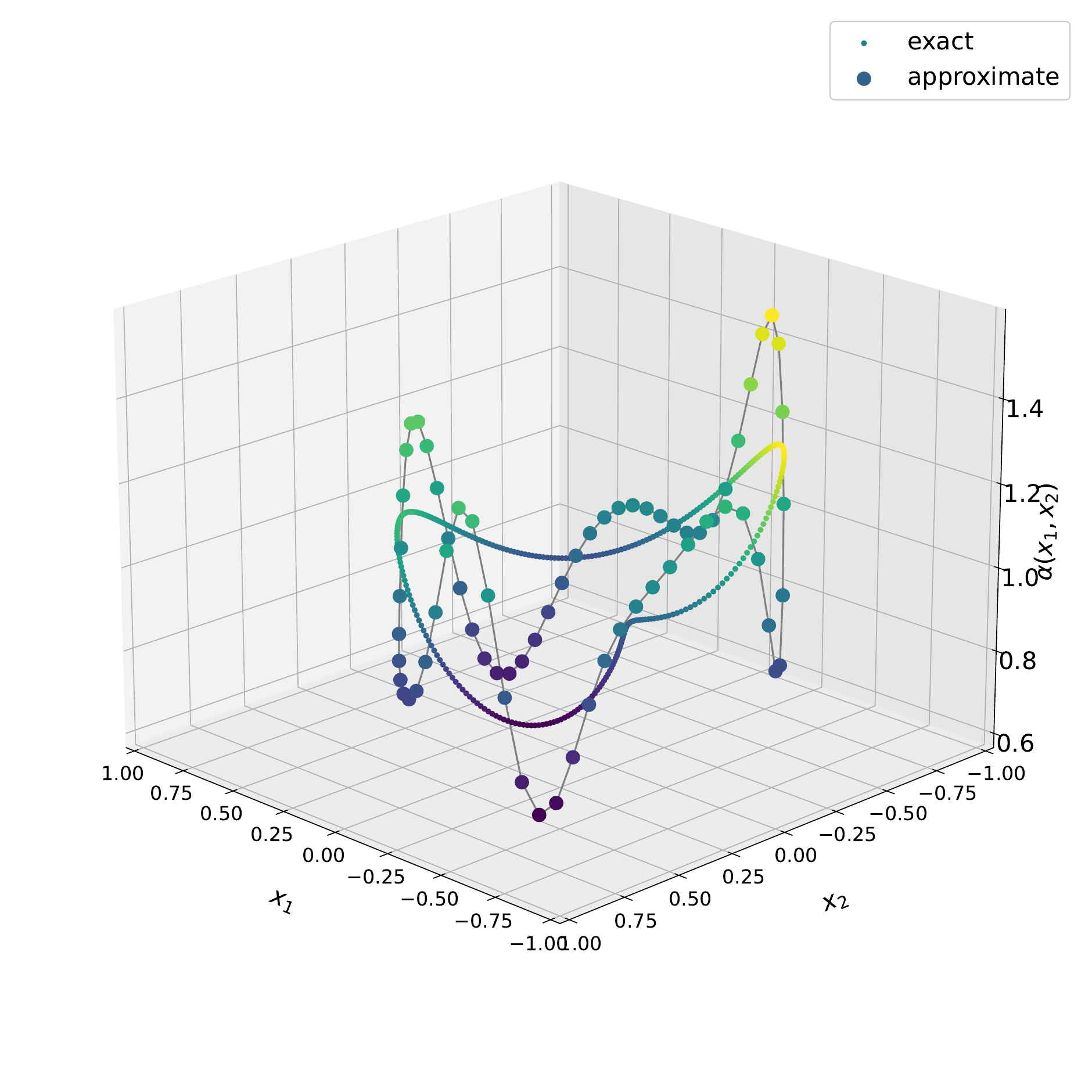}} &
	\resizebox{0.325\linewidth}{!}{\includegraphics{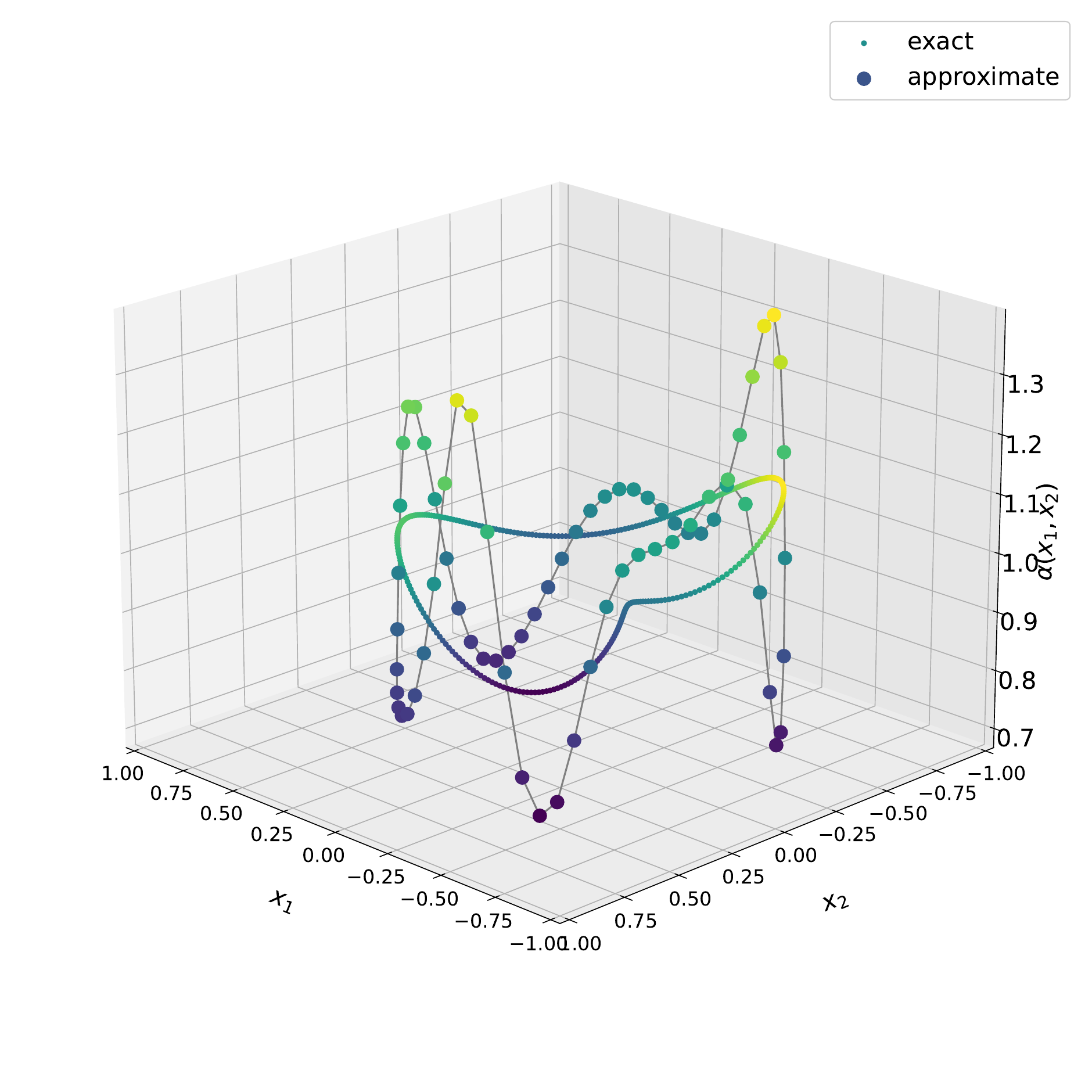}} &
	\resizebox{0.325\linewidth}{!}{\includegraphics{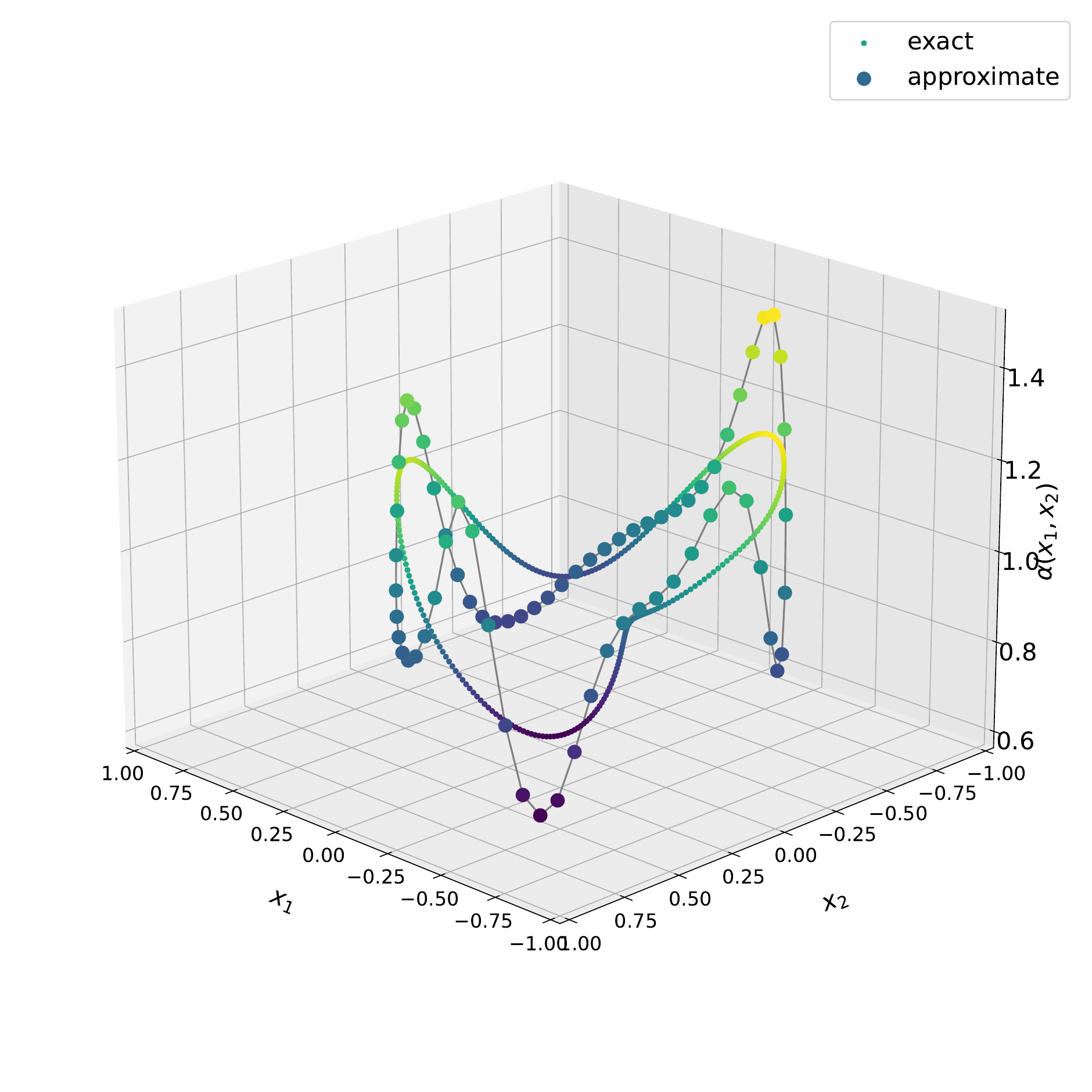}} \\ \hline
\end{tabular}
\end{adjustbox}
\caption{3D view of the reconstructed Robin coefficient, rotated by $135^\circ$, under $0.5\%$ noise}
\label{tab:noisy_robin_135}
\end{table}

\begin{figure}[htp!]
\centering
\resizebox{0.325\linewidth}{!}{\includegraphics{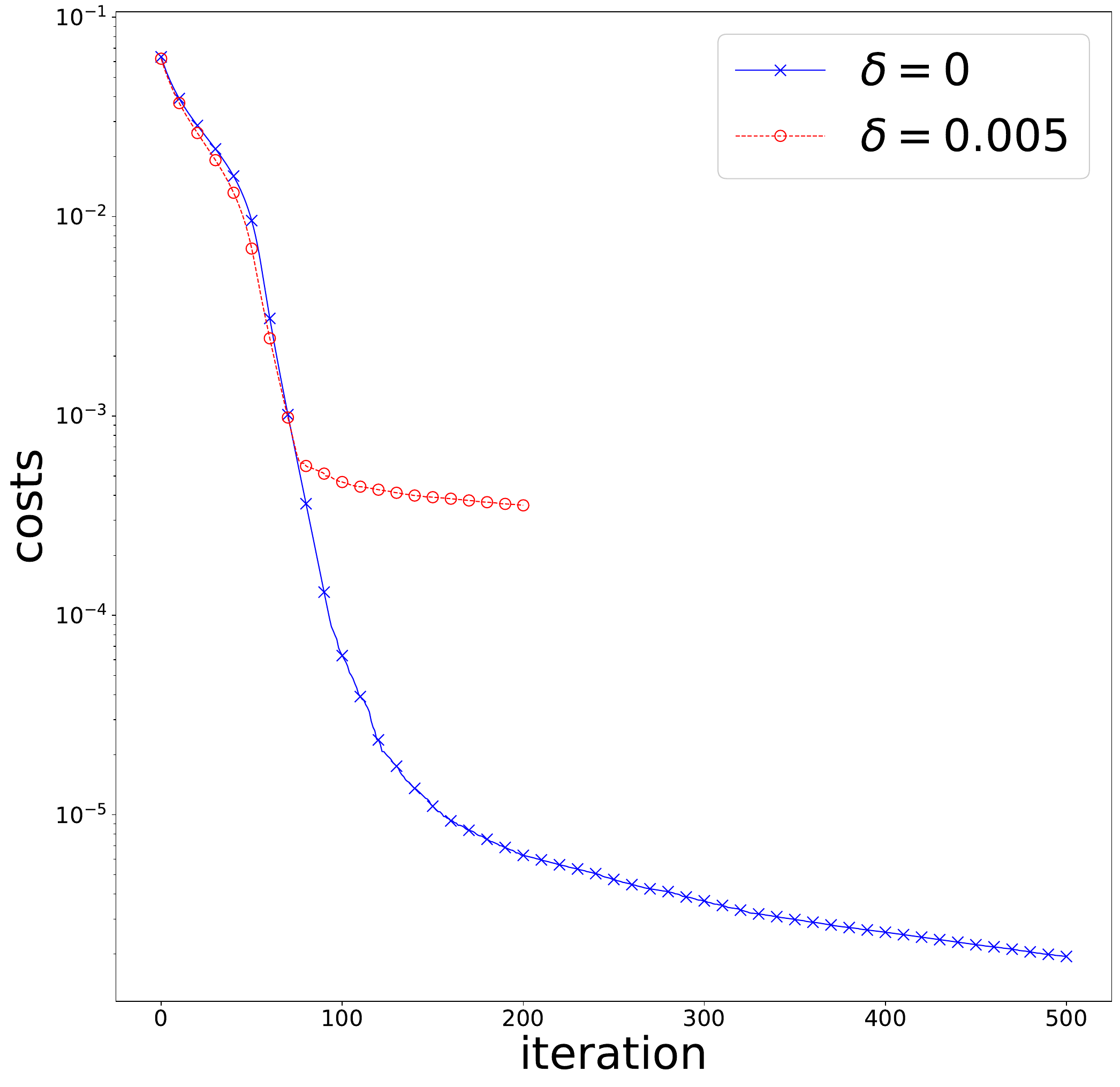}} 
\resizebox{0.325\linewidth}{!}{\includegraphics{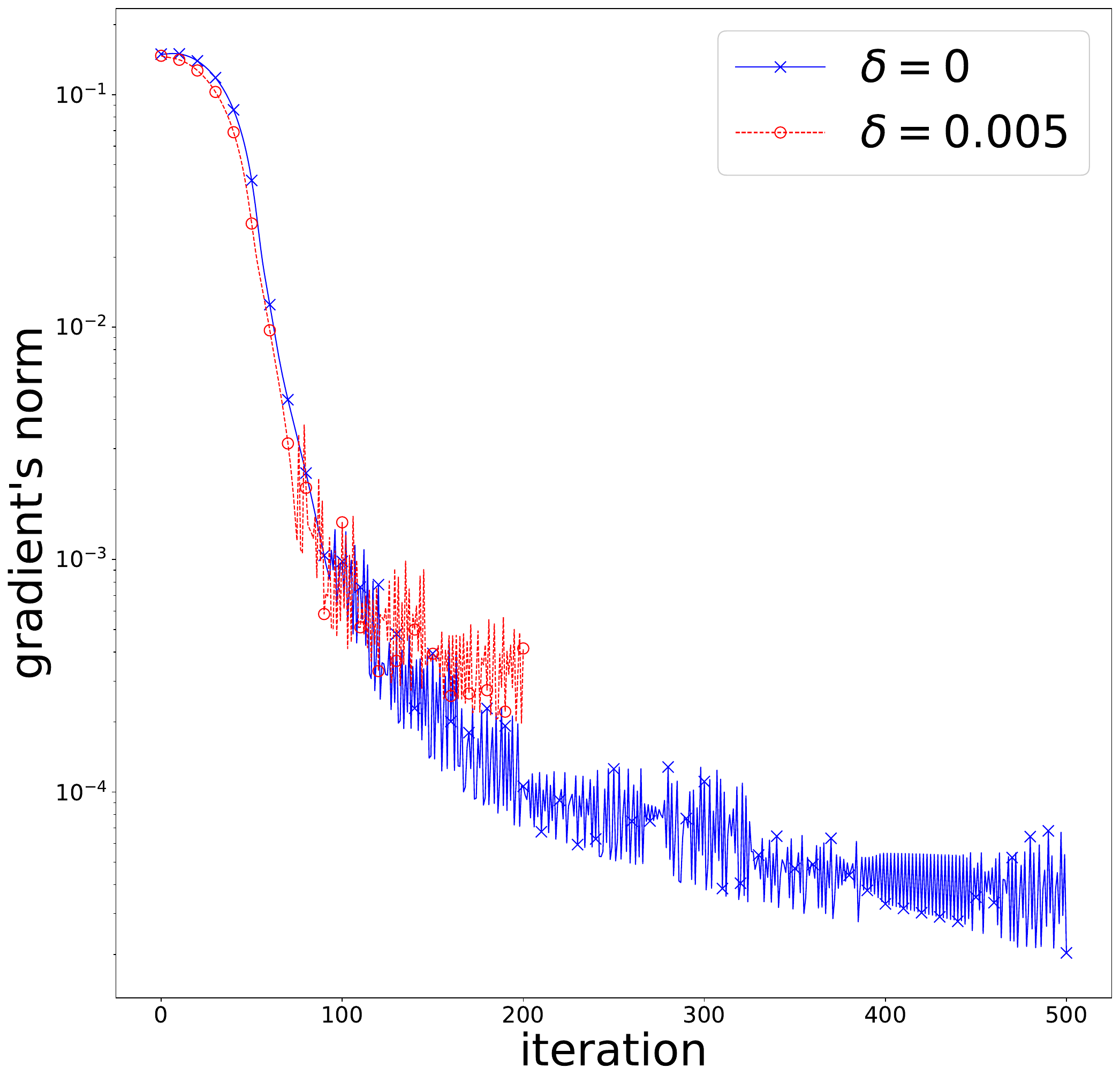}} 
\resizebox{0.325\linewidth}{!}{\includegraphics{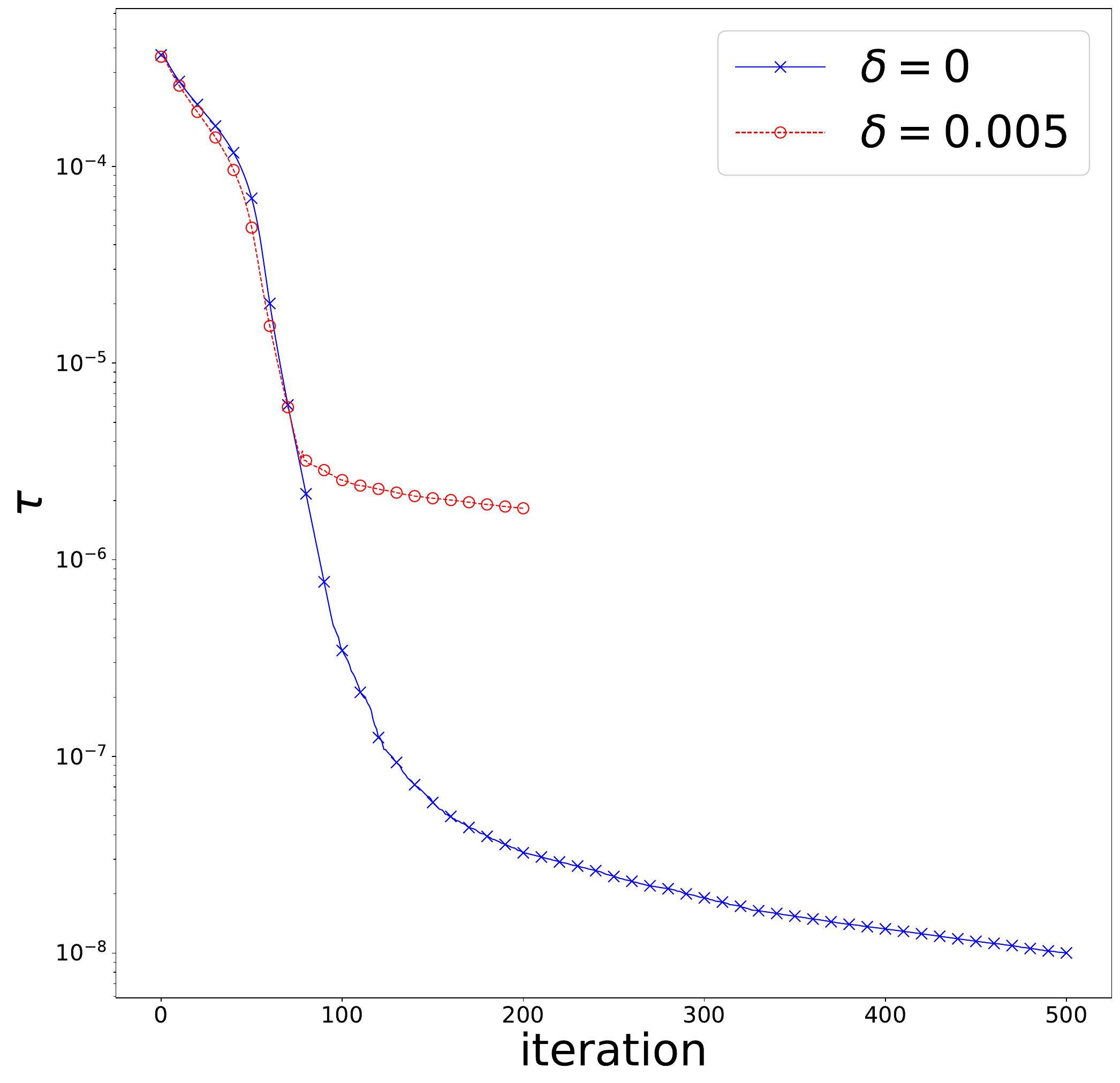}} 
\caption{Histories of cost values, gradient norms, and parameter $\tau$ for Case A3-B3}
\label{fig:summary_A3B3}
\end{figure}
\newpage
\section{Conclusion}\label{sec:conclusion} 
In this study, we presented a numerical method using the Kohn-Vogelius framework to simultaneously reconstruct an unknown inclusion and the Robin coefficient associated with an electrostatic problem. 
By utilizing gradient information derived from the energy-gap cost functional with respect to both the inclusion's shape and the Robin parameter, we formulated an iterative algorithm implemented via the finite element method. 
Numerical experiments demonstrated the method's feasibility and effectiveness for various inclusion geometries and Robin coefficients, provided exact measurements were available. 
For noisy data, however, incorporating multiple Cauchy pairs appears essential to achieve accurate reconstructions. 
Future work will address this aspect and be documented in a subsequent study.

\medskip
\section*{Acknowledgements} JFTR is supported by the JSPS Postdoctoral Fellowships for Research in Japan and partially by the JSPS Grant-in-Aid for Early-Career Scientists under Japan Grant Number JP23K13012 and the JST CREST Grant Number JPMJCR2014.

%
%
\bibliographystyle{alpha} 
\bibliography{main}   

@article{Harris2023,
	author = {I. Harris},
	journal = {Appl. Anal.},
	number = {5},
	pages = {1511--1529},
	title = {A direct method for reconstructing inclusions and boundary conditions from electrostatic data},
	volume = {102},
	year = {2023}}

@unpublished{MachidaNotsuRabago2024,
	author = {M. Machida and H. Notsu and J. F. T. Rabago},
	month = {November},
	note = {in preparation},
	title = {Reconstruction approach for optical tomography with fluorophores by single boundary measurement},
	year = {2024}}

@article{Meftahi2021,
	author = {H. Meftahi},
	journal = {SIAM J. Math. Anal.},
	number = {6},
	pages = {6326--6354},
	title = {Uniqueness, {L}ipschitz stability, and reconstruction for the inverse optical tomography problem},
	volume = {53},
	year = {2021}}

@article{Clason2012,
	author = {C. Clason},
	journal = {Inverse Probl.},
	pages = {104007},
	title = {$L^{\infty}$ fitting for inverse problems with uniform noise},
	volume = {28},
	year = {2012}}

@article{ClasonJin2012,
	author = {C. Clason and B. Jin},
	journal = {SIAM J. Imaging Sci.},
	pages = {505--536},
	title = {A semismooth {N}ewton method for nonlinear parameter identification problems with impulsive noise},
	volume = {5},
	year = {2012}}

@article{ClasonJinKunisch2010,
	author = {C. Clason and B. Jin and K. Kunisch},
	journal = {SIAM J. Sci. Comput.},
	pages = {1484--1505},
	title = {A duality-based splitting method for $l^{1}$-TV image restoration with automatic regularization parameter choice},
	volume = {32},
	year = {2010}}

@article{ClasonJinKunisch2010b,
	author = {C. Clason and B. Jin and K. Kunisch},
	journal = {SIAM J. Imaging Sci.},
	pages = {199--231},
	title = {A semismooth Newton method for $l^{1}$ data fitting with automatic choice of regularization parameters and noise calibration},
	volume = {3},
	year = {2010}}

@article{ItoJinTakeuchi2011,
	author = {K. Ito and B. Jin and T. Takeuchi},
	journal = {SIAM J. Sci. Comput.},
	pages = {1415--1438},
	title = {A regularization parameter for nonsmooth Tikhonov regularization},
	volume = {33},
	year = {2011}}

@book{Azegami2020,
	address = {Singapore},
	author = {H. Azegami},
	publisher = {Springer},
	series = {Springer Optimization and Its Applications},
	title = {Shape Optimization Problems},
	volume = {164},
	year = {2020}}

@article{LinFang2005,
	author = {F. Lin and W. Fang},
	journal = {Inverse Probl.},
	pages = {1757--1772},
	title = {A linear integral equation approach to the Robin inverse problem},
	volume = {21},
	year = {2005}}

@article{Jin2007,
	author = {B. Jin},
	journal = {Internat. J. Numer. Methods Engrg.},
	pages = {433--453},
	title = {Conjugate gradient method for the robin inverse problem associated with the Laplace equation},
	volume = {71},
	year = {2007}}

@article{Hauptmannetal2019,
	author = {A. Hauptmann and M. Ikehata and H. Itou and S. Siltanen},
	journal = {Inverse Probl.},
	number = {025004},
	pages = {24 pages},
	title = {Revealing cracks inside conductive bodies by electric surface measurements},
	volume = {35},
	year = {2019}}

@article{ChaabaneElhechmiJaoua2004,
	author = {S. Chaabane and C. Elhechmi and M. Jaoua},
	journal = {Math. Comput. Simulation},
	pages = {367--38},
	title = {A stable recovery method for the {R}obin inverse problem},
	volume = {66},
	year = {2004}}

@article{AfraitesRabago2024,
	author = {L. Afraites and J. F. T. Rabago},
	journal = {Comput. Appl. Math.},
	number = {270},
	pages = {37 pages},
	title = {Boundary shape reconstruction with Robin condition: existence result, stability analysis, and inversion via multiple measurements},
	volume = {43},
	year = {2024}}

@article{KaupSantosaVogelius1996,
	author = {P. G. Kaup and F. Santosa and M. Vogelius},
	journal = {Inverse Problems},
	pages = {279--293},
	title = {Method for imaging corrosion damage in thin plates from electrostatic data},
	volume = {12},
	year = {1996}}

@article{FangLu2004,
	author = {W. Fang and M. Lu},
	journal = {Internat. J. Numer. Methods Engrg.},
	pages = {1563--1585},
	title = {A fast collocation method for an inverse boundary value problem,},
	volume = {59},
	year = {2004}}

@article{AfraitesRabago2025,
	author = {L. Afraites and J. F. T. Rabago},
	journal = {Discrete Contin. Dyn. Syst. Ser. S},
	number = {1},
	pages = {43--76},
	title = {Shape optimization methods for detecting an unknown boundary with the {R}obin condition by a single measurement},
	volume = {18},
	year = {2025}}

@article{CakoniKressSchuft2010b,
	author = {F. Cakoni and R. Kress and C. Schuft},
	journal = {Methods Appl. Anal.},
	pages = {357--378},
	title = {Simultaneous reconstruction of shape and impedance in corrosion detection},
	volume = {17},
	year = {2010}}

@article{CakoniKressSchuft2010a,
	author = {F. Cakoni and R. Kress and C. Schuft},
	journal = {Inverse Problems},
	pages = {Art. 095012 24pp.},
	title = {Integral equations for inverse problems in corrosion detection from partial {C}auchy data},
	volume = {26},
	year = {2010}}

@article{Sincich2010,
	author = {E. Sincich},
	journal = {SIAM J. Math. Anal.},
	pages = {2922--2943},
	title = {Stability for the determination of unknown boundary and impedance with a {R}obin boundary condition},
	volume = {42},
	year = {2010}}

@article{PaganiPierotti2009,
	author = {C. D. Pagani and D. Peerotti},
	journal = {Inverse Problems},
	pages = {Art. 055007 12pp},
	title = {Identifiability problems of defects with {R}obin condition},
	volume = {25},
	year = {2009}}

@article{Fang2022,
	author = {W. Fang},
	journal = {J. Comput. Appl. Math.},
	pages = {Art. 114376 13 pp},
	title = {Simultaneous recovery of {R}obin boundary and coefficient for the {L}aplace equation by shape derivative},
	volume = {413},
	year = {2022}}

@article{FangZeng2009,
	author = {W. Fang and S. Zeng},
	journal = {J. Comput. Appl. Math.},
	pages = {573--580},
	title = {Numerical recovery of {R}obin boundary from boundary measurements for the {L}aplace equation},
	volume = {224},
	year = {2009}}

@article{FangLinMa2019,
	author = {W. Fang and F. Lin and Y. Ma},
	journal = {East Asian J. Appl. Math.},
	pages = {485--505},
	title = {Fast algorithms for boundary integral equations on elliptic domains and related inverse problems},
	volume = {9},
	year = {2019}}

@article{KohnVogelius1987,
	author = {R. Kohn and M. Vogelius},
	journal = {Commun. Pure Appl. Math.},
	number = {6},
	pages = {745--777},
	title = {Relaxation of a variational method for impedance computed tomography},
	volume = {40},
	year = {1987}}

@article{AfraitesDambrineKateb2007,
	author = {L. Afraites and M. Dambrine and D. Kateb},
	journal = {Numer. Funct. Anal. Optim.},
	number = {5--6},
	pages = {519--551},
	title = {Shape methods for the transmission problem with a single measurement},
	volume = {28},
	year = {2007}}

@article{AfraitesMasnaouiNachaoui2022,
	author = {L. Afraites and C. Masnaoui and M. Nachaoui},
	journal = {Discrete Contin. Dyn. Syst. Ser. S},
	number = {1},
	pages = {1--21},
	title = {Shape optimization method for an inverse geometric source problem and stability at critical shape},
	volume = {15},
	year = {2022}}

@article{CaubetDambrineKatebTimimoun2013,
	author = {F. Caubet and M. Dambrine and D. Kateb and C. Z. Timimoun.},
	journal = {Inverse Prob. Imaging},
	number = {1},
	pages = {123--157},
	title = {A {K}ohn-{V}ogelius formulation to detect an obstacle immersed in a fluid},
	volume = {7},
	year = {2013}}

@article{AfraitesDambrineKateb2008,
	author = {L. Afraites and M. Dambrine and D. Kateb},
	journal = {SIAM J. Control Optim.},
	number = {3},
	pages = {1556--1590},
	title = {On second order shape optimization methods for electrical impedance tomography},
	volume = {47},
	year = {2008}}

@book{ColtonKress2013,
	address = {New York},
	author = {R. Kress and D. Colton},
	edition = {3rd},
	editor = {S. S. Antman and P. Holmes and K. Sreenivasan},
	publisher = {Springer},
	series = {Applied Mathematical Sciences},
	title = {Inverse Acoustic and Electromagnetic Scattering Theory},
	volume = {93},
	year = {1998}}

@article{Bacchelli2009,
	author = {V. Bacchelli},
	journal = {Inverse Problems},
	pages = {Art. 015004 (4pp)},
	title = {Uniqueness for the determination of unknown boundary and impedance with the homogeneous {R}obin condition},
	volume = {25},
	year = {2009}}

@article{KaupSantosa1995,
	author = {P. G. Kaup and F. Santosa},
	journal = {J. Nondestr. Eval.},
	pages = {127--136},
	title = {Nondestructive evaluation of corrosion damage using electrostatic measurements},
	volume = {14},
	year = {1995}}

@article{FasinoInglese2007,
	author = {D. Fasino and G. Inglese},
	journal = {J. Comput. Appl. Math.},
	pages = {460--470},
	title = {Recovering nonlinear terms in an inverse boundary value problem for {L}aplace's equation: a stability estimate},
	volume = {198},
	year = {2007}}

@article{Rundell2008,
	author = {W. Rundell},
	journal = {Inverse Problems},
	pages = {1--22},
	title = {Recovering an obstacle and its impedance from {C}auchy data},
	volume = {24},
	year = {2008}}

@article{IngleseMariani2004,
	author = {G. Inglese and F. Mariani},
	journal = {Inverse Problems},
	pages = {1207--1215},
	title = {Corrosion detection in conducting boundaries},
	volume = {20},
	year = {2004}}

@article{ChaabaneJaoua1999,
	author = {S. Chaabane and M. Jaoua},
	journal = {Inverse Problems},
	pages = {1425--1438},
	title = {Identification of {R}obin coefficients by means of boundary measurements},
	volume = {15},
	year = {1999}}

@article{CakoniKress2007,
	author = {F. Cakoni and R. Kress},
	journal = {Inverse Prob. Imaging},
	pages = {229--245},
	title = {Integral equations for inverse problems in corrosion detection from partial Cauchy data},
	volume = {1},
	year = {2007}}

@article{Doganetal2007,
	author = {G. Do\v{g}an and P. Morin and R.H. Nochetto and M. Verani},
	journal = {Comput. Methods Appl. Mech. Engrg.},
	pages = {3898--3914},
	title = {Discrete gradient flows for shape optimization and applications},
	volume = {196},
	year = {2007}}

@article{RabagoAzegami2020,
	author = {J. F. T. Rabago and H. Azegami},
	journal = {Comput. Optim. Appl.},
	number = {1},
	pages = {251--305},
	title = {A second-order shape optimization algorithm for solving the exterior {B}ernoulli free boundary problem using a new boundary cost functional},
	volume = {77},
	year = {2020}}

@article{RabagoAzegami2019b,
	author = {J. F. T. Rabago and H. Azegami},
	journal = {Evol. Equ. Control Theory},
	number = {4},
	pages = {785-824},
	title = {A new energy-gap cost functional cost functional approach for the exterior {B}ernoulli free boundary problem},
	volume = {8},
	year = {2019}}

@article{RabagoAzegami2018,
	author = {J. F. T. Rabago and H. Azegami},
	journal = {Japan J. Indust. Appl. Math.},
	number = {1},
	pages = {131-176},
	title = {Shape optimization approach to defect-shape identification with convective boundary condition via partial boundary measurement},
	volume = {31},
	year = {2018}}

@book{Neuberger1997,
	address = {Berlin},
	author = {J. W. Neuberger},
	publisher = {Springer-Verlag},
	title = {Sobolev Gradients and Differential Equations},
	year = {1997}}

@book{HenrotPierre2018,
	address = {Z\"{u}rich},
	author = {A. Henrot and M. Pierre},
	publisher = {European Mathematical Society},
	series = {Tracts in Mathematics},
	title = {Shape Variation and Optimization: A Geometrical Analysis},
	volume = {28},
	year = {2018}}

@article{Hecht2012,
	author = {F. Hecht},
	journal = {J. Numer. Math.},
	pages = {251-265},
	title = {New development in {F}ree{F}em++},
	volume = {20},
	year = {2012}}

@article{EpplerHarbrecht2012a,
	author = {K. Eppler and H. Harbrecht},
	journal = {Comput. Optim. App.},
	pages = {69-85},
	title = {On a {K}ohn-{V}ogelius like formulation of free boundary problems},
	volume = {52},
	year = {2012}}

@article{EpplerHarbrecht2005,
	author = {K. Eppler and H. Harbrecht},
	journal = {Control Cybern.},
	pages = {203-225},
	title = {A regularized {N}ewton method in electrical impedance tomography using shape Hessian information},
	volume = {34},
	year = {2005}}

@article{BenAbdaetal2013,
	author = {A. Ben Abda and F. Bouchon and G. H. Peichl and M. Sayeh and R. Touzani},
	journal = {J. Eng. Math.},
	pages = {157-176},
	title = {A {D}irichlet-{N}eumann cost functional approach for the {B}ernoulli problem},
	volume = {81},
	year = {2013}}

@article{BacaniPeichl2013,
	author = {J. B. Bacani and G. H. Peichl},
	journal = {Abstr. Appl. Anal.},
	pages = {19 pp. Article ID 384320},
	title = {On the first-order shape derivative of the {K}ohn-{V}ogelius cost functional of the {B}ernoulli problem},
	volume = {2013},
	year = {2013}}

@article{Azegami1994,
	author = {H. Azegami},
	journal = {Trans. Jpn. Soc. Mech. Eng., Ser. A.},
	pages = {1479-1486},
	title = {A solution to domain optimization problems},
	volume = {60},
	year = {1994}}
\end{document}